\documentclass[reqno,11pt]{amsart}
\usepackage{a4}
\usepackage{latexsym}
\usepackage{amssymb,latexsym,amscd,amsmath,amsthm}
\usepackage{hyperref}
\usepackage{tikz}
\newcommand\A{\ensuremath{\mathbb A}}
\newcommand\C{\ensuremath{\mathbb C}}
\renewcommand\H{\ensuremath{\mathbb H}}
\newcommand\Q{\ensuremath{\mathbb Q}}
\newcommand\R{\ensuremath{\mathbb R}}
\newcommand\Z{\ensuremath{\mathbb Z}}

\newcommand\id{\operatorname{id}}
\let\eps=\varepsilon
\let\thet=\vartheta
\let\phy=\varphi
\newcommand\im{\operatorname{im}}
\newcommand\rk{\operatorname{rk}}

\newcommand\del{\partial}
\newcommand\norm[1]{\left\|#1\right\|}
\newcommand\abs[1]{\left|#1\right|}
\newcommand\<{\langle}
\renewcommand\>{\rangle}
\newcommand\Hom{{\operatorname{Hom}}}
\newcommand\End{{\operatorname{End}}}

\newcommand\Vol{{\operatorname{vol}}}

\newcommand\ch{\operatorname{ch}\nolimits}

\newcommand\sign{\operatorname{sign}}
\newcommand\trace{\operatorname{tr}}

\newcommand\spec{\operatorname{spec}}
\newcommand\Pf{\operatorname{Pf}}
\newcommand\even{{\mathrm{even}}}

\newcommand\Span{\operatorname{span}\nolimits}
\newcommand\punkt{\mathord{\,\cdot\,}}

\renewcommand\Im{\operatorname{Im}}
\let\ek=\mu
\newcommand\lk{\operatorname{lk}}
\newcommand\dotcup{\mathbin{\dot\cup}}
\newcommand\Res{\operatorname{Res}}

\newcommand{\SO}{\mathrm{SO}}
\newcommand{\SU}{\mathrm{SU}}
\newcommand{\GL}{\mathrm{GL}}
\newcommand{\Spin}{\mathrm{Spin}}
\newcommand{\Pin}{\mathrm{Pin}}
\newcommand{\Sp}{\mathrm{Sp}}
\newcommand{\sP}{\mathfrak{sp}}
\newcommand{\U}{\mathrm{U}}
\renewcommand{\O}{\mathrm{O}}

\newcommand{\Ad}{\operatorname{Ad}}
\newcommand{\frg}{\mathfrak g}

\numberwithin{equation}{section}
\numberwithin[\alph]{subsection}{section}
\swapnumbers
\theoremstyle{plain}
\newtheorem{Lemma}{Lemma}
\numberwithin{Lemma}{section}
\newtheorem{Proposition}[Lemma]{Proposition}
\newtheorem{Corollary}[Lemma]{Corollary}

\newtheorem{Theorem}[Lemma]{Theorem}

\theoremstyle{definition}
\newtheorem{Definition}[Lemma]{Definition}

\theoremstyle{remark}

\newtheorem{Remark}[Lemma]{Remark}
\newtheorem{Example}[Lemma]{Example}

\begin{document}

\title[Adiabatic limits of Seifert fibrations]%
{Adiabatic limits of Seifert fibrations,\\Dedekind sums,
and the diffeomorphism type\\of certain 7-manifolds}

\author{Sebastian Goette}
\address{Mathematisches Institut\\
Universit\"at Freiburg\\
Eckerstr.~1\\
79104 Freiburg\\
Germany}
\email{sebastian.goette@math.uni-freiburg.de}
\thanks{Supported in part by DFG special programme
``Global Differential Geometry''}

\begin{abstract}
  We extend the adiabatic limit formula for $\eta$-invariants
  by Bismut-Cheeger and Dai to Seifert fibrations.
  Our formula contains a new contribution from the singular fibres
  that takes the form of a generalised Dedekind sum.

  As an application,
  we compute the Eells-Kuiper and $t$-invariants
  of certain cohomogeneity one manifolds that were studied by Dearricott, Grove,
  Verdiani, Wilking, and Ziller.
  In particular, we determine the diffeomorphism type
  of a new manifold of positive sectional curvature.
\end{abstract}

\maketitle

Manifolds of positive sectional curvature are a rare phenomenon,
and the differential topological conditions for the existence
of positive sectional curvature metrics are not yet fully understood.
For this reason,
one is still interested in finding new examples of positive sectional
curvature metrics.
Most known examples are quotients or biquotients of compact Lie groups.
Cohomogeneity one manifolds constitute another potential source of examples.
By work of Grove, Wilking and Ziller~\cite{GWZ},
there are only two families~$(P_k)$, $(Q_k)$
of seven-dimensional cohomogeneity one manifolds,
which possibly allow metrics of positive sectional curvature
and contain new examples.
The space~$R$ mentioned there does not admit a positive sectional curvature
metric by~\cite{VW}.
Grove, Verdiani and Ziller have succeeded in~\cite{GVZ}
to construct a positive sectional curvature metric on~$P_2$,
the first nontrivial member of the family~$(P_k)$.
This manifold is homeomorphic to the unit tangent bundle~$T^1S^4$
of the four-dimensional sphere.
In this paper, we will specify among other things an exotic sphere~$\Sigma$
such that~$P_2$ is diffeomorphic to the connected
sum of~$T^1S^4$ and~$\Sigma$.

The manifolds~$P_k$ are highly connected
with a finite cyclic cohomology group~$H^4(P_k)\cong\pi_3(P_k)\cong\Z/k\Z$.
By Crowley's work~\cite{C},
it suffices to compute the Eells-Kuiper invariant~$\mu(P_k)$
and a certain quadratic form~$q$ on~$H^4(P_k)$
to determine their diffeomorphism types.
These two invariants are classically defined on oriented spin manifolds~$N$
bounding~$P_k$, but it is not clear how to construct such a manifold~$N$
directly.
On the other hands, by results of Donnelly~\cite{Do1}, Kreck and Stolz~\cite{KS}
and Crowley and the author~\cite{CG},
both invariants can equivalently be expressed as linear combinations
of $\eta$-invariants of certain Dirac operators
and Cheeger-Chern-Simons correction terms on~$P_k$ itself.
Having computed these invariants,
one can write the spaces~$P_k$ as connected sums of exotic spheres
and $S^3$-bundles over~$S^4$ using the computations for these bundles
in~\cite{CE}.
In order to determine the necessary $\eta$-invariants,
we write the spaces~$P_k$ as Seifert fibrations as indicated in~\cite{GWZ}.
That is, the spaces~$P_k$ are fibered by compact manifolds over
some base orbifold~$B$.

The process of blowing up the base space of a fibration~$M\to B$
by a factor~$\eps^{-1}$ is called the adiabatic limit.
It has been shown by Bismut, Cheeger~\cite{BC1} and Dai~\cite{Dai}
that the $\eta$-invariants of a family of compatible
Dirac operators~$D_{M,\eps}$ converge in the adiabatic limit~$\eps\to 0$,
if the kernels of the associated vertical Dirac operators~$D_X$
form a vector bundle~$H\to B$.
This result can be generalised to Seifert fibrations~$M\to B$.
Thus, we consider adiabatic families of Dirac operators~$(D_{M,\eps})_\eps$
as in Definition~\ref{A3.D1}.
In particular,
we assume that~$H=\ker(D_X)$ is a vector orbibundle on~$B$.
Let~$\Lambda B$ be the inertia orbifold of~$B$
and let~$\hat A_{\Lambda B}(TB,\nabla^{TB})\in\Omega^\bullet(\Lambda B;\widetilde{\Lambda B}\otimes o(\Lambda B))$
denote the orbifold $\hat A$-form as in Kawasaki's
index theorem~\cite{Kawa}, see section~\ref{A2}.
Let~$\A$ denote Bismut's Levi-Civita superconnection
associated with~$D_{M,\eps}$.
In Definition~\ref{A3.D2}, we construct orbifold
$\eta$-forms~$\eta_{\Lambda B}(\A)\in\Omega^\bullet(\Lambda B;\widetilde{\Lambda B})$ as in~\cite{BC1} and~\cite{Go}.
In Definition~\ref{A3.D3}, we define the effective horizontal
operator~$D_B^{\mathrm{eff}}$
of the family~$D_{M,\eps}$, which acts on sections of~$H\to B$.
Let~$(\lambda_\nu(\eps))_\nu$ denote the finite family
of very small eigenvalues of~$D_{M,\eps}$, see section~\ref{A3}.
In~\cite{Rochon}, Rochon proved a special case of the following theorem
where~$B$ is a very good orbifold and the fibrewise operator is invertible.

\begin{Theorem}[cf.\ \cite{BC1}, \cite{Dai}, \cite{Rochon}]\label{Thm1}
  Let~$p\colon M\to B$ be a Seifert fibration
  and~$(D_{M,\eps})_\eps$ an adiabatic family of Dirac operators over~$M$
  as in Definition~\ref{A3.D1}.
  For~$\eps_0>0$ sufficiently small,
  we have
  \begin{equation*}
    \lim_{\eps\to 0}\eta(D_{M,\eps})
    =\int_{\Lambda B}\hat A_{\Lambda B}\bigl(TB,\nabla^{TB}\bigr)
		\,2\eta_{\Lambda B}(\A)
    +\eta\bigl(D_B^{\mathrm{eff}}\bigr)
    +\sum_{\nu}\sign(\lambda_\nu(\eps_0))\;.
  \end{equation*}
\end{Theorem}

With this result,
one can compute the Eells-Kuiper invariant and the quadratic form~$q$
and hence determine the diffeomorphism type of each space~$P_k$.

\begin{Theorem}\label{Thm2}
  The Eells-Kuiper invariant of~$P_k$ is given by
  \begin{gather*}
    \ek(P_k)=-\frac{4k^3-7k+3}{2^5\cdot 3\cdot 7}\quad\in\quad\Q
	/\Z\;.
    \tag{1}\label{Thm2.1}
  \end{gather*}
  The quadratic form~$q$ on~$H^4(P_k)\cong\Z/k\Z$ is given by
  \begin{gather*}
    q(\ell)=\frac{\ell(\ell-k)}{2k}\quad\in\quad\Q/\Z\;.
    \tag{2}\label{Thm2.2}
  \end{gather*}
\end{Theorem}

By comparing these values with the corresponding values for $S^3$-bundles
over~$S^4$ in~\cite{CE} and~\cite{CG},
one can construct manifolds that are diffeomorphic to~$P_k$.

\begin{Theorem}\label{Thm3}
  Let~$E_{k,k}\to S^4$ denote the principal $S^3$-bundle 
  with Euler class~$k\in H^4(S^4)\cong\Z$,
  and let~$\Sigma_7$ denote the exotic seven sphere
  with~$\ek(\Sigma_7)=\frac1{28}$.
  Then there exists an orientation preserving diffeomorphism
	$$P_k\cong E_{k,k}\mathbin{\#}\Sigma_7^{\#\frac{k-k^3}6}\;.$$
  In particular, $P_k$ and~$E_{k,k}$ are homeomorphic.
\end{Theorem}

More generally,
let~$E_{p,n}$ denote the unit sphere bundle of a fourdimensional
real spin vector bundle over~$S^4$ with Euler class~$n$
and half Pontrijagin class~$p\in H^4(S^4)\cong\Z$.

\begin{Corollary}\label{Cor1}
  For the space~$P_k$,
  there exists an $S^3$-bundle~$E_{ak,k}\to S^4$ that is
  \begin{enumerate}
  \item\label{Cor1.1}
    oriented diffeomorphic
    if and only if~$k$ is odd or~$8|k$,
    with
	$$a^2k\equiv\frac{7k-4k^3}3\quad\mod 224\Z\;;$$
  \item\label{Cor1.2}
    orientation reversing diffeomorphic
    if and only if
    \begin{enumerate}
    \item\label{Cor1.2a}
      $k$ is not divisible by~$7$,
    \item\label{Cor1.2b}
      $k\equiv 1$, mod~$4$ or~$k\equiv 2$, $10$ mod~$32$, and
    \item\label{Cor1.2c}
      $-1$ is a quadratic remainder mod~$k$,
    \end{enumerate}
    with
	$$a^2k\equiv2-\frac{7k-4k^3}3\quad\mod 224\Z\;.$$
  \end{enumerate}
\end{Corollary}
Some of the~$P_k$ are discussed in greater detail
in Example~\ref{C4.X1}.

The article is organised as follows.
In section~\ref{A},
we introduce Seifert fibrations
and define all the ingredients of Theorem~\ref{Thm1}.
Its proof is given in section~\ref{B}.
In section~\ref{C},
we introduce the family~$(P_k)$ as a subfamily of the larger
family~$(M_{(p_-,q_-),(p_+,q_+)})$ that was also considered in~\cite{GWZ}.
The quadratic forms~$q_{M_{(p_-,q_-),(p_+,q_+)}}$ for some of those manifolds
are given in Theorem~\ref{C2.T2},
and their Eells-Kuiper invariants in Theorem~\ref{C2.T1}.
For the spaces~$P_k$,
we obtain the simplified formulas of Theorem~\ref{Thm2}
and prove Theorem~\ref{Thm3} and Corollary~\ref{Cor1}.
Finally,
section~\ref{M} contains the computations of $\eta$-invariants
needed to prove Theorems~\ref{C2.T2} and~\ref{C2.T1}.

\subsubsection*{Acknowledgements}
The author wants to thank Wolfgang Ziller for bringing the manifolds~$M_{(p_-,q_-),(p_+,q_+)}$ to his attention, and to Owen Dearricott and Stefan Teufel for their interest in this paper. Thanks also to Diarmuid Crowley and Nitu Kitchloo for their help with the differential-topological classification problem, and to Frederic Rochon for explaining his adiabatic limit theorem.
The author is also grateful to Don Zagier who helped to simplify the computation of the Dedekind sums.
Last but not least, the author is indepted to Ralf Braun, Svenja Dahms, Nadja Fischer, Anja Fuchshuber and Natalie Peternell for their comments on a preliminary version of the proof of the adiabatic limit theorem, which led to considerable improvements.

\section{An Adiabatic Limit Theorem for \texorpdfstring{$\eta$}{eta}-Invariants
of Seibert Fibrations}\label{A}

A Seifert fibration is a map from a smooth manifold to an orbifold
that becomes a proper fibration over the smooth covering of each orbifold
chart.
Each Seifert fibration is thus a Riemannian foliation with compact leaves.
The leaves over singular points of the orbifolds are quotients of
the generic leaf over a regular point.

We extend the adiabatic limit theorem of Bismut-Cheeger and Dai
to Seifert fibrations.
We have to take care of additional terms arising at the singular locus
of the orbifold.
In some special cases, these extra terms give rise to Dedekind sums.

\subsection{Orbifolds, Orbibundles and Seifert Fibrations}\label{A1}

We recall the definition of an orbifold.
By Remark~\ref{A1.R0} below,
we may assume that the base orbifold~$B$ is effective.
This will be assumed for the rest of the paper
and makes some constructions a lot easier.

\begin{Definition}\label{A1.D1}
  Let~$G$ be a compact Lie group together with an action on~$\R^n$.
  An $n$-dimensional smooth $G$-{\em orbifold\/}
  is a second countable Hausdorff space~$B$
  with the following additional structure.
  \begin{enumerate}
  \item\label{A1.D1.1}
    For each point~$b\in B$ there exists a neighbourhood~$U\subset B$ of~$b$,
    an open subset~$V\subset\R^n$ invariant
    under the action of a finite group~$\Gamma$
    via~$\rho\colon\Gamma\hookrightarrow G\to\GL(n,\R)$,
    and a homeomorphism
    \begin{equation*}
      \psi\colon\rho(\Gamma)\backslash V\to U\qquad\text{with}\qquad
      \psi(0)=b\;.
    \end{equation*}
    We call~$\psi$ an {\em orbifold chart,\/}
    and we call~$\rho$ the {\em isotropy representation\/}
    and~$\Gamma$ the {\em isotropy group\/} of~$b$ in~$B$.
  \item\label{A1.D1.2}
    If~$b\in U\subset B$ and~$\psi\colon\rho(\Gamma)\backslash V\to U$
    are as above,
    if~$b'\in U$,
    and if~$\psi'\colon\rho'(\Gamma')\backslash V'\to U'$
    are choosen analogously for~$b'$,
    then there exists an open embedding~$\phy\colon\psi^{\prime-1}(U)\to V$
    and a group homomorphism~$\thet\colon\Gamma'\to\Gamma$, such that
    \begin{equation*}
      \phy\circ\rho'_{\gamma'}=\rho_{\thet(\gamma')}\circ\phy
    \end{equation*}
    for all~$\gamma'\in\Gamma'$, and
    \begin{equation*}
      \psi\bigl(\rho(\Gamma)\,\phy(v')\bigr)
      =\psi'\bigl(\rho(\Gamma')v'\bigr)\;.
    \end{equation*}
    We call~$\phy$ a {\em coordinate change\/} and~$\thet$ an {\em intertwining
    homomorphism.\/}
  \end{enumerate}
  An {\em oriented orbifold\/} is an $\SO(n)$-orbifold
  where all coordinate changes are orientation preserving.
  A {\em spin orbifold\/} is an oriented $\Spin(n)$-orbifold.
\end{Definition}

We will say $n$-orbifold shortly for~$\O(n)$-orbifold,
and we will drop~$\rho$ from the notation when the action of~$\Gamma$
is clear from the context.
If~$\phy$ is a coordinate change
with intertwining homomorphism~$\thet$ as above,
then~$\rho_\gamma\circ\phy$ is another coordinate change
with intertwining homomorphism~$\gamma'\mapsto\gamma\thet(\gamma')\gamma^{-1}$
for each~$\gamma\in\Gamma$.
We do not impose any further condition (like a cocycle condition)
on the choices of coordinate changes and intertwining homomorphisms.

\begin{Definition}\label{A1.D2}
  Let~$B$ be a $G$-orbifold and let~$X$ be a smooth manifold.
  An {\em orbibundle with fibre\/}~$X$ is
  a map~$p$ from a topological space~$M$ to~$B$
  with the following extra structure.
  \begin{enumerate}
  \item\label{A1.D2.1}
    For each~$b\in B$, there exists an orbifold
    chart~$\psi\colon\rho(\Gamma)\backslash V\to U\subset B$ around~$b$,
    a fibre-preserving action~$\sigma$ of~$\Gamma$ by diffeomorphisms
    on~$V\times X$ covering~$\rho$ and a
    homeomorphism~$\bar\psi\colon\sigma(\Gamma)\backslash(V\times X)\to p^{-1}(U)$
    such that the diagram
    \begin{equation*}
      \begin{CD}
	V\times X@>>>\sigma(\Gamma)\backslash(V\times X)@>\bar\psi>>p^{-1}(U)\\
	@VVV@VVV@VVpV\\
	V@>>>\rho(\Gamma)\backslash V@>\psi>>U
      \end{CD}
    \end{equation*}
    commutes.
  \item\label{A1.D2.2}
    If~$\psi\colon\rho(\Gamma)\backslash V\to U$
    and~$\psi'\colon\rho'(\Gamma')\backslash V'\to U'$
    are orbifold charts as in Definition~\ref{A1.D1}~\eqref{A1.D1.2}
    with coordinate change~$\phy$ and intertwining homomorphism~$\thet$,
    and~$\sigma$, $\sigma'$ and~$\bar\psi$, $\bar\psi'$ are as above,
    then there exists a diffeomorphism~$\bar\phy\colon\psi^{\prime-1}(U)
    \times X\to V\times X$
    such that
    \begin{equation*}
      \bar\phy\circ\sigma'_{\gamma'}=\sigma_{\thet(\gamma')}\circ\bar\phy
    \end{equation*}
    for all~$\gamma'\in\Gamma'$, and such that the diagram
    \begin{equation*}
      \begin{CD}\qquad
	\psi^{\prime-1}(U)\times X
	  @>>>\sigma'(\Gamma')\backslash(\psi^{\prime-1}(U)\times X)
	  @>\bar\psi'>>p^{-1}(U\cap U')\\
	@V\bar\phy VV@VVV@VVV\\
	V\times X@>>>\sigma(\Gamma)\backslash(V\times X)@>\bar\psi>>p^{-1}(U)
      \end{CD}
    \end{equation*}
    commutes.
  \end{enumerate}
  If all actions~$\sigma$ are free,
  then~$M$ carries the structure of a smooth manifold,
  and we call~$p\colon M\to B$ a {\em Seifert fibration.\/}
  If~$X$ is a vector space and all actions~$\sigma$
  and all diffeomorphisms~$\bar\phy$ are fibrewise linear,
  then we call~$p\colon M\to B$ a {\em vector orbibundle.\/}
  If~$X=G$ is a Lie group and all~$\sigma$ and all~$\bar\phy$ commute
  with the right action of~$G$ on~$X$,
  then~$G$ acts on~$M$, and
  we call~$p\colon M\to B$ a $G$-{\em principal orbibundle.\/}
\end{Definition}

\begin{Remark}\label{A1.R0}
  Alternatively,
  a Seifert fibration with compact fibres
  is a connected manifold~$M$ with a Riemannian foliation~$\mathcal F$
  such that all leaves are compact.
  To see this,
  we pick a holonomy invariant metric on~$M$
  and let~$B=M/\mathcal F$ denote the space of leaves
  and~$p\colon M\to B$ the quotient map.

  Let~$L$ be a leaf with normal bundle~$N_L\to L$.
  By compactness, there exists~$r>0$ such that
  the normal exponential map~$\exp_L$ is an injective
  local diffeomorphism from the disc bundle~$N_r$ to~$M$.
  We use~$\exp_L$ to construct an orbifold chart for~$B$
  and an orbibundle chart for~$M$ around~$L$.
  Fix~$\ell\in L$ and let~$\tilde\rho_{L,\ell}\colon\pi_1(L,\ell)\to O(N_\ell)$
  denote the holonomy representation.
  Then~$\tilde\rho_{L,\ell}$ induces a representation
	$$\rho_{L,\ell}\colon\Gamma_{L,\ell}
		=\pi_1(L,\ell)/\ker(\tilde\rho_{L,\ell})
		\longrightarrow O(N_\ell)\;,$$
  and~$D_rN_\ell$ is a bundle chart for~$B$ around~$L$ with
  isotropy group~$\Gamma_{L,\ell}$ and isotropy representation~$\rho_{L,\ell}$.
  The transition maps and intertwining homomorphisms are not hard
  to construct, either.

  Moreover, let~$\tilde L$ denote the universal covering space of~$L$,
  so~$L=\pi_1(L,\ell)\backslash\tilde L$.
  Then~$\Gamma_{L,\ell}$ acts on
	$$X_{L,\ell}=\ker(\tilde\rho_{L,\ell})\backslash\tilde L\;,$$
  and because~$M$ is connected,
  all~$X_{L,\ell}$ are diffeomorphic.
  This way,
  we can construct orbibundle charts and transition maps
  as in Definition~\ref{A1.D2}.
\end{Remark}

If~$B$ is an orbifold,
then there is a natural tangent orbibundle~$TB\to B$.
There is a natural notion of a Riemannian metric on~$B$,
and such metrics always exist.

If~$p\colon M\to B$ is a Seifert fibration,
then there exists a natural map~$dp\colon TM\to TB$
and a well-defined vertical subbundle~$TX=\ker dp\subset TM$.
If~$g^{TM}$ is a Riemannian metric on~$M$,
let~$T^HM=(TX)^\perp\to M$ denote the horizontal subbundle.
Then~$g^{TM}$ is a {\em submersion metric\/}
if there exists a Riemannian metric~$g^{TB}$ on~$B$
such that~$dp|_{T^HM}$ is a fibrewise isometry.

\begin{Remark}\label{A1.R1}
  Whitney sums, Whitney tensor products, dual bundles and exterior powers
  can be defined for vector orbibundles over a base orbifold~$B$.
  However, because in general not all ``fibres'' of a vector orbibundle
  are vector spaces,
  one cannot apply these constructions fibrewise.
  Instead,
  one has to perform the respective constructions fibrewise
  with the bundle charts and transition maps of Definition~\ref{A1.D2}.
  By functoriality, the resulting collection of bundle charts and
  transition maps define another vector orbibundle on~$B$.
  Similarly,
  there is a natural notion of a {\em Dirac orbibundle\/} over an orbifold.

  If~$W\to B$ is a vector orbibundle with fibre~$\Bbbk^r$,
  the space of sections is given locally
  in a chart~$V\to\Gamma\backslash V\cong U$
  as a space of $\Gamma$-invariant maps
  	$$\Gamma(W|_U)
	\cong\mathcal C^\infty(V;\Bbbk^r)^\Gamma\;.$$
  After these preparations, we may now write
	$$\Omega^\bullet(B;W)=\Gamma(\Lambda^\bullet T^*B\otimes W)\;.$$
  If~$W$ is graded, the tensor product is understood in the graded sense.

  Let~$M\to B$ be a Seifert fibration with fibre~$X$,
  let~$T^HM$ denote a horizontal subbundle,
  and let~$V\to M$ be a vector bundle.
  Let~$\Omega^\bullet(M/B;W)\to B$ denote the infinite-dimensional
  vector orbibundle with fibre~$\Omega^\bullet(X;W|_X)$.
  Then
	$$\Omega^\bullet(M;W)
	\cong\Omega^\bullet\bigl(B;\Omega^\bullet(M/B;W)\bigr)\;,$$
  and this isomorphism depends explicitly on the choice of~$T^HM$.
  This follows by regarding the pullback of the local situation
  to bundle charts.

  In particular,
  all constructions of local family index theory such as adiabatic limits
  and Getzler rescaling are still well-defined for Seifert fibrations.
\end{Remark}

\subsection{The Inertia Bundle and Characteristic Classes}\label{A2}
Kawasaki's index theorem for orbifolds has been formulated for general
elliptic differential operators.
The topological index is formulated in terms
of characteristic classes of symbols.
For the task at hand, we need to specialise these classes to the case of
twisted Dirac operators.

We recall the definition of the inertia orbifold~$\Lambda B$
of~$B$ in~\cite{Kawa}.
Its points are given as pairs~$(p,(\gamma))$,
where~$p\in B$ and~$(\gamma)$ is the $\Gamma$-conjugacy class
of an element of the isotropy group~$\Gamma$ of~$p$.
If~$\psi\colon\Gamma\backslash V\to U$ is an orbifold chart for~$B$
around~$p=\psi(0)$,
we obtain an orbifold chart
\begin{equation}\label{A2.1}
  \psi_{(\gamma)}\colon C_\Gamma(\gamma)\backslash V^\gamma
  \to\psi(V^\gamma)\times\{(\gamma)\}\subset\Lambda B
\end{equation}
by restriction,
where~$V^\gamma$ denotes the fixpoint set of~$\gamma$
and~$C_\Gamma(\gamma)$ is the centraliser of~$\gamma$ in~$\Gamma$.
In general,
the inertia orbifold is no longer effective.
Hence,
let
\begin{equation}\label{A2.0}
  m(\gamma)
  =\#\,\bigl\{\,\thet\in C_\Gamma(\gamma)\bigm|\rho_\thet|_{V^\gamma}
	=\id_{V^\gamma}\,\bigr\}
\end{equation}
denote the {\em multiplicity\/} of~$(p,(\gamma))\in\Lambda B$.
Then~$m(\gamma)$ defines a locally constant function on~$\Lambda B$.

Let~$N_\gamma\to V^\gamma$ denote the normal bundle
to~$V^\gamma$ in~$V$,
and let~$R^{N_\gamma}$ be the curvature of the connection
on~$N_\gamma$ induced by the pullback of the Levi-Civita connection.
Let~$\tilde\gamma$ denote a lift of the action of~$\gamma$
on~$N_\gamma$ to the spin group under the natural projection~$\Spin(N_\gamma)
\to\SO(N_\gamma)$.
If~$B$ is a spin orbifold, such a lift is part of the orbifold spin structure.
Otherwise, the lift~$\tilde\gamma$ is determined uniquely up to sign.
Hence, the inertia orbifold has a natural double cover
\begin{equation}\label{A2.4}
  \widetilde{\Lambda B}
  =\bigl\{\,\bigl(p,(\tilde\gamma)\bigr)\bigm|
	\tilde\gamma\text{ lifts }\gamma\,\bigr\}
  \longrightarrow\Lambda B\;.
\end{equation}
As in~\eqref{A2.1},
one constructs charts for~$\widetilde{\Lambda B}$ by
\begin{equation}\label{A2.5}
  \psi_{(\tilde\gamma)}\colon C_\Gamma(\gamma)\backslash V^\gamma
  \to\psi(V^\gamma)\times\{(\tilde\gamma)\}\subset\widetilde{\Lambda B}\;.
\end{equation}

The equivariant Chern character form of a Hermitian vector bundle~$(E,\nabla^E)$
with connected,
equipped with a parallel fibrewise automorphism~$g$, is classically defined as
\begin{equation}\label{A2.6}
	\ch_g(E,\nabla^E)
	=\trace\biggl(g\,e^{-\tfrac{(\nabla^E)^2}{2\pi i}}\biggr)\;.
\end{equation}
There exists a local spinor bundle~${SN}_\gamma\to V^\gamma$ for~$N_\gamma$.
Given a local orientation of~$N_\gamma$,
there is a natural local splitting~${SN}_\gamma={S^+N}_\gamma\oplus{S^-N}_\gamma$.
Using~$R^{N_\gamma}$ and a lift~$\tilde\gamma$ of~$\gamma$ as above,
we can define
the equivariant $\hat A$-form on~$V^\gamma$ by
\begin{equation}\label{A2.2}
  \hat A_{\tilde\gamma}\bigl(TV,\nabla^{TV}\bigr)
  =(-1)^{\tfrac{\rk N_\gamma}2}\,\frac{\hat A(TV^\gamma,\nabla^{TV^\gamma})}
	{\ch_{\tilde\gamma}(S^{+}N_\gamma-S^{-}N_\gamma,\nabla^{{SN}_\gamma})}\;.
\end{equation}

\begin{Remark}\label{A2.R1}
  This construction of~$\hat A_{\tilde\gamma}(TV,\nabla^{TV})$ has the following
  properties.
  \begin{enumerate}
  \item\label{A2.R1.1}
    Because~$1$ is not an eigenvalue of~$\gamma|_{N_\gamma}$,
    the denominator is invertible in~$\Omega^\bullet(V^\gamma)$.
    In fact, as explained in~\cite[section~6.4]{BGV},
    one has
	$$\ch_{\tilde\gamma}(S^{+}N_\gamma-S^{-}N_\gamma,\nabla^{{SN}_\gamma})
	=\pm i^{\tfrac{\rk N_\gamma}2}\,\det\nolimits_{N_\gamma}
	\biggl(\id-\gamma\,e^{-\tfrac{(\nabla^E)^2}{2\pi i}}\biggr)^{\tfrac12}\;.$$
  \item\label{A2.R1.2}
    The form~$\hat A_{\tilde\gamma}(TV,\nabla^{TV})$ only depends
    on the conjugacy class of~$\tilde\gamma$.
  \item\label{A2.R1.3}
    Replacing the lift~$\tilde\gamma$ of~$\gamma$ by the lift~$-\tilde\gamma$
    changes the sign of~$\hat A_{\tilde\gamma}(TV,\nabla^{TV})$.
  \item\label{A2.R1.4}
    If one changes the orientation on~$V^\gamma$ but keeps the orientation
    of the total tangent space~$TV|_{V^\gamma}$,
    then the orientation of~$N_\gamma$ changes as well,
    and the subbundles~$S^+N_\gamma$ and~$S^-N_\gamma$ are swapped.
    Hence the form~$\hat A_{\tilde\gamma}(TV,\nabla^{TV})$ changes its sign.
    On the other hand, its integral over the corresponding stratum
    of~$\Lambda B$ then does not depend on the orientation chosen on~$V^\gamma$,
    only on the orientation of~$V$.
  \end{enumerate}
\end{Remark}

  Let us introduce the notation
  \begin{equation*}
    \Omega^\bullet\bigl(\Lambda B;\widetilde{\Lambda B}\bigr)
    =\,\bigl\{\,\alpha\in\Omega^\bullet\bigl(\widetilde{\Lambda B}\bigr)\bigm|
	\alpha|_{(p,(-\tilde\gamma))}=-\alpha|_{(p,(\tilde\gamma))}\,\bigr\}\;,
  \end{equation*}
  and let us denote by~$\Omega^\bullet\bigl(\Lambda B;
  \widetilde{\Lambda B}\otimes o(\Lambda V)\bigr)$
  the space of forms that change sign depending
  on the choice of a local orientation of~$\Lambda B$.
  We assume that~$B$ is an oriented orbifold.
  By~\eqref{A2.R1.2}--\eqref{A2.R1.4},
  the forms~$\hat A_{\tilde\gamma}(TV,\nabla^{TV})$ in local coordinates
  can be used to construct a well-defined
  form~$\hat A_{\Lambda B}(TB,\nabla^{TB})
  \in\Omega^\bullet(\Lambda B;\widetilde{\Lambda B}\otimes o(\Lambda V))$ with
  \begin{equation}\label{A2.3}
    \psi_{(\tilde\gamma)}^*\hat A_{\Lambda B}\bigl(TB,\nabla^{TB}\bigr)
    =\frac1{m(\gamma)}\,\hat A_{\tilde\gamma}\bigl(TV,\nabla^{TV}\bigr)\;,
  \end{equation}
  where~$m(\gamma)$ is the multiplicity of~\eqref{A2.0}.
  With the choice of~$\tilde\gamma$ given by a spin structure,
  this is the integrand in Kawasaki's orbifold index theorem~\cite{Kawa}
  when specialised to untwisted Dirac operators.

Let~$(E,\nabla^E,g^E,c)$ denote a Dirac orbibundle over~$V$,
and let~$\gamma^E$ be a compatible action of~$\gamma$ on~$E$.
Then~$\tilde\gamma^{E/S}=\gamma^E\cdot\tilde\gamma^{-1}$ commutes with
Clifford multiplication and has the same sign ambiguity as~$\tilde\gamma$.
If we write~$E= SM\otimes W$ locally,
then~$W$ carries a natural connection~$\nabla^W$
with curvature~$R^W=R^{E/S}$, and~$\tilde\gamma^{E/S}$ acts on~$W$,
and we can define the equivariant twist Chern character
form
	$$\ch_{\tilde\gamma}(E/S,\nabla^E)=\ch_{\tilde\gamma^{E/S}}(W,\nabla^W)\;.$$
Then~$\hat A_{\tilde\gamma}(TV,\nabla^{TV})\,\ch_{\tilde\gamma}(E/S,\nabla^E)$
is the integrand
in the local Atiyah-Segal-Singer equivariant index theorem.
Berline, Getzler and Vergne propose a particular choice of the
lift~$\tilde\gamma$ in~\cite[section~6.4]{BGV}.

Because~$\ch_{\tilde\gamma}(E/S,\nabla^E)$ only depends on the conjugacy
class and the sign of the lift~$\tilde\gamma$, 
there exists a well-defined class~$\ch_{\Lambda B}(E/S,\nabla^E)
\in\Omega^\bullet\bigl(\Lambda B;\widetilde{\Lambda B}\bigr)$ such that
	$$\psi_{(\tilde\gamma)}^*\ch_{\Lambda B}(E/S,\nabla^E)
	=\ch_{\tilde\gamma}(E/S,\nabla^E)\;.$$
Then~$\hat A_{\Lambda B}(TB,\nabla^{TB})\,\ch_{\Lambda B}(E/S,\nabla^E)
\in\Omega^\bullet(\Lambda B;o(\Lambda B))$
is the integrand in Kawasaki's index theorem for orbifolds when specialised
to twisted Dirac operators.
In particular,
the integral over~$\Lambda B$ does not depend on the choices above.
In the special case of an untwisted Dirac operator,
we have~$E=S$, and~$\gamma^E$ is itself a lift of~$\gamma$.
In this case, we simply have~$\ch_{\tilde\gamma}(E/S,\nabla^E)=1$
if we take~$\tilde\gamma=\gamma^E$, and~$\ch_{\tilde\gamma}(E/S,\nabla^E)=-1$
otherwise.

\subsection{Adiabatic Limits}\label{A3}
If~$M\to B$ is a Seifert fibration and~$g^{TM}$ is a submersion metric
as in Section~\ref{A1},
we obtain a family of submersion metrics~$(g^{TM}_\eps)_{\eps>0}$ with the same
horizontal bundle~$T^HM\to M$
such that~$g^{TM}_\eps|_{TX}=g^{TM}|_{TX}$
and~$g^{TM}_\eps|_{T^HM}=\eps^{-2}\,g^{TM}|_{T^HM}$.
The limit~$\eps\to0$ is called the adiabatic limit.

Let~$e_1$, \dots $e_n$ and~$f_1$, \dots, $f_{m-n}$ be local orthonormal frames
of~$TX$ and~$TB$.
The horizontal lift of a vector field~$v$ on~$B$ will be denoted by~$\bar v$.
Then a local orthonormal frame of~$TM$ for~$g_\eps$ is given by
\begin{equation}\label{A3.0}
	e_1^\eps=e_1,\quad \dots,\quad e_n^\eps=e_n,\quad
	e_{n+1}^\eps=\eps\bar f_1,\quad \dots,\quad e_m^\eps=\eps\bar f_{m-n}\;.
\end{equation}

\begin{Definition}\label{A3.D1}
  An {\em adiabatic family of Dirac bundles\/} for~$p\colon M\to B$
  consists of a Hermitian vector bundle~$(E,g^E)$,
  a Clifford multiplication~$c\colon TM\to\End E$,
  and a family of connections~$(\nabla^{E,\eps})_{\eps\ge 0}$,
  such that
  \begin{enumerate}
  \item\label{A3.D1.1}
    The quadruple~$(E,\nabla^{E,\eps},g^E,c^\eps)$
    is a Dirac bundle on~$(M,g^{TM}_\eps)$ for all~$\eps>0$,
    where the Clifford multiplication~$c^\eps$
    is given by~$c^\eps(e^\eps_I)=c(e^1_I)$.
  \item\label{A3.D1.2}
    The connection~$\nabla^{E,\eps}$ is analytic in~$\eps$ around~$\eps=0$.
  \item\label{A3.D1.3}
    The kernels of the fibrewise Dirac operators
	$$D_X=\sum_{i=1}^nc(e_i)\nabla^{E,0}_{e_i}$$
    acting on~$E|_{p^{-1}(b)}$ form a vector orbibundle~$H\to B$.
  \end{enumerate}
  We will call the associated family~$(D_{M,\eps})_{\eps>0}$ with
	$$D_{M,\eps}=\sum_{I=1}^mc^\eps(e^\eps_I)\nabla^{E,\eps}_{e_I}$$
  an {\em adiabatic family of Dirac operators\/} for~$p$.
\end{Definition}

We consider the infinite
dimensional vector orbibundle~$p_*E\to B$ with
	$$p_*E|_b=\Gamma(E|_{p^{-1}(b)})$$
for all regular points~$b\in B_0$.
It carries a fibrewise $L^2$-metric~$g^{p_*E}_{L^2}$
that is independent of~$\eps$.

Associated to~$(E,\nabla^{E,0},g^E,c)$
is Bismut's Levi-Civita superconnection
\begin{equation}\label{A3.1}
  \A_t=t^{\tfrac12}\,\A_0+\A_1+t^{-\tfrac12}\,\A_2\;,
\end{equation}
on~$p_*E\to B$ for~$t>0$, see~\cite{BGV}, \cite{BC1}.
Here, $\A_0=D_X$ is the fibrewise Dirac operator of~\eqref{A3.D1.3} above.
The part~$\A_1=\nabla^{p_*E,0}$ is the unitary connection on~$p_*E$ that is
induced by
\begin{equation}\label{A3.2}
  \nabla^{E,0}-\frac12\,h\;,
\end{equation}
where~$h\colon T^HM\to\End TX$ is the mean curvature of the fibres of~$p$,
and~$\A_2$ is an endomorphism of~$E\to M$ with coefficients in~$\Lambda^2T^*B$.

Let~$\gamma\in\Gamma$ be an element of the isotropy group of~$b\in B$,
then~$\gamma$ acts on~$p_*E$.
If~$SB\to V^\gamma$ denotes a local spinor bundle on~$B$,
then there exists a fibrewise Dirac bundle~$W\to M$ such that as vector bundles,
locally
\begin{equation}\label{A3.4}
  E\cong p^*SB\otimes W\longrightarrow p^{-1}(V_\gamma)
  \qquad\text{and}\qquad
  p_*E\cong SB\otimes p_*W\longrightarrow V^\gamma\;.
\end{equation}
After choosing a lift~$\tilde\gamma\in\Spin(N_\gamma)$ of the action of~$\gamma$ on~$N_\gamma$ as in~\eqref{A2.4}, we can split~$\gamma=\tilde\gamma^W\circ\tilde\gamma$, see~\eqref{B6.2} below.
Over~$V^\gamma$,
we consider the {\em equivariant $\eta$-form\/}
\begin{equation}\label{A3.3}
  \eta_{\tilde\gamma}(\A)
  =\int_0^\infty\frac1{\sqrt\pi}\,(2\pi i)^{-\tfrac{N^{V^\gamma}}2}
	\,\trace_{p_*W}
		\biggl(\tilde\gamma^W\,\frac{\del\A_t}{\del t}\,e^{-\A_t^2}\biggr)\,dt
  \in\Omega^{\bullet}\bigl(V^\gamma\bigr)
\end{equation}
as in~\cite[Remark~3.12]{Go},
where~$N^{V^\gamma}$ denotes the number operator on~$\Omega^\bullet(V^\gamma)$.
Again, the sign of~$\eta_{\tilde\gamma}(\mathbb A)$ depends
on the choice of~$\tilde\gamma$; if~$B$ is a spin orbifold,
then this choice is natural.
Note that in contrast to~\cite{Go},
we already eliminate $2\pi i$-factors inside the differential
form~$\eta_\gamma(\A)$ and not after integration.
By assumption~\eqref{A3.D1.3} above,
the integral converges uniformly near~$t=\infty$
because the operators~$D_X$ have a uniform spectral gap around
the possible eigenvalue~$0$.
If~$\tilde\gamma=e$ is the neutral element,
then~$\eta_{\tilde\gamma}(\A)=\eta(\A)$ is the $\eta$-form of Bismut and Cheeger,
and the integral in~\eqref{A3.3} also converges near~$t=0$ by~\cite{BC1}.
Otherwise, $\gamma$ acts freely on the fibres,
and small time convergence is not an issue.

\begin{Definition}\label{A3.D2}
  The {\em orbifold $\eta$-form\/}~$\eta_{\Lambda B}(\A)\in\Omega^\bullet\bigl(\Lambda B;\widetilde{\Lambda B}\bigr)$ is defined such that in the orbifold charts
  of~\eqref{A2.5},
  \begin{equation}\label{A3.E1}
    \psi_{(\tilde\gamma)}^*\eta_{\Lambda B}(\A)
    =\eta_{\tilde\gamma}(\A)\;.
  \end{equation}
\end{Definition}

This is well-defined because~$\eta_\gamma(\A)$ only depends
on the conjugacy class and the sign of~$\tilde\gamma$.
Moreover, the integrand~$\hat A_{\Lambda B}\bigl(TB,\nabla^{TB}\bigr)\,2\eta_{\Lambda B}(\A)\in\Omega^\bullet(\Lambda B;o(\Lambda B))$ in the first term on the right hand side of Theorem~\ref{Thm1}
depends only on the orientation of the fibres of~$p\colon M\to B$;
in particular, the integral over~$\Lambda B$ only depends
on the global orientation of~$M$ by Remark~\ref{A2.R1}~\eqref{A2.R1.4}.
Note the different normalisation of $\eta$-forms and $\eta$-invariants.
The component of~$2\eta_\gamma(\A)$ of degree~$0$ in~$\Omega^\bullet(\Lambda B)$
is the equivariant $\eta$-invariant of the fibre.
This explains the additional factor~$2$ in the integrand in Theorem~\ref{Thm1}.

Locally, there exists a unique family of
spinor bundles~$(SM,\nabla^{SM,\eps},g^{SM},c)_\eps$ on~$(M,g^{TM}_\eps)$,
and the Dirac bundle~$E$ splits as~$E\cong SM\otimes W$.
There exists a locally uniquely defined family
of connections~$\nabla^{E/S,\eps}=\nabla^{W,\eps}$
such that~$\nabla^{E,\eps}$ is the tensor product connection
induced by~$\nabla^{SM,\eps}$ and~$\nabla^{W,\eps}$.
Note that for the family of odd signature operators~$B_{M,\eps}$,
we cannot assume that~$\nabla^{E/S,\eps}$ is independent of~$\eps$.
We obtain globally well-defined endomorphism-valued differential forms
\begin{equation}\label{A3.5}
  \nabla^{E/S,\eps}-\nabla^{E/S,0}\in\Omega^1(M;\End E)
	\qquad\text{and}\qquad
  R^{E/S,\eps}\in\Omega^2(M;\End E)
\end{equation}
that commute with Clifford multiplication.
The curvature~$R^{E/S,\eps}$ of~$\nabla^{E/S,\eps}$
is called the twisting curvature in~\cite{BGV}.
By assumption~\eqref{A3.D1.2} above, both~$\nabla^{E/S,\eps}-\nabla^{E/S,0}$
and~$R^{E/S,\eps}$ depend analytically on~$\eps$ around~$\eps=0$.

Let~$P_X\colon p_*E\to H$ denote the $L^2$-orthogonal projection
onto~$H=\ker D_X$.

\begin{Definition}\label{A3.D3}
  The {\em effective horizontal operator\/} of an adiabatic family
  of Dirac bundles~$(E,\nabla^{E,\eps},g^E,c)$ is defined as
	$$D_B^{\mathrm{eff}}=P_X\circ\biggl(
	\sum_{\alpha=1}^{m-n}c(\bar f_\alpha)\,\nabla^{p_*E,0}_{f_\alpha}
	+\sum_{i=1}^nc(e_i)\,\frac d{d\eps}\Bigr|_{\eps=0}\nabla^{E/S,\eps}_{e_i}
	\biggr)\circ P_X\;.$$
\end{Definition}

The operator~$D_B^{\mathrm{eff}}$ is selfadjoint.
Its $\eta$-invariant is further investigated in Proposition~\ref{B2.P2}.
If the fibres are odd-dimensional,
we will see in section~\ref{A5} that in important special cases,
the $\eta$-invariant of the effective horizontal operator vanishes.

Again by assumption~\eqref{A3.D1.2} above,
there exists~$\eps_0>0$ such that the kernel of~$D_{M,\eps}$
has constant dimension for all~$\eps\in(0,\eps_0)$.
By~\cite{Dai} and section~\ref{B5} below,
there are finitely many eigenvalues~$\lambda_\nu(\eps)$ of~$D_{M,\eps}$
(counted with multiplicity),
called the ``very small eigenvalues'',
such that
\begin{equation}\label{A4.1}
  \lambda_\nu(\eps)=O(\eps^2)
	\qquad\text{and}\qquad
  0\ne\lambda_\nu(\eps)
	\quad\text{for all }\eps\in(0,\eps_0)\;.
\end{equation}
We have now defined all ingredients of Theorem~\ref{Thm1}.
Its proof is deferred to section~\ref{B}.

\subsection{Special cases of the adiabatic limit theorem}\label{A5}
We consider fibres of positive scalar curvature,
and the signature operator.

We assume first that~$M$ and~$B$ are spin,
then the vertical tangent bundle of~$p\colon M\to B$ also carries a spin
structure.
By abuse of notation,
we write~$\eta_{\Lambda B}(D_{M/B})$ instead of~$\eta_{\Lambda B}(\A)$
for the orbifold $\eta$-form associated to the untwisted Dirac operator.

If the fibres of~$M\to B$ have positive scalar curvature,
then for the untwisted Dirac operator~$D_{M,\eps}$,
the fibrewise operator is invertible,
hence~$H=0$ and there is neither a effective horizontal operator
nor are there very small eigenvalues.
In particular,
$D_{M,\eps}$ satisfies the conditions of Dai's theorem.
The same still holds for the Dirac operator~$D^{p^*W}_{M,\eps}$
that is twisted by the pullback
of an orbibundle~$W\to B$ with connection~$\nabla^W$.

\begin{Corollary}\label{A5.C1}
  If the fibration~$M\to B$ and~$B$ are spin,
  the fibres of~$M\to B$ have positive scalar curvature,
  and if~$W\to B$ is an orbibundle,
  then
  \begin{equation*}
    \lim_{\eps\to 0}\eta\bigl(D^{p^*W}_{M,\eps}\bigr)
    =\int_{\Lambda B}\hat A_{\Lambda B}\bigl(TB,\nabla^{TB}\bigr)
    		\,\ch_{\Lambda B}\bigl(W,\nabla^W\bigr)
		\,2\eta_{\Lambda B}\bigl(D_{M/B}\bigr)\;.
  \end{equation*}
\end{Corollary}

\begin{proof}
  If~$W$ is trivial, then the corollary follows from Theorem~\ref{Thm1}
  by the considerations above.
  If~$W\to B$ is an orbibundle,
  then the result follows from remark~\ref{B6.R1}.
\end{proof}

The odd signature operator~$B_{M,\eps}$ on~$M$ also satisfies
the conditions of Dai's theorem.
Here, the bundle~$H\to B$ corresponds to the fibrewise cohomology,
regarded as a $\Z_2$-graded vector bundle.
In contrast to~\cite{Dai},
we regard~$B_{M,\eps}$ as an operator on~$\Omega^\even(M)$,
not on all forms.
Note that the twisting curvature depends on~$\eps$.

There is a natural notion of a differentiable Leray-Serre spectral
sequence of~$M\to B$,
and by Mazzeo and Melrose~\cite{MM},
the very small eigenvalues~$(\lambda_\nu(\eps))_\nu$ of~$B_{M,\eps}$
are related to its higher differentials.
The effective horizontal operator~$B_B^{\mathrm{eff}}$ is related to the $E_1$-term
of this sequence by results of Dai~\cite[section~4.1]{Dai}.
Dai also constructs a signature~$\tau_r\in\Z$ on the $r$-th term~$E_r$
of this spectral sequence for all~$r\ge 2$.

Let~$N^B$ denote the number operator on~$\Omega^\bullet(B)$,
then the rescaled $L$-class
\begin{equation*}
  \hat L\bigl(TB,\nabla^{TB}\bigr)
  =\hat A\bigl(TB,\nabla^{TB}\bigr)\,\ch\bigl(SB,\nabla^{SB}\bigr)
  =2^{\tfrac{\dim B-N^B}2}\,L\bigl(TB,\nabla^{TB}\bigr)
\end{equation*}
has a natural general equivariant generalisation
leading to~$\hat L_{\Lambda B}\bigl(TB,\nabla^{TB}\bigr)
\in\Omega^\bullet(\Lambda B)$.
Finally,
let us write~$\eta_{\Lambda B}(B_{M/B})$ instead of~$\eta_{\Lambda B}(\A)$
in this setting.

\begin{Corollary}[cf.\ Dai~\cite{Dai}, Theorem~0.3]\label{A5.C2}
  If the fibration~$M\to B$ is oriented,
  then
  \begin{equation*}
    \lim_{\eps\to 0}\eta\bigl(B_{M,\eps}\bigr)
    =\int_{\Lambda B}\hat L_{\Lambda B}\bigl(TB,\nabla^{TB}\bigr)
		\,2\eta_{\Lambda B}\bigl(B_{M/B}\bigr)
    +\eta\bigl(B_B^{\mathrm{eff}}\bigr)
    +\sum_{r=2}^\infty\tau_r\;.
  \end{equation*}
  Moreover,
  if~$\dim B$ is even, then~$\eta(B_B^{\mathrm{eff}})=0$.
\end{Corollary}

\begin{proof}
  This follows from Theorem~\ref{Thm1} as in Dai's paper~\cite{Dai}.
  The first term again arises because of Remark~\ref{B6.R1}.
  The vanishing of~$\eta(B_B^{\mathrm{eff}})$ for even dimensional base
  orbifolds follows from Proposition~\ref{B2.P3}.
\end{proof}

\subsection{Seifert fibrations with compact structure group}\label{A6}
We assume that the fibres of the map~$p\colon M\to B$
are totally geodesic submanifolds of~$M$.

Assume for the moment that~$B$ is a connected Riemannian manifold
and that~$p\colon M\to B$ is an ordinary Riemannian submersion.
Each path in~$B$ induces a parallel transport
between the fibres over its endpoints.
By a result of Hermann~\cite{Hermann},
all these parallel translations are isometries if and only if the fibres of~$p$
are totally geodesic.
In this case, let~$X$ be isometric to a fibre of~$p$
and let~$G$ denote the isometry group of~$X$ acting from the left.
Then there is a natural fibre $G$-principal bundle
	$$P=\bigl\{\,f\colon X\to M
	\bigm|f\text{ is an isometry onto a fibre of }p\,\bigr\}$$
with a natural right $G$-action,
and we have~$M=P\times_GX$.

If we are given an adiabatic family~$(E,\nabla^{E,\eps},g^E,c)$
of Dirac bundles as in Definition~\ref{A3.D1},
then we assume further that the parallel transport between fibres
lifts to isomorphisms between the restrictions of~$(E,\nabla^{E,\eps},g^E,c)$
to the fibres of~$p$.
In this case,
let~$G$ denote the automorphism group of
	$$\bigl((E,\nabla^{E,\eps},g^E,c)|_X\bigr)
		\longrightarrow\bigl(X,g^X\bigr)\;,$$
then the family~$p\colon M\to B$ is still associated to a $G$-principal
bundle~$P\to B$.
In this case, we say that the adiabatic family~$(E,\nabla^{E,\eps},g^E,c)$
has {\em compact structure group\/}~$G$.

Let~$b\in B$ and identify~$p^{-1}(b)$ with~$X$ and~$E_{p^{-1}(b)}$
with~$E|_X\to X$,
then for~$v$, $w\in T_bB$,
the fibre bundle curvature~$\overline{[v,w]}-[\bar v,\bar w]$
together with its natural action on~$E$ is described by an
element~$\Omega(v,w)\in\frg$.
Different identifications of~$E\to X$ with~$E|_{p^{-1}(b)}$ give
elements of~$\frg$ in the same $\Ad_G$-orbit.

By~\cite[Lemma~1.14]{Go},
there exists an $\Ad_G$-invariant formal power
series~$\eta_{\mathfrak g}(D_X)\in\C[\![\mathfrak g]\!]$,
the {\em infinitesimally equivariant $\eta$-invariant,\/} such that
\begin{equation}\label{A6.2}
  2\eta(\mathbb A)=\eta_{\frac{\Omega}{2\pi i}}(D_X)\;.
\end{equation}
Here, $D_X$ is the component of~$\A$ of degree~$0$,
regarded as an operator on a bundle~$W\to X$ such that locally,
$E\cong W\otimes p^*SB$.
This invariant has been computed for the untwisted Dirac operator
and the signature operator on~$S^3$ in~\cite[Theorem~3.9]{Go}.
A more general formula for quotients of compact Lie groups
with normal metrics can be found in~\cite[section~2.4]{Gohome}.

Now, let~$p\colon M\to B$ be a Seifert fibration with generic fibre~$X$
and assume again that all fibres are totally geodesic.
Then the construction above still applies to bundle charts as in
Definition~\ref{A1.D2}.
If we trivialise~$p$ over~$V$ by parallel translation along radial geodesics
in~$V$,
then the isotropy group~$\Gamma$ acts on~$V\times X$ by
	$$\sigma_\gamma(v,x)=\bigl(\rho_\gamma(v),\sigma_\gamma(x)\bigr)$$
with~$\sigma_\gamma\in G$ for all~$\gamma\in\Gamma$.
Thus, we obtain a $G$-principal orbibundle
	$$P=\bigl\{\,f\colon X\to M
	\bigm|f\text{ is a local isometry onto a fibre of }p\,\bigr\}\;,$$
and again, we have~$M=P\times_GX\to B$.
Moreover, for~$\gamma\in\Gamma$,
the restricted curvature~$\Omega|_{TV^\gamma}$
takes values in the $\Ad_\gamma$-invariant part of~$\mathfrak g$.
Thus, if~$(p,(\gamma))\in\Lambda B\setminus B$,
let~$\psi_{(\gamma)}\colon C_\Gamma(\gamma)\backslash V^\gamma
\to\Lambda B$ be an orbifold chart for~$\Lambda B$ around~$(p,(\gamma))$
as in~\eqref{A2.1}.
We regard the pullback of~$M\to B$ restricted to~$V^\gamma$
and identify~$\gamma$ with~$\sigma_\gamma\in G$ acting on~$X$ and~$E$.
Then~$\Omega|_{V^\gamma}$ takes values in the Lie algebra~$\mathfrak c(\sigma_\gamma)$
of the centraliser~$C_G(\sigma_\gamma)$ of~$\sigma_\gamma$ in~$G$.

\begin{Theorem}\label{A6.T1}
  Let~$p\colon M\to B$ be a Seifert fibration,
  and let~$(E,\nabla^{E,\eps},g^E,c)$ be an adiabatic family
  with compact structure group~$G$.
  For each~$(p,(\gamma))\in\Lambda B$,
  there exists a formal power series
	$$\eta_{g,\mathfrak c(\sigma_\gamma)}(D_X)
		\in\R[\![\mathfrak c(\sigma_\gamma)]\!]$$
  such that the orbifold $\eta$-form is given
  in an orbifold chart~$\psi_{(\tilde\gamma)}$
  around~$(p,(\gamma))$ as
	$$\psi^*_{(\tilde\gamma)}\eta_{\Lambda B}(\A)
	=\eta_{\sigma_\gamma,\frac{\Omega}{2\pi i}}(D_X)\;.$$

  If~$\gamma=\id$, then~$\eta_{g,\mathfrak c(\sigma_\gamma)}(D_X)
  =\eta_{\mathfrak g}(D_X)$ is the infinitesimally equivariant $\eta$-invariant.
  If~$\gamma$ acts freely on the typical fibre~$X$,
  then~$\eta_{\gamma,\mathfrak c(\sigma_\gamma)}(D_X)$
  is the formal power series expansion
  of the classical equivariant $\eta$-invariant~$\eta_{\sigma_\gamma\,e^{-\Xi}}(D_X)$
  at~$\Xi=0\in\mathfrak c(\sigma_\gamma)$.
\end{Theorem}

\begin{proof}
  If~$\gamma=\id$,
  this is just \cite[Lemma~1.14]{Go}.
  If~$\gamma\ne\id$,
  then~$\gamma$ acts freely on the fibre~$X$ because~$p\colon M\to B$
  is a Seifert fibration,
  and the result is explained and proved in~\cite{Go}, Remark~3.12.
\end{proof}

Note that over each singular stratum of~$B$,
the fibres of~$p$ are finite quotients of~$X$,
so that we are in a situation similar to Lemma~3.11 in~\cite{Go}.

\section{A proof of the adiabatic limit Theorem}\label{B}
In this section,
we sketch a proof of Theorem~\ref{Thm1}.
We will omit most of the details,
in particular those explained by Bismut, Cheeger in~\cite{BC1}
and by Dai in~\cite{Dai}.
The proof is based on the well-known formula
\begin{equation*}
  \eta(D_{M,\eps})=\int_0^\infty\frac1{\sqrt{\pi t}}
	\,\trace\Bigl(D_{M,\eps}\,e^{-tD_{M,\eps}^2}\Bigr)\,dt\;.
\end{equation*}
We define a spectral projection~$P_\eps$ onto the sum of the eigenspaces
for the very small eigenvalues in section~\ref{B4},
which commutes with~$D_{M,\eps}$ for each~$\eps>0$.
We also find a small constant~$\alpha>0$ and write
\begin{align*}
  \eta(D_{M,\eps})
  &=\int_0^{\eps^{\alpha-2}}\frac1{\sqrt{\pi t}}
	\,\trace\Bigl(D_{M,\eps}\,e^{-tD_{M,\eps}^2}\Bigr)\,dt\\
  &\qquad+\int_{\eps^{\alpha-2}}^\infty\frac1{\sqrt{\pi t}}
	\,\trace\Bigl((1-P_\eps)\,D_{M,\eps}\,e^{-tD_{M,\eps}^2}\Bigr)\,dt\\
  &\qquad+\int_{\eps^{\alpha-2}}^\infty\frac1{\sqrt{\pi t}}
	\,\trace\Bigl(P_\eps\,D_{M,\eps}\,e^{-tD_{M,\eps}^2}\Bigr)\,dt\;.
\end{align*}
The three terms on the right hand side give rise to the three
expressions on the right hand side of Theorem~\ref{Thm1}
by Propositions~\ref{B6.P1}, \ref{B4.P1} and~\ref{B5.P1},
respectively,
which will be stated and proved below.
Thus, Theorem~\ref{Thm1} follows from the results of this section.

\subsection{Local Computations}\label{B1}
We will use small roman indices in~$\{1,\dots, n\}$ referring to coordinates
of the fibres,
small greek indices in~$\{n+1,\dots,m\}$ referring to coordinates of the base,
and capital indices in~$\{1,\dots, m\}$.
Let~$\nabla^{TM,\eps}$ denote the Levi-Civita connection with respect
to the bundle-like metric~$g_\eps=g^{TX}\oplus\eps^{-2}\,p^*g^{TB}$.

Let~$\nabla^{TB}$ denote the Levi-Civita connection
on the orbibundle~$TB\to B$.
By~\cite{BC1}, there exists a connection~$\nabla^{TX}$ on~$TX\to M$,
a symmetric tensor~$S\colon TX\otimes TX\to T^HM$
and an antisymmetric tensor~$T\colon T^HM\otimes T^HM\to TX$,
with coefficients~$s_{ij\gamma}$ and~$t_{\alpha\beta k}$,
such that
\begin{align*}
  \nabla^{TM,\eps}_{e_i}e_j
  &=\nabla^{TX}_{e_i}e_j+\eps\,\sum_\gamma s_{ij\gamma}\,e_\gamma^\eps \;,\\
  \nabla^{TM,\eps}_{e_\alpha}e_j
  &=\nabla^{TX}_{e_\alpha}e_j
	-\eps\,\sum_\beta t_{\alpha\beta j}\,e_\beta^\eps \;,\\
  \nabla^{TM,\eps}_{e_i}e_\beta^\eps 
  &=-\eps\,\sum_ks_{ik\alpha}e_k
	-\eps^2\,\sum_\gamma t_{\beta\gamma i}\,e_\gamma^\eps \;,\\
	\text{and}\qquad
  \nabla^{TM,\eps}_{e_\alpha}e_\beta^\eps 
  &=\bigl(p^*\nabla^{TB}\bigr)_{e_\alpha}e_\beta^\eps 
	+\eps\,\sum_k t_{\alpha\beta k}\,e_k\;.
\end{align*}
We identify the tangent bundles~$(TM,g_\eps)$ orthogonally
for different~$\eps>0$ by sending~$(e_1, \dots, e_m)$ at~$\eps=1$
to the $g_\eps$-orthonormal frame of~\eqref{A3.0}.
With respect to this identification,
we obtain the limit connection
\begin{equation}\label{B1.3}
  \nabla^{TM,0}
  =\lim_{\eps\to 0}\nabla^{TM,\eps}
  =\nabla^{TX}\oplus p^*\nabla^{TB}\;.
\end{equation}
Note that this differs from the geometric limit of Levi-Civita connections
described for example in~\cite[section~10.1]{BGV}.

Let us assume for the moment that the base~$B$ and the map~$p$ are spin,
which is always true locally on~$M$,
and let~$SX\to M$ be a spinor bundle for~$TX\to M$.
Then we have an isomorphism of vector bundles
	$$SM\cong SX\otimes p^*SB$$
independent of~$\eps$,
and the connection~$\nabla^{SM,0}$ induced by~$\nabla^{TM,0}$ is the tensor
product connection.
To define Clifford multiplication~$c_I$ by~$e_I^\eps$ on~$SM$
for all~$I=1$, \dots, $m$,
the tensor product is understood in a $\Z_2$-graded sense.
The connections~$\nabla^{TM,\eps}$ induce connections~$\nabla^{SM,\eps}$
on the spinor bundle~$SM\to M$ for all~$\eps\ge 0$.
We have
\begin{align}
  \begin{split}
  \nabla^{SM,\eps}_{e_i}
  &=\nabla^{SM,0}_{e_i}
	+\frac\eps 2\,\sum_{j,\gamma}s_{ij\gamma}c_jc_\gamma
        -\frac{\eps^2}4\,\sum_{\alpha,\beta}t_{\alpha\beta i}\,c_\alpha c_\beta\;,\\
  \nabla^{SM,\eps}_{e_\alpha}
  &=\nabla^{SM,0}_{e_\alpha}-\frac\eps2\,\sum_{i,\beta}t_{\alpha\beta i}c_ic_\beta\;.
  \end{split}\label{B1.0}
\end{align}

Let~$(E,\nabla^{E,\eps},g^E,c)_{\eps>0}$ be an adiabatic family of Dirac bundles
on~$M$ as in Definition~\ref{A3.D1}.
We can now define a vertical and a horizontal Dirac operator by
\begin{equation}\label{B1.1}
  D_X=\sum_ic_i\nabla^{E,0}_{e_i}\;,
	\qquad\text{and}\qquad
  D_{B,\eps}=\frac1\eps\,\bigl(D_{M,\eps}-D_X\bigr)\;.
\end{equation}
Let~$\nabla^{E/S,\eps}-\nabla^{E/S,0}$ denote the one-form
of~\eqref{A3.5}.
Then
\begin{align}
  \begin{split}\label{B1.2}
  D_{B,\eps}&=\sum_\alpha c_\alpha\,\biggl(\nabla^{E,0}_{e_\alpha}
	+\frac12\,\sum_{i,j}s_{ij\alpha}\,c_ic_j
	-\frac\eps4\sum_{i,\beta}t_{\alpha\beta i}
		\,c_i c_\beta\biggr)\\
     &\qquad
	+\sum_ic_i\,\frac1\eps\,\bigl(\nabla^{E/S,\eps}-\nabla^{E/S,0}\bigr)_{e_i}
	+\sum_\alpha c_\alpha
		\,\bigl(\nabla^{E/S,\eps}-\nabla^{E/S,0}\bigr)_{e_\alpha}\\
     &=\sum_\alpha c_\alpha\,\biggl(\nabla^{E,0}_{e_\alpha}
	-\frac12\,h_\alpha\biggr)
	+\sum_ic_i\,\frac1\eps\,\bigl(\nabla^{E/S,\eps}-\nabla^{E/S,0}\bigr)_{e_i}\\
     &\qquad+\eps\,\sum_\alpha c_\alpha\biggl(
	\frac1\eps\,\bigl(\nabla^{E/S,\eps}-\nabla^{E/S,0}\bigr)_{e_\alpha}
	-\frac14\sum_{i,\beta}t_{\alpha\beta i}\,c_i c_\beta\biggr)
	\;,
  \end{split}
\end{align}
where~$h\in T^HM$ denotes the mean curvature vector of the fibres in~$(M,g)$,
and~$h_\alpha$ denotes its component in the direction of~$\alpha$.
Note that the connection~$\nabla^{E,0}-\frac12\,\<h,\punkt\>$ for~$\eps=0$
in the above expression for~$D_{B,\eps}$ is not unitary on~$E\to M$ in general,
but it induces a unitary connection~$\nabla^{p_*E,0}$
on the infinite dimensional vector orbibundle~$p_*E\to B$.

\begin{Lemma}\label{B1.L1}
  Let~$(E,\nabla^{E,\eps},g^E_\eps)_{\eps>0}$ be family of Dirac bundles
  on the family of Riemannian manifolds~$(M,g^{TM}_\eps)_{\eps>0}$.
  Decompose the associated family of Dirac
  operators~$D_{M,\eps}=D_X+\eps\,D_{B,\eps}$ as above.
  Then the anticommutator of~$D_X$ and~$D_{B,\eps}$
  is the sum of a fibrewise differential operator of order one
  and an endomorphism of~$E$.
\end{Lemma}

We write supercommutators as~$[\punkt,\punkt]$.

\begin{proof}
  Because~$D_X$ is of order one and involves only fibrewise differentiation,
  supercommutators of~$D_X$ with a zero order operator satisfy the assertion
  above. Hence, it suffices to consider
  \begin{multline*}
    \sum_\alpha\Bigl[D_X,\nabla^{E,0}_{e_\alpha}\Bigr]
    =\sum_{i,\alpha}\Bigl(c_ic\bigl(\nabla^{TM,0}_{e_i}e_\alpha\bigr)
		\,\nabla^{E,0}_{e_\alpha}\\
	+c_ic_\alpha\,(\nabla^{E,0})^2_{e_i,e_\alpha}
	+c_ic_\alpha\,\nabla^{E,0}_{[e_i,e_\alpha]}
        +c_\alpha c\bigl(\nabla^{TM,0}_{e_\alpha}e_i\bigr)\nabla^{E,0}_{e_i}\Bigr)\;.
  \end{multline*}
  Because~$e_\alpha$ is the horizontal lift of a vector field on~$B$,
  we have~$\nabla^{TM,0}_{e_i}e_\alpha=(p^*\nabla^{TB})_{e_i}e_\alpha=0$,
  and~$[e_i,e_\alpha]$ is a vertical vector field.
  Our claim follows.
\end{proof}

\subsection{The effective horizontal operator}\label{B1b}
We regard the infinite-dimensional bundle~$p_*E\to B$.
Together with the connection~$\nabla^{p_*E,0}$ of~\eqref{A3.2},
it becomes an infinite-dimensional Dirac orbibundle on~$B$.

Let~$P_X\in\End(p_*E)$ denote the fibrewise $L^2$-projection
on~$\ker D_X$.
By assumption~\eqref{A3.D1.3} in Definition~\eqref{A3.D1},
$H=\ker D_X=\im P_X$ is a finite rank vector bundle over~$B$.
Note that~$P_X$ does not necessarily commute
with the connection~$\nabla^{p_*E,0}$.
We define a connection~$\nabla^H$ on~$H$ by
\begin{equation*}
  \nabla^H=P_X\circ\nabla^{p_*E,0}\circ P_X
  =P_X\circ\biggl(\nabla^{SM,0}-\frac12\,\<h,\punkt\>\biggr)\circ P_X\;.
\end{equation*}

\begin{Proposition}\label{B2.P0}
  Let~$P_X$ and~$H$ be as above.
  \begin{enumerate}
  \item\label{B2.P0.1}
    The operator~$P_X$ is a fibrewise smoothing operator of finite rank
    that commutes with~$D_X$ and with Clifford multiplication
    with horizontal vectors.
  \item\label{B2.P0.2}
    The orbibundle~$H\to B$,
    equipped with the restriction of the fibrewise $L^2$-metric
    and the connection~$\nabla^H$,
    becomes a finite-dimensional Dirac orbibundle on~$B$.
  \end{enumerate}
\end{Proposition}

\begin{proof}
  The projection~$P_X$ commutes with~$P_X$ by construction,
  and with~$c_\alpha$ because~$D_X$ anticommutes with~$c_\alpha$.

  The connection~$\nabla^{p_*E,0}$ respects
  the $L^2$-scalar product,
  so its contraction~$\nabla^H$ onto~$H$ respects the induced scalar product.
  Because~$P_X$ commutes with~$c_\alpha$,
  we obtain a Dirac orbibundle.
\end{proof}

\begin{Remark}\label{B2.R1}
  By Definition~\ref{A3.D3},
  the effective horizontal operator is a Dirac operator if the family
  of local twist connections~$\nabla^{W,\eps}$ considered in the previous
  section is constant in~$\eps$.
  This is not the case for the odd signature operator~$B_{M,\eps}$
  on~$(M,g^{TM}_\eps)$,
  as explained in~\cite[section~4.1]{Dai}.
  The local twist bundle~$W$ is now given by~$(SM,\nabla^{SM,\eps})$.
  Hence, by equation~\eqref{B1.0},
  we have
	$$\frac d{d\eps}\Bigr|_{\eps=0}\nabla^{E/S,\eps}_{e_i}
	=\frac12\sum_{j,\gamma}s_{ij\gamma}c_jc_\gamma\;.$$
  This term only depends on the second fundamental form
  of the fibres, in particular, for totally geodesic fibrations
  the effective horizontal operator is in fact the Dirac
  operator on the Dirac bundle~$(H,g^H,\nabla^H)$
  of Proposition~\ref{B2.P0}~\eqref{B2.P0.2} above.
\end{Remark}

\begin{Proposition}\label{B2.P2}
  The $\eta$-invariant of~$D_B^{\mathrm{eff}}$ is given by a convergent integral,
  $$\eta\bigl(D_B^{\mathrm{eff}}\bigr)
    =\int_0^\infty\frac1{\sqrt{\pi t}}\,\trace
	\Bigl(D_B^{\mathrm{eff}}\,e^{-t(D_B^{\mathrm{eff}})^2}\Bigr)\,dt\;.$$
\end{Proposition}

\begin{proof}
  Convergence for~$t\to\infty$ is clear because we assumed that~$B$
  is compact, and hence~$D_B^{\mathrm{eff}}$ has discrete spectrum.

  For small-time convergence,
  we adapt the proof of~\cite[section~II]{BF}.
  We put
	$$A=D_B^{\mathrm{eff}}-D_B^H
	=\sum_iP_X\circ\biggl(c_i\,\frac d{d\eps}\Bigr|_{\eps=0}\nabla^{E/S,\eps}_{e_i}\biggr)\circ P_X\;.$$
  Because~$c_\alpha$ commutes with~$P_X$,
  we find that~$A$ anticommutes with Clifford multiplication,
	$$[c_\alpha,A]
	=P_X\circ\sum_i\biggl[c_\alpha,c_i\frac d{d\eps}\Bigr|_{\eps=0}\nabla^{E/S,\eps}_{e_i}\biggr]\circ P_X
        =0\;.$$
  In particular,
	$$(D_B^{\mathrm{eff}})^2
	=(D_B^H)^2+\sum_\alpha c_\alpha\bigl[\nabla^H_{f_\alpha},A\bigr]+A^2\;.$$

  We introduce an exterior variable~$z$ that anticommutes
  with the Clifford multiplication~$c_\alpha$
  and is parallel with respect to~$\nabla^H$.
  Consider the connection
  \begin{equation}\label{B2.0}
    \nabla^{H,z}=\nabla^H-\frac z2\,c(\punkt)
  \end{equation}
  on the Dirac bundle~$H$ of Proposition~\ref{B2.P0}.
  Then instead of the usual Bochner-Lichnerowicz-Weitzenb\"ock formula,
  one has
  \begin{multline}\label{B2.1}
    (D_B^{\mathrm{eff}})^2+z\,D_B^{\mathrm{eff}}
    =\nabla^{H,z,*}\nabla^{H,z}+\frac\kappa 4
	+\frac14\sum_{\alpha,\beta}c_\alpha c_\beta\,F^{H/S}_{f_{\alpha},f_{\beta}}\\
	+\sum_\alpha c_\alpha\bigl[\nabla^H_{f_\alpha},A\bigr]+A^2+zA\;.
  \end{multline}
  If~$P$, $Q$ are endomorphisms of a vector space,
  define~$\trace_z(P+zQ)=\trace(Q)$,
  then
  \begin{equation}\label{B2.2}
    \trace\Bigl(D_B^{\mathrm{eff}}\,e^{-t(D_B^{\mathrm{eff}})^2}\Bigr)
    =\frac1t\,\trace_z\Bigl(e^{-t((D_B^{\mathrm{eff}})^2-zD_B^{\mathrm{eff}})}\Bigr)\;.
  \end{equation}

  We want to compute the heat kernel of~$e^{-t((D_B^{\mathrm{eff}})^2-zD_B^{\mathrm{eff}})}$
  using Getzler rescaling.
  To see that this is possible,
  we have to distinguish two cases.
  
  If~$\dim B$ is even,
  only even elements of the Clifford algebra contribute to the trace.
  Hence, in the asymptotic expansion of the heat kernel,
  only terms involving the operator~$A$ an odd number of times will contribute.
  But~$A$ acts as~$\tau\otimes A'$,
  where~$\tau$ denotes the Clifford volume element and~$A'$ commutes with
  Clifford multiplication.
  Hence, we may replace $A$ formally by~$A'$ and the trace by a supertrace,
  so Getzler rescaling is appropriate.

  On the other hand, let~$\dim B$ be odd.
  Then~$\dim X$ is even, and the bundle~$H$ splits as~$H^+\oplus H^-$.
  The splitting is preserved by~$D_B^H$, but~$A$ exchanges the summands.
  Hence, in the asymptotic expansion of the heat kernel,
  only terms involving the operator~$A$ an even number of times will contribute,
  so we have to take the trace on the odd part of the Clifford algebra,
  and Getzler rescaling is again appropriate.

  Either way,
  we perform Getzler rescaling of the Clifford variables~$c_\alpha$,
  and~$A$ is not affected.
  Then the additional terms in the second line of~\eqref{B2.1}
  cause no trouble
  because~$A$ and~$\bigl[\nabla^H_{f_\alpha},A\bigr]$ do not involve
  Clifford multiplication at all.
  Hence, small time convergence follows as in~\cite{BF}.
\end{proof}

\begin{Proposition}\label{B2.P3}
  If~$\dim B$ is even, the effective horizontal operator
  of the adiabatic family of odd signature operators~$(B_{M,\eps})_\eps$
  on~$M$ has vanishing $\eta$-invariant.
\end{Proposition}

\begin{proof}
  The effective horizontal operator~$B_B^{\mathrm{eff}}$ acts on~$\Omega^\bullet(B;H)$
  and exchanges even and odd forms by~\cite[section~4.1]{Dai}.
  Thus, the odd heat kernel~$B_B^{\mathrm{eff}}\,e^{-t(B_B^{\mathrm{eff}})^2}$
  also exchanges even and odd forms, and hence, its trace is zero.
  Hence, the integrand in Proposition~\ref{B2.P2} vanishes.
\end{proof}

\subsection{The Dirac operator as a matrix}\label{B2}
The following sections are inspired by work of Bismut and
Lebeau~\cite[chapter~9]{BL} and Ma~\cite[chapter~5]{Ma}.
We will write operators acting on~$p_*E=\ker D_X\oplus\im D_X$
as matrices of the form
\begin{equation*}
  Y=
  \begin{pmatrix}
    P_X\,Y\,P_X&P_X\,Y\,(1-P_X)\\
    (1-P_X)\,Y\,P_X&(1-P_X)\,Y\,(1-P_X)
  \end{pmatrix}=
  \begin{pmatrix}Y_1&Y_2\\Y_3&Y_4\end{pmatrix}\;,
\end{equation*}
in particular
\begin{equation*}
  \frac1\eps\,D_{M,\eps}=\eps^{-1}
  \begin{pmatrix}
    D_{M,\eps,1}&D_{M,\eps,2}\\
    D_{M,\eps,3}&D_{M,\eps,4}
  \end{pmatrix}=
  \begin{pmatrix}
    D_{B,\eps,1}&D_{B,\eps,2}\\
    D_{B,\eps,3}&\eps^{-1}\,D_X+D_{B,\eps,4}
  \end{pmatrix}\;.
\end{equation*}

\begin{Proposition}\label{B2.P1}
As~$\eps\to 0$,
\begin{enumerate}
\item\label{B2.P1.1} the operator~$D_{B,\eps,1}-D_B^{\mathrm{eff}}$
  is an endomorphism of~$H\to B$ of magnitude~$O(\eps)$, and
\item\label{B2.P1.2} the operators~$D_{B,\eps,2}$ and~$D_{B,\eps,3}$
are uniformly bounded fibrewise smoothing operators of finite rank.
\end{enumerate}
\end{Proposition}

\begin{proof}
The first claim follows from the Definition~\ref{A3.D3}
of the effective horizontal operator and equation~\eqref{B1.2}.

The projection~$P_X$ is a fibrewise smoothing operator of finite rank.
It commutes with Clifford multiplication~$c_\alpha$ by horizontal vectors.
We conclude from~\eqref{B1.2} that the commutator~$[D_{B,\eps},P_X]$
is again a fibrewise smoothing operator of finite rank.
Now~\eqref{B2.P1.2} follows because
\begin{align*}
  D_{B,\eps,2}
  &=P_X\circ D_{B,\eps}\circ(1-P_X)=-[D_{B,\eps},P_X]\circ(1-P_X)\\
	\text{and}\qquad
  D_{B,\eps,3}
  &=(1-P_X)\circ D_{B,\eps}\circ P_X=(1-P_X)\circ[D_{B,\eps},P_X]
\end{align*}
are uniformly bounded fibrewise smoothing operators of finite rank.
\end{proof}

\subsection{A resolvent estimate}\label{B2a}
Let~$\lambda_B$ denote the smallest absolute value of
a nonzero eigenvalue of the effective horizontal operator~$D_B^{\mathrm{eff}}$,
and let~$0<c<\frac{\lambda_B}2$.
Let~$\Gamma=\Gamma_+\dotcup\Gamma_0\dotcup\Gamma_-$ denote a contour in~$\C$,
where~$\Gamma_\pm$ goes around~$\pm[\lambda_B,+\infty]$ with distance~$c$,
and~$\Gamma_0$ is a circle around~$0$ with radius~$c$.
We choose~$\eps_0>0$ such that Proposition~\ref{B3.P1} is satisfied
and such that all eigenvalues of~$\eps^{-1}D_{M,\eps}$
lie inside the area enclosed by~$\Gamma$ for all~$\eps>0$.
$$
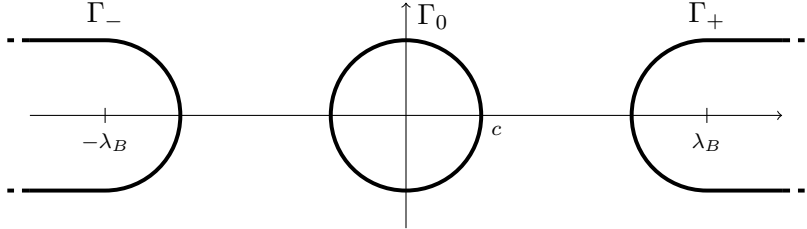
\begin{tikzpicture}
  \draw[->] (-5,0) -- (5,0) ; 
  \draw[->] (0,-1.5) -- (0,1.5) ; 
  \draw (-4,0.1) -- (-4,-0.1) node[below] {$\scriptstyle-\lambda_B$} ;
  \draw (4,0.1) -- (4,-0.1) node[below] {$\scriptstyle\lambda_B$} ;
  \node[below right] at (1,0) {$\scriptstyle c$} ;
  \begin{scope}[line width=1.5pt]
    \draw (0,1) arc (90:-270:1cm) node[above right] {$\Gamma_0$} ;
    \draw (-5,1) -- (-4,1) node[above] {$\Gamma_-$}
	arc (90:-90:1cm) -- (-5,-1) ;
    \draw (5,1) -- (4,1) node[above] {$\Gamma_+$}
	arc (90:270:1cm) -- (5,-1) ;
    \begin{scope}[dashed]
      \draw (-5,1) -- (-5.3,1) ;
      \draw (-5,-1) -- (-5.3,-1) ;
      \draw (5,1) -- (5.3,1) ;
      \draw (5,-1) -- (5.3,-1) ;
    \end{scope}
  \end{scope}
\end{tikzpicture}$$

For~$\lambda\notin\spec(D_{M,\eps,4})$, we consider the resolvent
	$$R_\eps(\lambda)=\frac{1-P_X}{\lambda-\eps^{-1}D_{M,\eps,4}}\;.$$
We regard the family of Schatten norms on operators acting on~$L^2(E)$,
given by
\begin{equation*}
  \norm A_p=\trace\Bigl((A^*A)^{\tfrac p2}\Bigr)^{\tfrac 1p}
\end{equation*}
for~$1\le p<\infty$, and let~$\norm A_\infty$ denote the operator norm.

\begin{Proposition}\label{B3.P1}
  There exist constants~$C$ and~$\eps_0>0$, such that
  for all~$p>\dim M$, all~$\eps\in(0,\eps_0)$
  and all~$\lambda\in\Gamma$,
  one has
  \begin{align*}
    \norm{R_\eps(\lambda)}_\infty
    &\le C\;,\tag{1}\label{B3.P1.1}\\
    \norm{R_\eps(\lambda)}_\infty
    &\le C\eps\,\abs\lambda\;,\tag{2}\label{B3.P1.2}\\
    \norm{R_\eps(\lambda)}_p
    &\le C\,\abs\lambda\;.\tag{3}\label{B3.P1.3}
  \end{align*}
\end{Proposition}

\begin{proof}
  For~$\sigma\in\im(1-P_X)$, we have
  \begin{multline}\label{B3.9}
    \norm{(i-\eps^{-1}D_{M,\eps,4})\sigma}_{L^2}^2\\
    =\bigl\<\bigl(1+\eps^{-2}D_X^2+\eps^{-1}[D_X,D_{B,\eps,4}]
	+D_{B,\eps,4}^2\bigr)\sigma,\sigma\bigr\>\;.
  \end{multline}
  The operator~$D_{B,\eps,4}^2$ is selfadjoint with nonnegative spectrum.

  The operator~$D_X^2|_{\im(1-P_X)}$ is a fibrewise differential operator
  of order~$2$,
  hence its spectrum is contained in~$[\lambda_0,\infty)$
  for some~$\lambda_0>0$.
  Let~$\Delta_X$ denote the fibrewise connection Laplacian acting on~$E\to M$,
  and let~$\mathcal R^X$ denote the curvature term in the classical
  Bochner-Lichnerowicz-Weitzenb"ock formula for~$D_X$.
  Write~$A\ge B$ if~$A-B$ is a nonnegative selfadjoint operator.
  Because~$D_X^2\ge\lambda_0^2>0$,
  we find a parameter~$s>0$ such that
  \begin{align}
    \begin{split}\label{B3.10}
      D_X^2-s\Delta_X
      &=(1-s)\,D_X^2+s\Bigl(\frac{\kappa_X}4+\mathcal R^X\Bigr)\\
      &\ge\frac{1-s}2\,D_X^2+\frac{(1-s)\lambda_0^2}2
	-s\norm{\frac{\kappa_X}4+\mathcal R^X}_\infty\\
      &\ge\frac{1-s}2\,D_X^2\ge\frac{(1-s)\lambda_0^2}2>0\;.
    \end{split}
  \end{align}

  By Lemma~\ref{B1.L1},
  the anticommutator
	$$[D_X,D_{B,\eps,4}]=(1-P_X)\,(D_XD_{B,\eps}+D_{B,\eps}D_X)\,(1-P_X)$$
  is the projection of a fibrewise differential operator of order~$1$.
  Write
  \begin{equation*}
    [D_X,D_{B,\eps}]
    =\sum_\nu a_\nu\,\nabla_{V_\nu}+b\;,
  \end{equation*}
  where the~$V_\nu$ are vertical vector fields
  and~$b$ and the~$a_\nu$ are endomorphisms of~$E\to M$ depending on~$\eps$.
  Note that~$V_\nu$, $a_\nu$ and~$b$ are uniformly bounded as~$\eps\to 0$.
  Because~$[D_X,D_{B,\eps}]$ is selfadjoint,
  \begin{equation*}
    \sum_\nu a_\nu\nabla_{V_\nu}
    =\Bigl(\sum_\nu a_\nu\nabla_{V_\nu}\Bigr)^*
    =-\sum_\nu\bigl(a_\nu^*\nabla_{V_\nu}+[\nabla_{V_\nu},a_\nu^*]
	+(\operatorname{div}V_\nu)\,a_\nu^*\bigr)\;.
  \end{equation*}
  Regard the nonnegative generalised fibrewise Laplace operator
  \begin{align*}
    0
    &\le
    s\biggl(\eps^{-1}\nabla+\frac1{2s}\sum_\nu\,\<V_\nu,\punkt\>a^*_\nu\biggr)^*
    \biggl(\eps^{-1}\nabla+\frac1{2s}\sum_\nu\,\<V_\nu,\punkt\>a^*_\nu\biggr)\\
    &=s\eps^{-2}\Delta_X
        +\frac1{4s}\sum_{\mu,\nu}\<V_\mu,V_\nu\>\,a_\mu a^*_\nu\\
    &\qquad
        -\frac1{2\eps}\,\sum_\nu\biggl((a^*_\nu-a_\nu)\,\nabla_{V_\nu}
	+(\operatorname{div}V_\nu)\,a^*_\nu
        +[\nabla_{V_\nu},a^*_\nu]\biggr)\\
    &=s\eps^{-2}\Delta_X+\eps^{-1}\bigl([D_X,D_{B,\eps}]-b\bigr)
        +\frac1{4s}\sum_{\mu,\nu}\<V_\mu,V_\nu\>\,a_\mu a^*_\nu\;.
  \end{align*}
  Because~$b$ acts as a fibrewise endomorphism on~$E\to M$,
  we conclude that
  \begin{equation}\label{B3.11}
    s\eps^{-2}\Delta_X+\eps^{-1}[D_X,D_{B,\eps}]
    \ge -\eps^{-1}C\;.
  \end{equation}
  A similar conclusion still holds if we replace~$D_{B,\eps}$
  by~$D_{B,\eps,4}$ because
	$$[D_X,D_{B,\eps}]-[D_X,D_{B,\eps,4}]
	=D_XD_{B,\eps,3}+D_{B,\eps,2}D_X$$
  is a fibrewise smoothing operator of finite rank
  by Proposition~\ref{B2.P1}~\eqref{B2.P1.2}.

  If we put~\eqref{B3.9}--\eqref{B3.11} together,
  we see that
  \begin{equation}\label{B3.12}
    1+\eps^{-2}D_{M,\eps,4}^2
    \ge1+\frac{1-s}{2\eps^2}\,D_X^2
	+\frac{\lambda_0^2}{4\eps^2}-\frac C\eps
	+D_{B,\eps,4}^2
  \end{equation}
  We immediately find that
  \begin{equation*}
    \norm{R_\eps(i)}_\infty\le C\,\eps\;.
  \end{equation*}

  Hence there exists~$\eps_0>0$ such that for all~$\eps\in(0,\eps_0)$,
  the spectrum of~$\eps^{-1}D_{M,\eps,4}$ is contained
  in~$\R\setminus\eps^{-1}(-c',c')$ for some constant~$c'>0$.
  The first estimate~\eqref{B3.P1.1} follows from our choice of~$\Gamma$.

  We obtain~\eqref{B3.P1.2} from~\eqref{B3.P1.1} and
  \begin{align*}
    \norm{R_\eps(\lambda)}
    &\le\norm{R_\eps(i)-R_\eps(i)\,(\lambda-i)\,R_\eps(\lambda)}\\
    &\le C\eps\,\bigl(1+\abs{\lambda-i}\,C\bigr)\;.
  \end{align*}

  Moreover, \eqref{B3.12} implies that there exists an~$\eps_0>0$
  such that for all~$\eps\in(0,\eps_0)$,
  the operator~$1+\bigl(\eps^{-1}D_X+D_{B,\eps,4}\bigr)^2$
  differs from a fixed selfadjoint second order elliptic operator
  by some selfadjoint operator with nonnegative eigenvalues.
  By the variational characterisation of eigenvalues
  and the definition of the $p$-norm,
  we conclude that
  \begin{equation*}
    \norm{R_\eps(i)}_p
    =\bigl\|\bigl(i-(\eps^{-1}D_X+D_{B,\eps,4})\bigr)^{-1}\bigr\|_p
    \le C\eps
  \end{equation*}
  for all~$\eps\in(0,\eps_0)$ and all~$p>\dim M$.
  By a similar argument as above, the proposition follows
  for all~$\lambda\in\Gamma$.
\end{proof}

In particular,
the resolvent~$R_\eps(\lambda)$
is uniformly bounded and of order~$O(\eps\abs\lambda)$
for all~$\lambda\in\Gamma$ as~$\eps\to 0$.
We write~$R_\eps(\lambda)=O(1,\eps\abs\lambda)$.
In particular,
we may extend this operator by~$0$ for~$\eps=0$.

\subsection{The Schur complement}\label{B2b}
To compute the full resolvent of~$\eps^{-1}D_{M,\eps}$,
we consider the Schur complement~$M_\eps(\lambda)$
of~$\lambda-\eps^{-1}D_{M,\eps,4}$
in the matrix representation of section~\ref{B2}.
The Schur complement is given by
\begin{equation*}
  M_\eps(\lambda)
  =\lambda-D_{B,\eps,1}-D_{B,\eps,2}\circ R_{\eps}(\lambda)\circ D_{B,\eps,3}\;.
\end{equation*}

\begin{Proposition}\label{B3.P0}
  There exists~$\eps_0>0$ small such that
  for all~$\eps\in(0,\eps_0)$ and all~$\lambda\in\Gamma$,
  the operator~$M_\eps(\lambda)$ is invertible.
  Moreover, there exists~$C>0$ such that for all~$p>\dim M$,
  \begin{align*}
    \norm{M_\eps(\lambda)^{-1}}_\infty
    &\le C\;,\tag{1}\label{B3.P0.1}\\
    \norm{M_\eps(\lambda)^{-1}-(\lambda-D_B^{\mathrm{eff}})^{-1}}_\infty
    &\le C\,\min\bigl(1,\eps\abs\lambda\bigr)\;,\tag{2}\label{B3.P0.2}\\
    \norm{M_\eps(\lambda)^{-1}}_p
    &\le C\,\abs\lambda\;\tag{3}\label{B3.P0.3}\\
    \norm{(\lambda-D_B^{\mathrm{eff}})^{-1}}_p
    &\le C\,\abs\lambda\;,\tag{4}\label{B3.P0.4}
  \end{align*}
\end{Proposition}

\begin{proof}
  By Propositions~\ref{B2.P1} and~\ref{B3.P1},
	$$M_\eps(\lambda)=\lambda-D_B^{\mathrm{eff}}+O(1,\eps\abs\lambda)\;.$$
  As~$D_{B,\eps,1}+D_{B,\eps,2}\circ R_{\eps}(\lambda)\circ D_{B,\eps,3}$ is a
  selfadjoint operator, its spectrum is contained in~$\R$,
  and~$M_\eps(\lambda)$ is invertible
  with~$\norm{M_\eps(\lambda)}\le\frac 1c$ for all~$\lambda\in\Gamma$
  with~$\Im\lambda=\pm ic$.

  The remaining~$\lambda\in\Gamma$ satisfy~$\abs\lambda\le 2\lambda_B$,
  so the remainder term~$M_\eps(\lambda)-\lambda-D_B^{\mathrm{eff}}$ is a bounded
  endomorphism of~$H$ with operator norm uniformly of order~$O(\eps)$.
  Then in particular,
  the series
  \begin{equation}\label{B2.10}
    M_\eps(\lambda)^{-1}
    =\frac1{\lambda-D_B^{\mathrm{eff}}}\sum_{k=0}^\infty\biggl(
	\bigl((\lambda-D_B^{\mathrm{eff}})-M_\eps(\lambda)\bigr)
		\,\frac1{\lambda-D_B^{\mathrm{eff}}}\biggr)^k
  \end{equation}
  converges if~$\eps>0$ is small enough.
  This proves invertibility of~$M_\eps(\lambda)$.
  Together with the above, we obtain~\eqref{B3.P0.1}.

  We deduce~\eqref{B3.P0.2} from~\eqref{B3.P0.1} and our choice of~$\Gamma$
  in section~\ref{B2a} because
  \begin{multline*}
    M_\eps(\lambda)^{-1}-(\lambda-D_B^{\mathrm{eff}})^{-1}\\
    =(\lambda-D_B^{\mathrm{eff}})^{-1}\,
	\bigl((\lambda-D_B^{\mathrm{eff}})-M_\eps(\lambda)\bigr)\,M_\eps(\lambda)^{-1}
    =O(1,\eps\abs\lambda)\;.
  \end{multline*}

For~\eqref{B3.P0.3},
we use that~$\norm{(i-D_B^{\mathrm{eff}})^{-1}}_p\le C$.
Moreover
\begin{align*}
  \norm{M_\eps(\lambda)^{-1}}_p
  &=\norm{(i-D_B^{\mathrm{eff}})^{-1}
	-(i-D_B^{\mathrm{eff}})^{-1}\,\bigl(M_\eps(\lambda)-(i-D_B^{\mathrm{eff}})\bigr)
		\,M_\eps(\lambda)^{-1}}_p\\
  &\le\norm{(i-D_B^{\mathrm{eff}})^{-1}}_p\\
  &\qquad
	+\norm{(i-D_B^{\mathrm{eff}})^{-1}}_p\,\norm{M_\eps(\lambda)-(i-D_B^{\mathrm{eff}})}_\infty
		\,\norm{M_\eps(\lambda)^{-1}}_\infty\\
  &\le C\,\bigl(1+(\abs{\lambda-i}+O(1,\eps\abs\lambda))\,C\bigr)\;.
\end{align*}
The last estimate~\eqref{B3.P0.4} is similar.
\end{proof}

We can now write the resolvent of~$\eps^{-1}D_{M,\eps}$ as
\begin{multline*}
  \frac1{\lambda-\eps^{-1}D_{M,\eps}}\\
  \begin{aligned}
    &=
    \begin{pmatrix}
      M_\eps(\lambda)^{-1}&
      M_\eps(\lambda)^{-1}\,D_{B,\eps,2}\,R_{\eps}(\lambda)\\
      R_{\eps}(\lambda)\,D_{B,\eps,3}\,M_\eps(\lambda)^{-1}&
      \;R_\eps(\lambda)+
      R_{\eps}(\lambda)\,D_{B,\eps,3}\,M_\eps(\lambda)^{-1}
	\,D_{B,\eps,2}\,R_{\eps}(\lambda)
    \end{pmatrix}\\
    &=\begin{pmatrix}\frac1{\lambda-D_B^{\mathrm{eff}}}&0\\0&R_\eps(\lambda)\end{pmatrix}
	+O(1,\eps\abs\lambda)\;.
  \end{aligned}
\end{multline*}
The remainder terms consist of
the resolvent of~$D_B^{\mathrm{eff}}$ and one or more of
the following finite-rank endomorphisms of~$p_*E$,
\begin{gather*}
  \bigl((\lambda-D_B^{\mathrm{eff}})-M_\eps(\lambda)\bigr)
	\,\frac1{\lambda-D_B^{\mathrm{eff}}}\;,\\
  D_{B,\eps,2}\,R_{\eps}(\lambda)\;,\\
  R_{\eps}(\lambda)\,D_{B,\eps,3}\;.
\end{gather*}
the behaviour of which is described in Propositions~\ref{B2.P1}--\ref{B3.P0}.
We summarize the results of this section.

\begin{Proposition}\label{B3.P2}
  There exist constants~$C$ and~$\eps_0>0$ such that
  for all~$p>\dim M$, all~$\eps\in(0,\eps_0)$,
  and all~$\lambda\in\Gamma$,
  one has
  \begin{gather*}
    \norm{(\lambda-\eps^{-1}D_{M,\eps})^{-1}}_\infty
    \le C\;,\qquad
    \norm{(\lambda-D_B^{\mathrm{eff}})^{-1}}_\infty
    \le C\;,\tag{1}\label{B3.P2.1}\\
    \norm{(\lambda-\eps^{-1}D_{M,\eps})^{-1}}_p
    \le C\,\abs\lambda\;,\qquad
    \norm{(\lambda-D_B^{\mathrm{eff}})^{-1}}_p
    \le C\,\abs\lambda\;,\tag{2}\label{B3.P2.2}\\
    \norm{(\lambda-\eps^{-1}D_{M,\eps})^{-1}-(\lambda-D_B^{\mathrm{eff}})^{-1}}_\infty
	\le C\eps\,\abs\lambda\;.\tag{3}\label{B3.P2.3}
  \end{gather*}
\end{Proposition}

In particular, the resolvent of~$\eps^{-1}D_{M,\eps}$ converges
to the resolvent of the effective horizontal operator~$D_B^{\mathrm{eff}}$
in a certain precise sense.

\subsection{Long time convergence}\label{B4}
Define a spectral projection~$P_\eps$ on~$\Gamma(p_*E)$ by
\begin{equation*}
  P_\eps=\frac1{2\pi i}\int_{\Gamma_0}\frac{dz}{z-\eps^{-1}D_{M,\eps}}\;.
\end{equation*}
Then~$P_\eps$ obviously commutes with~$D_{M,\eps}$.
By Proposition~\ref{B3.P2}~\eqref{B3.P2.3} and our choice of~$c$ and~$\Gamma_0$,
we find that~$P_0=\lim_{\eps\to 0}P_\eps$ is the projection onto the kernel
of the effective horizontal operator~$D_B^{\mathrm{eff}}$. 
In particular, $\im P_\eps$ is of constant finite dimension
for all~$\eps>0$ sufficiently small.

\begin{Proposition}\label{B4.P1}
  There exists~$\alpha>0$ such that
  \begin{equation*}
    \lim_{\eps\to 0}\int_{\eps^{\alpha-2}}^\infty\frac1{\sqrt{\pi t}}
	\trace\Bigl((1-P_\eps)\circ\bigl(D_{M,\eps}\,e^{-tD_{M,\eps}^2}\bigr)
		\circ(1-P_\eps)\Bigr)\,dt
    =\eta\bigl(D_B^{\mathrm{eff}}\bigr)\;.
  \end{equation*}
\end{Proposition}

\begin{proof}
  By Proposition~\ref{B2.P2},
  we may write
  $$\eta(D_B^{\mathrm{eff}})
    =\int_0^\infty\frac1{\sqrt{\pi t}}\,\trace
	\Bigl((1-P_0)\circ D_B^{\mathrm{eff}}\,e^{-t(D_B^{\mathrm{eff}})^2}\circ(1-P_0)\Bigr)\,dt\;,$$
  because~$P_0$ projects onto the kernel of~$D_B^{\mathrm{eff}}$.

  We rewrite the integral on the left hand side in the Proposition as
  \begin{multline*}
    \int_{\eps^{\alpha-2}}^\infty\frac1{\sqrt{\pi t}}
	\trace\Bigl((1-P_\eps)\circ\bigl(D_{M,\eps}\,e^{-tD_{M,\eps}^2}\bigr)
		\circ(1-P_\eps)\Bigr)\,dt\\
    =\int_{\eps^{\alpha}}^\infty\frac1{\sqrt{\pi t}}
	\trace\Bigl((1-P_\eps)
		\circ\bigl(\eps^{-1}D_{M,\eps}\,e^{-t\eps^{-2}D_{M,\eps}^2}\bigr)
		\circ(1-P_\eps)\Bigr)\,dt\;.
  \end{multline*}
  Using dominated convergence,
  we will show that this integral converges to~$\eta(D_B^{\mathrm{eff}})$
  as~$\eps\to 0$.

  For~$t>0$ and each integer~$k\ge 0$,
  we define two holomorphic functions~$F_{k,t}^+$, $F_{k,t}^-\colon\C\to\C$ with
  \begin{equation*}
    \frac{d^k}{dz^k}F^{\pm}_{k,t}(z)
    =z\,e^{-tz^2}
	\qquad\text{and}\qquad
    \lim_{z\to\pm\infty}F^\pm_{k,t}(z)=0\;.
  \end{equation*}
  Then obviously
  \begin{equation}\label{B3.13}
    F^{\pm}_{k,t}(z)=t^{-\tfrac{k+1}2}\,F^{\pm}_{k,1}\bigl(\sqrt t\,z\bigr)\;.
  \end{equation}

  By holomorphic functional calculus,
  \begin{multline}\label{B3.14}
    (1-P_\eps)\circ\Bigl(\eps^{-1}D_{M,\eps}\,e^{-t\eps^{-2}D_{M,\eps}^2}\Bigr)
	\circ(1-P_\eps)\\
    \begin{aligned}
      &=\frac1{2\pi i}\int_{\Gamma_+\dotcup\Gamma_-}
	\frac{z\,e^{-tz^2}}{z-\eps^{-1}\,D_{M,\eps}}\,dz\\
      &=\frac1{2\pi i\,k!}\int_{\Gamma_+}F^+_{k,t}(z)
	\,\bigl(z-\eps^{-1}\,D_{M,\eps}\bigr)^{-k-1}\,dz\\
      &\qquad+\frac1{2\pi i\,k!}\int_{\Gamma_-}F^-_{k,t}(z)
	\,\bigl(z-\eps^{-1}\,D_{M,\eps}\bigr)^{-k-1}\,dz\;.
    \end{aligned}
  \end{multline}
  A similar expression holds for
  \begin{equation*}
    D_B^{\mathrm{eff}}\,e^{-t(D_B^{\mathrm{eff}})^2}
    =(1-P_0)\circ\Bigl(D_B^{\mathrm{eff}}\,e^{-t(D_B^{\mathrm{eff}})^2}\Bigr)\circ(1-P_0)\;.
  \end{equation*}

  By the H\"older inequality, $\norm{X^p}_1\le\norm X_p^p$.
  We choose~$k>\dim M+1$.
  By Proposition~\ref{B3.P2}~\eqref{B3.P2.2} and~\eqref{B3.P2.3},
  there exist constants~$C$ varying from line to line such that
  \begin{multline}\label{B3.15}
    \biggl\|\frac1{2\pi i\,k!}\int_{\Gamma_\pm}F^\pm_{k,t}(z)
	\,\Bigl(\bigl(z-\eps^{-1}\,D_{M,\eps}\bigr)^{-k-1}
		-\bigl(z-D_B^{\mathrm{eff}}\bigr)^{-k-1}\Bigr)\,dz\biggr\|_1\\
    \begin{aligned}
      &\le C\,\int_{\Gamma_\pm}F^\pm_{k,t}(z)\sum_{j=0}^k
	\norm{\bigl(z-\eps^{-1}\,D_{M,\eps}\bigr)^{-1}}_k^j\\
      &\kern10em
        \cdot\norm{\bigl(z-\eps^{-1}\,D_{M,\eps}\bigr)^{-1}
		-\bigl(z-D_B^{\mathrm{eff}}\bigr)^{-1}}_\infty\\
      &\kern12em
        \cdot\norm{\bigl(z-D_B^{\mathrm{eff}}\bigr)^{-1}}_k^{k-j}\,dz\\
      &\le C\eps\,\int_{\Gamma_\pm}F^\pm_{k,t}(z)\,\abs z^{k+1}\,dz\;.
    \end{aligned}
  \end{multline}
  A similar estimate also holds for the integral over~$\Gamma_-$.
  Equation~\eqref{B3.15} clearly implies that
  \begin{multline}\label{B3.16}
    \lim_{\eps\to 0}\trace\Bigl((1-P_\eps)\circ
	\bigl(\eps^{-1}D_{M,\eps}\,e^{-t\eps^{-2}D_{M,\eps}^2}\bigr)
	\circ(1-P_\eps)\Bigr)\\
    =\trace\Bigl(D_B^{\mathrm{eff}}\,e^{-t(D_B^{\mathrm{eff}})^2}\Bigr)\;.
  \end{multline}

  Let~$\mu$ denote the arc length measure on~$\Gamma$.
  Using~\eqref{B3.13} and~\eqref{B3.15}, we estimate
  \begin{multline}\label{B3.18}
    \norm{(1-P_\eps)\circ
		\Bigl(\eps^{-1}D_{M,\eps}\,e^{-t\eps^{-2}D_{M,\eps}^2}\Bigr)
		\circ(1-P_\eps)
        -\Bigl(D_B^{\mathrm{eff}}\,e^{-t(D_B^{\mathrm{eff}})^2}\Bigr)}_1\\
    \begin{aligned}
      &\le C\eps\int_{\Gamma^\pm}\abs{F_{k,t}^\pm(z) z^{k+1}}\,d\mu(z)\\
      &\le C\eps t^{-k-1}\int_{\Gamma^\pm}
	\abs{F_{k,1}^\pm\bigl(\sqrt t\,z\bigr)
		\cdot\bigl(\sqrt t\,z\bigr)^{k+1}}\,d\mu(z)\\
      &\le C\eps t^{-k-\frac32}\int_{\sqrt t\Gamma^\pm}\abs{F_{k,1}^\pm(z)\, z^{k+1}}\,d\mu(z)
        \le C\eps t^{-k-\frac32}\,e^{-ct}\;.
    \end{aligned}
  \end{multline}

  Choose~$0<\alpha<\frac1{k+2}$.
  For~$\eps^\alpha\le t$, \eqref{B3.18} implies
  \begin{multline*}
    \frac1{\sqrt t}\,\norm{(1-P_\eps)\circ
		\Bigl(\eps^{-1}D_{M,\eps}\,e^{-t\eps^{-2}D_{M,\eps}^2}\Bigr)
		\circ(1-P_\eps)
        -\Bigl(D_B^{\mathrm{eff}}\,e^{-t(D_B^{\mathrm{eff}})^2}\Bigr)}_1\\
    \le Ct^{\tfrac1\alpha-k-2}e^{-ct}\;.
  \end{multline*}
  Because~$t$ occurs with positive exponent in the last line,
  the integral of the right hand side above over~$(0,\infty)$
  converges,
  and we may apply dominated convergence and~\eqref{B3.16}
  to complete the proof.
\end{proof}

\subsection{The very small eigenvalues}\label{B5}
We now want to estimate the contribution of the finite
dimensional vector space~$\im P_\eps$.
The operator~$P_\eps\circ\eps^{-1}D_{M,\eps}\circ P_\eps$ depends holomorphically
on~$\eps$,
so its eigenvalues are given by analytic functions~$\lambda_\nu$ in~$\eps$.
In particular,
we may choose~$\eps_0$ in section~\ref{B2} such that
\begin{equation*}
  \dim\ker\bigl(P_\eps\circ\eps^{-1}D_{M,\eps}\circ P_\eps\bigr)
  =\dim\ker D_{M,\eps}
\end{equation*}
is constant for all~$\eps\in(0,\eps_0]$.
By Proposition~\ref{B2.P1}~\eqref{B2.P1.1},
we have~$\lambda_\nu(\eps)=O(\eps)$,
and by the above,
the sign of~$\lambda_\nu(\eps)$ does not change on~$(0,\eps_0]$.

\begin{Proposition}\label{B5.P1}
  For~$0<\eps<\eps_0$, we have
  \begin{equation*}
    \lim_{\eps\to 0}\int_{\eps^{\alpha-2}}^\infty\frac1{\sqrt{\pi t}}
	\,\trace\Bigl(P_\eps\circ\bigl(D_{M,\eps}\,e^{-tD_{M,\eps}^2}\bigr)
		\circ P_\eps\Bigr)\,dt\\
    =\sum_{\nu=1}^{\dim\ker D_B^{\mathrm{eff}}}\sign(\lambda_\nu(\eps))\;.
  \end{equation*}
\end{Proposition}

\begin{proof}
  We have
  \begin{multline*}
    \int_{\eps^{\alpha-2}}^\infty\frac1{\sqrt{\pi t}}
	\,\trace\Bigl(P_\eps\circ\bigl(D_{M,\eps}\,e^{-tD_{M,\eps}^2}\bigr)
		\circ P_\eps\Bigr)\,dt\\
    \begin{aligned}
      &=\int_{\eps^{\alpha}}^\infty\frac1{\sqrt{\pi t}}
	\,\trace\Bigl(P_\eps
		\circ\bigl(\eps^{-1}D_{M,\eps}
			\,e^{-t\eps^{-2}D_{M,\eps}^2}\bigr)
		\circ P_\eps\Bigr)\,dt\\
      &=\sum_{\nu=1}^{\dim\ker D_B^{\mathrm{eff}}}
	\int_{\eps^\alpha}^\infty\frac{\lambda_\nu(\eps)}{\sqrt{\pi t}}
		\,e^{-t\lambda_\nu(\eps)^2}\,dt\\
      &=\sum_{\nu=1}^{\dim\ker D_B^{\mathrm{eff}}}\sign(\lambda_\nu(\eps))
	+O\Bigl(\eps^{\tfrac\alpha2}\Bigr)\;.\quad\qedhere
    \end{aligned}
  \end{multline*}
\end{proof}

\subsection{Short time convergence}\label{B6}
Let~$\alpha>0$ denote the constant introduced in Proposition~\ref{B4.P1}
and consider
\begin{equation*}
  \int_0^{\eps^{\alpha-2}}\frac1{\sqrt{\pi t}}\,
	\trace\Bigl(D_{M,\eps}\,e^{-tD_{M,\eps}^2}\Bigr)\,dt\;.
\end{equation*}
We treat the limit of this integral as~$\eps\to0$
as in~\cite{BC1} and~\cite{Dai}.
Over the singular strata of~$B$,
we get additional contributions involving equivariant $\eta$-forms,
see Definition~\ref{A3.D2} of the orbifold $\eta$-forms.

\begin{Proposition}\label{B6.P1} For~$\alpha>0$ sufficiently small,
  we have
  \begin{equation*}
    \lim_{\eps\to 0}\int_0^{\eps^{\alpha-2}}\frac1{\sqrt{\pi t}}
	\,\trace\Bigl(D_{M,\eps}\,e^{-tD_{M,\eps}^2}\Bigr)\,dt
    =\int_{\Lambda B}\hat A_{\Lambda B}\bigl(TB,\nabla^{TB}\bigr)
		\,2\eta_{\Lambda B}(\A)\;.
  \end{equation*}
\end{Proposition}

\begin{proof}
  We introduce an exterior variable~$z$ that anticommutes
  with the Clifford multiplication~$c$
  and is parallel with respect to~$\nabla^{E,\eps}$ for all~$\eps$.
  In analogy with~\eqref{B2.0}, consider the connection
  \begin{equation*}
    \nabla^{E,\eps,z}=\nabla^{E,\eps}-z\,c(\punkt)\;.
  \end{equation*}
  We will use the $g^{TM}_\eps$-orthonormal frame~$e^\eps_I$ of~\eqref{A3.0}.
  Then as in~\eqref{B2.1}, the Bochner-Lichnerowicz-Weitzenb\"ock formula
  implies
  \begin{equation*}
    D_{M,\eps}^2+2z\,D_{M,\eps}
    =\nabla^{E,\eps,z,*}\nabla^{E,\eps,z}+\frac\kappa 4
	+\frac14\sum_{I,J}c_Ic_J\,F^{E/S}_{e^\eps_{I},e^\eps_{J}}\;.
  \end{equation*}
  Define~$\trace_z$ as in the proof of Proposition~\ref{B2.P2},
  then as in~\eqref{B2.2},
  \begin{equation}\label{B6.0}
    \trace\Bigl(D_{M,\eps}\,e^{-tD_{M,\eps}^2}\Bigr)
    =\frac1t\,\trace_z\Bigl(e^{-t(D_{M,\eps}^2-zD_{M,\eps})}\Bigr)\;.
  \end{equation}

  From now on, we assume that~$t\le\eps^{\alpha-2}$
  for some small~$\alpha>0$.
  We fix~$q\in B$ and choose an orbifold
  chart~$\psi\colon\rho(\Gamma)\backslash V\to U\subset B$ with~$q=\psi(0)$
  and a local trivialisation~$\bar\psi\colon\Gamma\backslash(V\times X)
  \to p^{-1}(U)$ as in Definition~\ref{A1.D2}.
  We assume that~$\psi$ defines geodesic coordinates,
  and that~$\bar\psi$ is the trivialisation
  by horizontal lifts of radial geodesics.
  By parallel transport along these geodesics
  with respect to~$\nabla^{E,\eps}$,
  we also identify~$E|_{p^{-1}(V)}$ with~$E|_X\times V$.

  As explained in~\cite{Dai}, section~3.1,
  to compute the $z$-trace of the heat kernel over~$q$,
  we may assume that~$V=\R^{m-n}$, and that outside
  a suitably large compact subset, the metric on~$V$
  is flat and the geometry of the fibration is of product type.
  It is possible to perform all these modifications
  in a $\Gamma$-invariant way.

  Let~$v$ denote the $V$-coordinates of a point in~$V\times X$.
  As in~\cite{BC1}, we consider the operator
  \begin{equation*}
    H_{\eps,t}
    =\biggl(1+\frac{zc(v)}{2\eps\sqrt t}\biggr)
	\,\bigl(tD_{M,\eps}^2+2z\sqrt t\,D_{M,\eps}\bigr)
	\,\biggl(1-\frac{zc(v)}{2\eps\sqrt t}\biggr)\;.
  \end{equation*}
  In the trivialisations above, let
  \begin{equation*}
    \tilde k_{\eps,t}((v,x),(v',x'))\colon E_{x'}\to E_x
  \end{equation*}
  denote the heat kernel of the operator~$e^{-H_{\eps,t}}$ on~$V\times X$.
  The corresponding heat kernel~$k_{\eps,t}$ on~$\Gamma\backslash(V\times X)$
  then lifts to
  \begin{equation*}
    k_{\eps,t}([v,x],[v',x'])
    =\sum_{\gamma\in\Gamma}\tilde k_{\eps,t}((v,x),\gamma(v',x'))\circ\gamma
    \colon E_{x'}\to E_x\;.
  \end{equation*}
  Thus, we have
  \begin{equation*}
    \trace_z\bigl(k_{\eps,t}([v,x],[v,x])\bigr)
    =\sum_{\gamma\in\Gamma}\trace_z\bigl(\tilde k_{\eps,t}((v,x),\gamma(v,x))
	\circ\gamma\bigr)\;.
  \end{equation*}
  We will consider the contribution of each~$\gamma\in\Gamma$ over~$V$
  to the overall trace of~$e^{-t(D_{M,\eps}^2+z\,D_{M,\eps})}$ separately
  in the limit~$\eps\to0$.
  Moreover,
  \begin{multline}\label{B6.1}
    \int_{\Gamma\backslash(V\times X)}\trace_z\bigl(k_{\eps,t}([v,x],[v,x])\bigr)
	\,d(v,x)\\
    =\frac1{\#\Gamma}\int_{V\times X}
	\sum_{\gamma\in\Gamma}\trace_z\bigl(\tilde k_{\eps,t}((v,x),\gamma(v,x))
	\circ\gamma\bigr)\,d(v,x)
  \end{multline}
  because each point~$[v,x]\in\Gamma\backslash(V\times X)$
  has~$\#\Gamma$ different preimages in~$V\times X$.

  For a fixed~$\gamma\in\Gamma$,
  let~$V_\gamma\subset V$ denote the fixpoint set of~$\gamma$,
  which is a linear subspace of~$V$.
  Let~$N_\gamma$ denote its orthogonal complement.
  Because we have assumed that~$B$ is orientable, $\dim N_\gamma$ is even.
  Put
	$$m_\gamma=m-\dim N_\gamma\;.$$
  The action of~$\gamma$ on~$E|_X$ can be decomposed as
  \begin{equation}\label{B6.2}
    \gamma=\tilde\gamma^{E/SB}\circ\tilde\gamma^{SB}
  \end{equation}
  such that~$\gamma^{SB}$ is an element in the Clifford algebra of~$N_\gamma$
  and~$\gamma^{E/SB}$ commutes with Clifford multiplication with horizontal
  vectors,
  and this decomposition is unique up to sign.

  As~$\eps\to 0$,
  we will rescale~$v\in V$ by a factor~$\eps\sqrt t$.
  We will apply Getzler rescaling by~$\eps\sqrt t$
  only to Clifford multiplication with elements of~$V_\gamma$,
  whereas Clifford multiplication with elements of~$N_\gamma$ and~$TX$
  will not be rescaled.
  Let us denote the complete rescaling by~$G_{\gamma,\eps}$.
  In particular, the action of~$\gamma$ commutes with~$G_{\gamma,\eps}$.

  We choose the basis in section~\ref{B1} such that~$f_{n+1}$,
  \dots, $f_{m_\gamma}$ are tangent to~$V_\gamma$.
  Let~$\eps^\alpha$ denote exterior multiplication with~$dv^\alpha$.
  For~$I\in\{1,\dots,m\}$, define
  \begin{equation*}
    \mu_I=
    \begin{cases}
      c_I&\text{if~$1\le I\le n$ or~$m_\gamma<I\le m$, and}\\
      t^{-\tfrac12}\,\eps^I&\text{if~$n<I\le m_\gamma$.}
    \end{cases}
  \end{equation*}
  Bismut's Levi-Civita superconnection can be defined as the operator
  \begin{equation*}
    \mathbb A_t=\sqrt t\,D_X+\nabla^{p_*E,0}
	-\frac{\sqrt t}4\sum_{i\alpha\beta}t_{\alpha\beta i}
		\,\mu_i\mu_\alpha\mu_\beta\;.
  \end{equation*}
  Then as in~\cite{BC1},
  we can compute the limit of the rescaled operator~$H_{\eps,t}$ as
  \begin{align}
    \begin{split}\label{B6.3}
    \lim_{\eps\to 0}G_{\gamma,\eps}(H_{\eps,t})
    &=-t\biggl(\nabla_{e_i}
	+\frac14\sum_{J,K=1}^{m_\gamma}s_{iJK}\mu_I\mu_J-t^{-\tfrac12}zc_i\biggr)^2\\
    &\qquad
        -\biggl(\frac{\del}{\del_\alpha}
	+\frac18\,\bigl\<R^B|_{V_\gamma}e_\alpha,v\bigr\>\biggr)^2\\
    &\qquad
        +\sum_{I,J=1}^{m_\gamma}t\,R^{E/S,0}_{e_I,e_J}\mu_I\mu_J+t\frac{\kappa_X}4\\
    &=\biggl(\mathbb A_t^2+tz\,\frac{d\mathbb A_t}{dt}\biggr)
	-\biggl(\frac{\del}{\del_\alpha}
		+\frac18\,\bigl\<R^B|_{V_\gamma}e_\alpha,v\bigr\>\biggr)^2\;.
    \end{split}
  \end{align}
  Both operators on the right hand side have coefficients
  in~$\Lambda^\bullet(V^\gamma)^*$.
  The operator~$\mathbb A_t^2+tz\,\tfrac{d\mathbb A_t}{dt}$
  acts on~$\Gamma(E\to X)$ and commutes with Clifford multiplication
  by horizontal vectors,
  while~$\bigl(\tfrac{\del}{\del_\alpha}+\tfrac18\,\bigl\<R^B|_{V_\gamma}e_\alpha,v\bigr\>\bigr)^2$
  acts on~$\Omega^\bullet(V)$.

  Let~$SB$ be a local spinor bundle on~$V$,
  then there exists a fibrewise Dirac bundle~$W\to M$ as in~\eqref{A3.4}.
  We continue as in~\cite{BC1},
  using the heat kernel proof of the equivariant index theorem
  in order to conclude that on~$V\times X$,
  \begin{multline*}
    \lim_{\eps\to0}\int_{V\times X}
	\trace_z\bigl(\tilde k_{\eps,t}((v,x),\gamma(v,x))\circ\gamma\bigr)
	d(v,x)\\
    \begin{aligned}
      &=\int_V\trace_{SB}\bigl(k_V(v,\gamma v)\circ\tilde\gamma^{SB}\bigr)
	\,(2\pi i)^{-\tfrac{N^{V^\gamma}}2}
	\,\trace_{p_*W}\biggl(2t\,\frac{d\mathbb A_t}{dt}\,e^{\mathbb A_t^2}\,\tilde\gamma^{E/SB}\biggr)\,dv\\
      &=\int_{V^\gamma}\hat A_{\tilde\gamma^{SB}}\bigl(TV,\nabla^{TV}\bigr)
	\,(2\pi i)^{-\tfrac{N^{V^\gamma}}2}
	\,\trace_{p_*W}\biggl(2t\,\frac{d\mathbb A_t}{dt}\,e^{\mathbb A_t^2}\,\tilde\gamma^{E/SB}\biggr)\;.
    \end{aligned}
  \end{multline*}

  From~\eqref{B6.1} and the above, we obtain
  \begin{multline*}
    \lim_{\eps\to 0}\int_{\Gamma\backslash(V\times X)}\trace_z\bigl(k_{\eps,t}([v,x],[v,x])\bigr)
	\,d(v,x)\\
    \begin{aligned}
      &=\frac1{\#\Gamma}\sum_{\gamma\in\Gamma}
	\int_{V^\gamma}\hat A_{\tilde\gamma^{SB}}\bigl(TV,\nabla^{TV}\bigr)
		\,(2\pi i)^{-\tfrac{N^{V^\gamma}}2}
		\,\trace_{p_*W}\biggl(2t\,\frac{d\mathbb A_t}{dt}
			\,e^{\mathbb A_t^2}\,\tilde\gamma^{E/SB}\biggr)\\
      &=\sum_{(\gamma)}\frac1{\# C_\Gamma(\gamma)}
	\int_{V^\gamma}\hat A_{\tilde\gamma^{SB}}\bigl(TV,\nabla^{TV}\bigr)
		\,(2\pi i)^{-\tfrac{N^{V^\gamma}}2}
		\,\trace_{p_*W}\biggl(2t\,\frac{d\mathbb A_t}{dt}
			\,e^{\mathbb A_t^2}\,\tilde\gamma^{E/SB}\biggr)\\
      &=\sum_{(\gamma)}
	\int_{C_\Gamma(\gamma)\backslash V^\gamma}\psi_{(\gamma)}^*
		\hat A_{\Lambda B}\bigl(TB,\nabla^{TB}\bigr)
		\,(2\pi i)^{-\tfrac{N^{V^\gamma}}2}
		\,\trace_{p_*W}\biggl(2t\,\frac{d\mathbb A_t}{dt}
			\,e^{\mathbb A_t^2}\,\tilde\gamma^{E/SB}\biggr)\\
    \end{aligned}
  \end{multline*}
  in analogy with the index computations in~\cite{Kawa}.
  By~\eqref{B6.0} and the above, we have the global formula
  \begin{multline*}
    \lim_{\eps\to 0}\trace\Bigl(\sqrt t\,D_{M,\eps}\,e^{-tD_{M,\eps}^2}\Bigr)\\
    =\int_{\Lambda B}\hat A_{\Lambda B}\bigl(TB,\nabla^{TB}\bigr)
	\,2
	\,(2\pi i)^{-\tfrac{N^{V^\gamma}}2}
        \,\trace_{p_*W}\biggl(\frac{d\mathbb A_t}{dt}
		\,e^{\mathbb A_t^2}\,\tilde\gamma^{E/SB}\biggr)\;.
  \end{multline*}
  By Theorem~3.1 of~\cite{Dai},
  we have uniform convergence as~$\eps\to 0$.
  By~\eqref{A3.3} and Definition~\ref{A3.D2},
  \begin{align*}
    \lim_{\eps\to0}\int_0^{\eps^{\alpha-2}}\frac1{\sqrt{\pi t}}
	\trace\Bigl(D_{M,\eps}e^{-tD_{M,\eps}^2}\,\gamma\Bigr)\,dt
    &=\int_{\Lambda B}\hat A_{\Lambda B}\bigl(TB,\nabla^{TB}\bigr)
	\,2\eta_{\Lambda B}(\mathbb A)\;.\qedhere
  \end{align*}
\end{proof}

\begin{Remark}\label{B6.R1}
  We replace the vector bundle~$E$ by~$E\otimes p^*W$,
  where~$W\to B$ is a vector orbibundle.
  We also assume that the twist connection~$\nabla^{(E\otimes p^*W)/S,0}$
  splits as the tensor product connection of~$\nabla^{E/S,0}$ and~$p^*\nabla^W$
  in the limit~$\eps\to 0$.
  A relevant special case is the case of the signature operator on~$M$,
  where the (local) spinor bundle of~$B$ plays the role of~$W$.
  In this case,
  equation~\eqref{B6.3} becomes
	$$\lim_{\eps\to 0}G_{\gamma,\eps}(H_{\eps,t})
	=\biggl(\A_t^2+tz\,\frac{d\A_t}{dt}\biggr)^2
		-\biggl(\frac{\del}{\del_\alpha}
			+\frac18\,\<R^B|_{V_\gamma}e_\alpha,v\>\biggr)^2
		+p^*R^W\;.$$
  This implies that now,
  \begin{multline*}
    \lim_{\eps\to 0}\int_0^{\eps^{\alpha-2}}\frac1{\sqrt{\pi t}}
	\trace\Bigl(D_{M,\eps}e^{-tD_{M,\eps}^2}\,\gamma\Bigr)\,dt\\
    =\int_{\Lambda B}\hat A_{\Lambda B}\bigl(TB,\nabla^{TB}\bigr)
	\,2\eta_{\Lambda B}(\mathbb A)\,\ch_{\Lambda B}(W,\nabla^W)\;.
  \end{multline*}
\end{Remark}

\section{The Spaces \texorpdfstring{$P_k$}{Pk}
and \texorpdfstring{$M_{(p_-,q_-),(p_+,q_+)}$}{M(p-,q-)(p+,q+)}}
\label{C}
We consider the family~$M_{(p_-,q_-),(p_+,q_+)}$ of manifolds
with a cohomogeneity one action of~$G=\Sp(1)\times\Sp(1)$ 
that are described in~\cite[chapter~13]{GWZ}.
This family contains the spaces~$P_k=M_{(1,1),(2k-1,2k+1)}$ as well as the Berger
space~$\SO(5)/\SO(3)=M_{(3,1),(1,3)}$.
The manifolds~$M_{(p_-,q_-),(p_+,q_+)}$ are two-connected with finite cyclic
third homotopy group, so by~\cite{C} and~\cite{CG},
it suffices to compute the Eells-Kuiper invariants and the modified
Kreck-Stolz invariants for quaternionic line bundles of~\cite{CG} to determine
the diffeomorphism type.

\subsection{Construction as Manifolds of Cohomogeneity One}\label{C1}

Let~$(p_+,q_+)$ and~$(p_-,q_-)$ be two pairs of relative prime
positive odd integers. We regard the subgroup
\begin{multline}\label{C1.1}
  H=\Bigl\{\pm (1,1),
	\pm \Bigl(i,(-1)^{\tfrac{q_--p_-}2}i\Bigr),
	\pm \Bigl(j,(-1)^{\tfrac{q_+-p_+}2}j\Bigr),\\
	\pm \Bigl(k,(-1)^{\tfrac{q_-+q_+-p_--p_+}2}k\Bigr)
	\Bigr\}\subset G=\Sp(1)\times\Sp(1)\;,
\end{multline}
which is isomorphic (in fact conjugate) to the diagonal subgroup~$\Delta Q$,
with
\begin{equation*}
  Q=\{\pm1,\pm i,\pm j,\pm k\}\subset\Sp(1)\;.
\end{equation*}

If~$a\in S^2\subset\H$ is an imaginary unit quaternion
and~$p$, $q$ are relative prime odd integers as above,
we consider the subgroup
\begin{equation*}
  C^a_{(p,q)}=\bigl\{\,(e^{ap\thet},e^{aq\thet})\bigm|\thet\in\R\,\bigr\}
  \subset G=\Sp(1)\times\Sp(1)\;,
\end{equation*}
which is isomorphic to~$S^1$.
For an odd integer~$2l+1$,
we have~$e^{\frac{a(2l+1)\pi}2}=(-1)^la$.
This implies that
\begin{equation*}
  \Bigl\{\pm \Bigl(a,(-1)^{\tfrac{p-q}2}a\Bigr)\;,\;\pm(1,1)\Bigr\}
  \subset C^a_{(p,q)}\;.
\end{equation*}

We put
\begin{equation}\label{C1.2}
  K_-=C^i_{(p_-,q_-)}\cdot H\qquad\text{and}\qquad
  K_+=C^j_{(p_+,q_+)}\cdot H\quad\subset\quad G\;.
\end{equation}
Then in particular~$H=K_-\cap K_+$,
and we have isomorphisms
\begin{align*}
  K_-&=C^i_{(p_-,q_-)}\cup
	\Bigl(j,(-1)^{\tfrac{q_+-p_+}2}j\Bigr)\,C^i_{(p_-,q_-)}
  \cong\Pin(2)\\\text{and}\qquad
  K_+&=C^j_{(p_+,q_+)}\cup
	\Bigl(i,(-1)^{\tfrac{q_--p_-}2}i\Bigr)\,C^j_{(p_+,q_+)}
  \cong\Pin(2)\;.
\end{align*}
The actions of~$K_\pm$ on~$S^1\cong K_+/H\cong K_-/H$ are $\R$-linear.

We now consider the cohomogeneity one manifolds~$M_{(p_-,q_-),(p_+,q_+)}$
with group diagram
\begin{equation}\label{C1.3}
  \begin{CD}
    &&G\\
    &\;\;\nearrow\;\;&&\;\;\nwarrow\;\;\\
    K_-&\cong&\!\Pin(2)\!&\cong&K_+\quad.\\
    &\;\;\nwarrow\;\;&&\;\;\nearrow\;\;\\
    &&H
  \end{CD}
\end{equation}
Thus,
the generic $G$-orbit takes the form~$G/H\cong S^3\times\R P^2/(\Z/2\Z)^2$,
and the two singular orbits are of the form~$M_\pm=G/K_\pm$.
We will study the geometry of~$M_{(p_-,q_-),(p_+,q_+)}$ in section~\ref{M}.

\begin{Theorem}[\cite{GWZ}, Theorem~13.1]
  The manifolds~$M=M_{(p_-,q_-),(p_+,q_+)}$ are two-connected.
  If~$p_-q_+=\pm p_+q_-$, then~$H^3(M)=H^4(M)=\Z$,
  otherwise~$H^3(M)=0$ and~$H^4(M)=\Z/k\Z$
  with~$k=\frac{p_-^2q_+^2-p_+^2q_-^2}8$.
\end{Theorem}

\subsection{The \texorpdfstring{$t$}{t}-invariant}\label{C2a}
In this section,
we want to determine the homeomorphism type of the spaces~$P_k$.

In~\cite{C}, Crowley has constructed a quadratic form~$q_M\colon H^4(M)\to\Q/\Z$
for all two-connected closed topological seven-manifolds with finite~$H^4(M)$
satisfieing
\begin{gather*}
  \lk_M(a,b)
  =q_M(a+b)-q_M(a)-q_M(b)\\
  \text{and}\qquad
  \lk_M\Bigl(a,\frac{p_1}2(TM)\Bigr)=q_M(a)-q_M(-a)
\end{gather*}
for all~$a$, $b\in H^4(M)$,
where~$\frac{p_1}2$ denotes the natural refinement of
the first Pontrijagin class~$p_1$ for spin manifolds.
Note that two quadratic forms with the properties above differ
by the pairing with an element of~$H_4(M;\Z/2\Z)$.
Crowley has then proved that two such manifolds~$M_0$, $M_1$
are homeomorphic (in fact almost diffeomorphic) if and only
if~$\bigl(H^4(M_0),q_{M_0}\bigr)$ and~$\bigl(H^4(M_1),q_{M_1}\bigr)$
are isomorphic.

In analogy with the Kreck-Stolz invariants~$s_2$ and~$s_3$ of~\cite{KS},
Crowley and the author have defined an invariant~$t_M(E)\in\Q/\Z$
for a two-connected smooth closed seven-manifolds~$M$
and a quaternionic line bundle~$E\to M$,
such that
	$$q_M\bigl(c_2(E)\bigr)=12\,t_M(E)\;.$$
For each cohomology class~$a\in H^4(M)$ of such a manifold~$M$,
there exist quaternionic line bundles~$E\to M$ with~$c_2(E)$.
We will thus compute~$t_M(E)$ for sufficiently many quaternionic
line bundles~$E\to M$ in order to determine the diffeomorphism type
of the spaces~$M=P_k$.

Let us recall the intrinsic definition of~$t_M(E)$ in~\cite{CG}.
We assume that~$E$ carries a quaternionic Hermitian metric~$g^E$
and a quaternionic Hermitian connection~$\nabla^E$.
Then there is a natural representative~$c_2(E,\nabla^E)\in\Omega^4(M)$
of the class~$c_2(E)$.
If~$H^3_{\mathrm{dR}}(M)=H^4_{\mathrm{dR}}(M)=0$,
there exists a differential form~$\hat c_2(E,\nabla^E)\in\Omega^3(M)$
such that
	$$d\hat c_2\bigl(E,\nabla^E\bigr)
	=c_2\bigl(E,\nabla^E\bigr)\in\Omega^4(M)\;,$$
and~$\hat c_2(E,\nabla^E)$ is unique up to an exact form.
Let~$D$ and~$D^E$ denote the untwisted Dirac operator on~$M$ and the Dirac
operator twisted with~$(E,g^E,\nabla^E)$,
and let~$h(D)=\dim\ker D$.

\begin{Definition}[\cite{CG}]\label{C2.D2}
  For a quaternionic line bundle~$E\to M$ on a compact oriented
  seven-dimensional spin manifold~$M$ with~$H^4_{\mathrm{dR}}(M)=0$,
  put
  \begin{multline*}
    t_M(E)=\frac{\eta+h}4(D^E_M)-\frac{\eta+h}2(D_M)\\
    -\frac1{24}\int_M\Bigl(\frac{p_1}2\bigl(TM,\nabla^{TM}\bigr)
	+c_2\bigl(E,\nabla^E\bigr)\Bigr)\,\hat c_2\bigl(E,\nabla^E\bigr)
    \quad\in\quad\Q/\Z\;.
  \end{multline*}
\end{Definition}

\begin{Theorem}\label{C2.T2}
  Assume that~$p_-$ and~$p_+$ are relatively prime.
  Then there exists an isomorphism~$H^4(M_{(p_-,q_-),(p_+,q_+)})\cong\Z/k\Z$
  such that Crowley's quadratic form~$q_M(\ell)$ for~$\ell\in\Z/k\Z$
  is given by
  \begin{equation*}
    q_M(\ell) 
    =\ell\,\frac{p_+^2-p_-^2+\ell\,p_-^2p_+^2}{2k}
	-\frac \ell2\quad\in\Q/\Z\;.
  \end{equation*}
\end{Theorem}

This theorem will be proved in section~\ref{M11}.

\begin{Remark}\label{C2.R3}
  More generally,
  suppose that~$a=(p_-^2,p_+^2)$ and~$b=(q_-^2,q_+^2)$ are the greatest common
  divisors.
  Because~$p_-$ and~$q_-$ are relatively prime, so are~$a$ and~$b$.
  Moreover,
  clearly~$a|k$ and~$b|k$.

  The proof of Theorem~\ref{C2.T2} gives a formula for~$q_M(\ell)$
  if~$a|\ell\in H^4(M)$ by identifying the class~$\ell$ with a class pulled back
  from the base of the Seifert fibration~$p\colon M\to B$
  considered in Proposition~\ref{M1.P1}.
  Swapping the roles of the $p$s and~$q$s gives an analogous formula
  for~$q_M(\ell)$ if~$b|\ell\in H^4(M)$.

  To see that these two formulas determine~$q_M$ uniquely,
  for each~$\ell\in H^4(M)\cong\Z/n\Z$ we find~$x$, $y\in\Z$ such that
	$$\ell=xa^2+yb^2\;.$$
  Because~$q$ refines the linking form,
  we have
  \begin{align*}
    q_M(\ell)&=q_M(xa^2+yb^2)=q_M(xa^2)+q_M(yb^2)+a^2b^2\lk(x,y)\\
    &=q_M(xa^2)+q_M(yb^2)+q_M(ab(x+y))-q_M(abx)-q_M(aby)\;,
  \end{align*}
  and each of the terms on the right hand side is computable.
  The main difficulty consists in determining the respective classes
  in the two base orbifolds.
\end{Remark}

\begin{Example}\label{M11.X1}
  We consider the special case~$P_k=M_{(1,1),(2k-1,2k+1)}$ and obtain
  \begin{equation}
    q_M(\ell)=\ell\,\frac{4k(k-1)+\ell\,(2k-1)^2}{2k}-\frac\ell 2
	=\frac{\ell(\ell-k)}{2k}\quad\mod\Z\;.
  \end{equation}\label{M11.3}
  We compute
  \begin{equation}\label{M11.4}
    \lk(i,j)
    =q_M(i+j)-q_M(i)-q_M(j)
    =\frac{ij}k
  \end{equation}
  which proves that the linking form on~$H^4(P_k)$ is standard
  and that the class represented by~$\ell=1$ is a generator.
  We also see that~$\frac{p_1}2(TP_k)=0$
  because
  \begin{equation}\label{M11.5}
    \lk\Bigl(\frac{p_1}2(TM),\ell\Bigr)
    =q_M(-\ell)-q_M(\ell)
    =\ell=0\quad\mod\Z\;.
  \end{equation}
\end{Example}

\subsection{The Eells-Kuiper invariant}\label{C2}
The Eells-Kuiper invariant~$\mu$ has first been defined in~\cite{EK}
for certain manifolds using zero bordisms.
It distinguishes all exotic spheres in dimension~$7$.
Crowley has shown in~\cite{C} that two homeomorphic two-connected closed smooth
seven-manifolds with finite~$H^4$ are diffeomorphic if and only if their Eells-Kuiper invariants
agree.

We will use the intrinsic description of~$\mu(M)$ by Donnelly~\cite{Do1}
and Kreck and Stolz~\cite{KS}.
Let~$M$ be an oriented spin Riemannian seven-manifold
with~$H^3_{\mathrm{dR}}(M)=H^4_{\mathrm{dR}}(M)=0$,
and let~$D_M$ denote he untwisted Dirac operator on~$M$.
Let~$B_M$ denote the odd signature operator, acting on~$\Omega^\even M$.
Let~$p_1(TM,\nabla^{TM})$ denote the first Pontrijagin form of~$M$,
then there exists a form~$\hat p_1(TM,\nabla^{TM})\in\Omega^3(M)$ such that
	$$d\hat p_1\bigl(TM,\nabla^{TM}\bigr)=p_1\bigl(TM,\nabla^{TM}\bigr)\;,$$
and~$\hat p_1(TM,\nabla^{TM})$ is uniquely determined up to an exact form.
Following~\cite{KS},
the Eells-Kuiper invariant of~$M$ can be computed as
\begin{multline}\label{C2.1}
  \mu(M)=\frac{\eta+h}2(D_M)+\frac\eta{2^5\cdot 7}(B_M)\\
	-\frac1{2^7\cdot 7}\int_M(p_1\wedge\hat p_1)\bigl(TM,\nabla^{TM}\bigr)
  \quad\in\Q/\Z\;.
\end{multline}

We will use Theorem~\ref{Thm1} to compute the $\eta$-invariants
in equation~\eqref{C2.1}.
Again, we make use of the Seifert fibration~$M\to B$
discussed in Proposition~\ref{M1.P1} below.
In analogy with the classical Dedekind sums occurring in the study of
quadratic forms,
we consider a particular family of sums
over rational functions in sines and cosines.
These sums represent the contribution of the twisted sectors of~$B$
to~$\mu(M_{(p_-,q_-),(p_+,q_+)})$.

\begin{Definition}\label{C2.D1}
  If~$p$, $q\in\mathbb N$ are odd and relatively prime,
  define the {\em generalised Dedekind sums\/}
  \begin{equation*}
    D(p,q)
    =\sum_{a=1}^{p-1}\Biggl(
	\frac{14\cos\frac{4\pi a}{p}+\cos^2\frac{q\pi a}{p}}
	{2^4\cdot 7\,p^2\sin^2\frac{4\pi a}{p}\sin^2\frac{q\pi a}{p}}
	+\frac{q\cos\frac{q\pi a}{p}
		\bigl(14+\cos\frac{4\pi a}{p}\bigr)}
	{2^5\cdot 7\,p^2\sin\frac{4\pi a}{p}\sin^3\frac{q\pi a}{p}}
	\Biggr)\;.
  \end{equation*}
\end{Definition}


We will give explicit formulas for some of these sums in the next subsection.

\begin{Theorem}\label{C2.T1}
  We have
  \begin{multline*}
    \ek\bigl(M_{(p_-,q_-),(p_+,q_+)}\bigr)
    =\frac{\sign(q_-^2p_+^2-q_+^2p_-^2)}{2^5\cdot 7}
	+D(p_-,q_-)-D(p_+,q_+)\\
	-\frac{(p_+^2-p_-^2)^2}
		{2^2\cdot 7\,p_-^2p_+^2(q_-^2p_+^2-q_+^2p_-^2)}
	-\frac{2^4(p_+^2-p_-^2)+(q_-^2p_+^2-q_+^2p_-^2)}
		{2^8\cdot 7\,p_-^2p_+^2}\;.
  \end{multline*}
\end{Theorem}

This theorem will be proved in section~\ref{M8}.

\begin{Corollary}\label{C2.C1}
  Assume that~$p$ and~$q$ are odd and relatively prime.
  Then there is a duality of generalised Dedekind sums
  \begin{equation*}
    D(p,q)+D(q,p)
    -\frac{2^6+2^4(p^2+q^2)+(p^4+q^4)}{2^8\cdot 7\,p^2q^2}
	+\frac7{2^7}\quad\in\quad\Z\;.
  \end{equation*}
\end{Corollary}

\begin{proof}
  Swapping the $p$s and $q$s in both pairs~$(p_\pm,q_\pm)$
  corresponds to changing the orientation on~$M$,
  hence the Eells-Kuiper invariant changes its sign.
  Let~$A(p,q)$ denote the expression in the Corollary,
  then by Theorem~\eqref{C2.T1},
  \begin{equation*}
    A(p_-,q_-)-A(q_+,p_+)
    =\ek\bigl(M_{(p_-,q_-),(p_+,q_+)}\bigr)
	+\ek\bigl(M_{(q_-,p_-),(q_+,p_+)}\bigr)\in\Z\;.
  \end{equation*}
  Because we can choose the pairs~$(p_-,q_-)$ and~$(p_+,q_+)$
  independent of each other,
  subject only to the relation~$\frac{p_-}{q_-}\ne\frac{p_+}{q_+}$,
  it is enough to check that~$A(1,1)=0\in\Z$.
\end{proof}

\subsection{Some Examples}\label{C5}
Some of the manifolds~$M_{(p_-,q_-),(p_+,q_+)}$ are diffeomorphic to well-known
spaces by Grove, Wilking and Ziller~\cite{GWZ}. In this subsection,
we make sure that our computations above agree with other computations
of the invariants.

Let us denote by~$E_{p,n}$ the unit sphere bundle of a fourdimensional
real vector bundle~$V\to S^4$ with Euler class~$n=e(V)$
and half Pontrijagin class~$p=\frac{p_1}2(V)\in\Z\cong H^4(S^4)$.
Such a bundle exists if and only if~$n$ and~$p$ are of the same parity,
and is unique up to isomorphism in this case.
It is known that~$\frac{p_1}2(TE_{p,n})\equiv p\in\Z/n\Z\cong H^4(E_{p,n})$.
The bundles~$E_{p,n}$ and~$E_{-p,n}$ are oriented diffeomorphic,
and~$E_{\pm p,n}$ and~$E_{\pm p,-n}$ are orientation reversing diffeomorphic.
By~\cite[Proposition 2.6]{CG} and~\cite{CE}, we know that
\begin{equation}\label{C5.1}
  q_{E_{p,k}}(\ell)=\frac{\ell(p+\ell)}{2k}
  \qquad\text{and}\qquad
  \mu(E_{p,k})=\frac{p^2-k}{2^5\cdot 7\,k}\;.
\end{equation}
Note that Crowley and Escher in~\cite{CE} use the parameters~$n=k$
and~$m=\frac{p-k}2$.

\begin{Example}\label{M8.X1}
  If~$p_+=p_-=1$,
  then the base~$B$ is the manifold~$S^4$,
  represented as the unit sphere in the space of real trace-free symmetric
  endomorphisms of~$\R^3$ with its natural $\SO(3)$-action by conjugation.
  The manifold~$M$ is a principal $S^3$-bundle over~$S^4$,
  and the induced $\R^4$-bundle~$V\to B$ has Euler number
  \begin{equation*}
    e(V)[B]=k=\frac{q_-^2-q_+^2}8\quad\in\Z
  \end{equation*}
  by~\eqref{M4.2} below.

  Because~$M$ is a principal bundle, we also have~$\frac{p_1}2=\pm k$.
  By~\cite{CG},
  the $q$-invariant is given by
	$$q_M(\ell)=\frac{\ell(k+\ell)}{2k}\;,$$
  which agrees with Theorem~\ref{C2.T2}.

  The formula for the Eells-Kuiper invariant reduces to
  \begin{equation*}
    \ek(M_{(1,q_-),(1,q_+)})
    =\frac{\sign(q_-^2-q_+^2)}{2^5\cdot 7}
	-\frac{q_-^2-q_+^2}{2^8\cdot 7}
    =\frac{\sign e(V)[B]-e(V)[B]}{2^5\cdot 7}\;,
  \end{equation*}
  which agrees with the computations by Crowley and Escher in~\cite{CE}.
\end{Example}

\begin{Example}\label{M8.X2}
  By~\cite{GWZ},
  the space~$M_{(3,1),(1,3)}$ is diffeomorphic
  to the Berger space~$B^7=\SO(5)/\SO(3)$.
  Kitchloo, Shankar and the author computed the Eells-Kuiper invariant
  of this space in~\cite{GKS} and obtained
  \begin{equation*}
    \ek(B^7)=-\frac{27}{1120}\;,
  \end{equation*}
  which agrees with Theorem~\ref{C2.T1}.

  In~\cite{GKS}, we concluded that~$M$ with reversed orientation is
  diffeomorphic to an $S^3$-bundle over~$S^4$ with Euler number~$10$
  and half Pontrijagin number~$8$.
  By Theorem~\ref{C2.T2}, we find
  \begin{equation*}
    q_{M_{(3,1),(1,3)}}(\ell)=\frac{\ell(2+9\ell)}{20}
	\equiv\frac{\ell(2+9\ell)}{20}-\frac{\ell(\ell+1}2
	=-\frac{\ell(8+\ell)}{20}\quad\mod\Z\;,
  \end{equation*}
  which agrees with~\cite{CG}, see~\eqref{C5.1} above.
\end{Example}

\subsection{The Spaces \texorpdfstring{$P_k$}{Pk}}\label{C3}
It is shown in~\cite{GWZ}
that among the~$M_{(p_-,q_-),(p_+,q_+)}$,
only the Berger space~$M_{(3,1),(1,3)}$ and the spaces
\begin{equation*}
  P_k=M_{(1,1),(2k-1,2k+1)}
\end{equation*}
can carry a metric of positive sectional curvature.
So far, such metrics have been found on~$P_1\cong S^7$, the manifold~$P_2$,
and the Berger space.
In this section,
we determine the diffeomorphism type of the~$P_k$.
In particular,
we prove Theorems~\ref{Thm2}, \ref{Thm3} and Corollary~\ref{Cor1}.

We start by evaluating the Dedekind sums of Definition~\ref{C2.D1}.
For~$q=p+2$, these sums simplify as follows.
\begin{align*}
  D(p,p+2)
  &=\sum_{a=1}^{p-1}\Biggl(
	\frac{14\cos\frac{4\pi a}{p}+\cos^2\frac{2\pi a}{p}}
	{2^4\cdot 7\,p^2\sin^2\frac{4\pi a}{p}\sin^2\frac{2\pi a}{p}}
	+\frac{(p+2)\cos\frac{2\pi a}{p}
		\bigl(14+\cos\frac{4\pi a}{p}\bigr)}
	{2^5\cdot 7\,p^2\sin\frac{4\pi a}{p}\sin^3\frac{2\pi a}{p}}
	\Biggr)\\
  &=\frac1{2^4\cdot 7\,p^2}\sum_{a=1}^{p-1}\biggl(
	\frac{15}{4\sin^4\frac{2\pi a}{p}}
	-\frac{14}{\sin^2\frac{4\pi a}{p}}\biggr)\\
  &\qquad
  +\frac{p+2}{2^5\cdot 7\,p^2}\sum_{a=1}^{p-1}\biggl(
        \frac{15}{2\sin^4\frac{2\pi a}{p}}
        -\frac{2}
	{2\sin^2\frac{2\pi a}{p}}\biggr)\\
  &=\frac{15(p+3)}{2^6\cdot 7\,p^2}\sum_{a=1}^{p-1}\frac1{\sin^4\frac{2\pi a}{p}}
  -\frac{p+30}{2^5\cdot 7\,p^2}\sum_{a=1}^{p-1}\frac1{\sin^2\frac{2\pi a}{p}}\;.
\end{align*}

As Zagier pointed out to us,
the sums above can be computed by substituting~$z$ for~$e^{\frac{4\pi a}p}$.
Because the resulting rational function in~$z$ vanishes to sufficiently high
order at~$z=\infty$, we obtain
\begin{align*}
  \sum_{a=1}^{p-1}\frac1{\sin^{2\ell}\frac{2\pi a}p}
  &=\sum_{\begin{smallmatrix}\zeta^p=1\\\zeta\ne1\end{smallmatrix}}
	\Res_{z=\zeta}\biggl(\frac{(-4z)^\ell}{(z-1)^{2\ell}}\cdot\frac p{z^p-1}
	\cdot\frac 1 z\biggr)\\
  &=-\Res_{z=1}\biggl(\frac{(-4)^\ell}{(z-1)^{2\ell+1}}
	\cdot\frac{p\,z^{\ell-1}\,(z-1)}{z^p-1}\biggl)\;.
\end{align*}
For~$\ell=1$ and~$2$, we obtain
\begin{align*}
  \sum_{a=1}^{p-1}\frac1{\sin^2\frac{2\pi a}{p}}
  &=\frac{p^2-1}3\;,\\
	\text{and}\qquad
  \sum_{a=1}^{p-1}\frac1{\sin^4\frac{2\pi a}{p}}
  &=\frac{p^4+10\,p^2-11}{45}\;.
\end{align*}

We combine the above and find that
\begin{align*}
  D(p,p+2)
  &=\frac{15(p+3)}{2^6\cdot 7\,p^2}
	\cdot\frac{p^4+10\,p^2-11}{45}
  -\frac{p+30}{2^5\cdot 7\,p^2}\cdot\frac{p^2-1}3\\
  &=\frac{(p^2-1)(p^3+3\,p^2+9\,p-27)}
	{2^6\cdot 3\cdot 7\,p^2}\;.
\end{align*}

\begin{proof}[Proof of Theorem~\ref{Thm2}]
With Theorem~\ref{C2.T1} and the above, we compute
\begin{align*}
  \ek\bigl(M_{(1,1),(p,p+2)}\bigr)
  &=\frac{\sign(p^2-(p+2)^2)}{2^5\cdot 7}
	+D(1,1)-D(p,p+2)\\
  &\qquad
	-\frac{(p^2-1)^2}{2^2\cdot 7\,p^2(p^2-(p+2)^2)}
	-\frac{2^4(p^2-1)+(p^2-(p+2)^2)}{2^8\cdot 7\,p^2}\\
  &=-\frac{1}{2^5\cdot 7}
	-\frac{(p^2-1)(p^3+3\,p^2+9\,p-27)}{2^6\cdot 3\cdot 7\,p^2}\\
  &\qquad
	+\frac{(p^2-1)(p-1)}{2^4\cdot 7\,p^2}
	-\frac{4(p^2-1)-(p+1)}
		{2^6\cdot 7\,p^2}\\
  &=-\frac{p^3+3\,p^2-4\,p}
	{2^6\cdot3\cdot7}\quad\in\Q/\Z\;.
\end{align*}
With~$p=2k-1$, we have
$$\mu(P_k)=\ek\bigl(M_{(1,1),(2k-1,2k+1)}\bigr)
	=-\frac{4\,k^3-7\,k+3}{2^5\cdot3\cdot7}\quad\in\Q/\Z\;.$$

We also compute~$q_{P_k}$ using Theorem~\ref{C2.T2} as
	$$q_{P_k}(\ell)
	=\frac{\ell((2k+1)^2-1+\ell(2k+1)^2)}{2k}-\frac\ell 2
	=\frac{\ell(k+\ell)}{2k}\quad\in\Q/\Z\;.\quad\qedhere$$
\end{proof}

Having computed the Eells-Kuiper invariant and Crowley's quadratic form~$q$,
we can now compare the spaces~$P_k$ with the principal $S^3$-bundles~$E_{k,k}$
over~$S^4$.


\begin{proof}[Proof of Theorem~\ref{Thm3}]
  By~\cite{C},
  highly connected seven-manifolds are classified up to oriented
  diffeomorphism
  by their Eells-Kuiper invariants and the quadratic function~$q$ on~$H^4$.
  These invariants have been computed for~$E_{k,k}$ in~\cite{CE} and~\cite{CG},
  see~\eqref{C5.1}.

  By Theorem~\ref{Thm2}~\eqref{Thm2.2} and equation~\eqref{C5.1},
  the quadratic forms~$q_{P_k}$ and~$q_{E_{k,k}}$ are isomorphic.
  Hence, the spaces $P_k$ and~$E_{k,k}$ are homeomorphic,
  and even almost diffeomorphic.

  Comparing the value of~$\ek(P_k)$ from Theorem~\ref{Thm2}~\eqref{Thm2.1}
  with~\eqref{C5.1},
  we find
  \begin{equation*}
    \ek(P_k)-\ek(E_{k,k})
    =-\frac{\frac{4k^3-7k}3+1}{2^5\cdot 7}-\frac{k-1}{2^5\cdot 7}
    =\frac{4\frac{k-k^3}3}{2^5\cdot 7}
    =\frac{k-k^3}6\cdot\frac1{28}
  \end{equation*}
  with~$\frac{k-k^3}6\in\Z$.
  Because both~$q$ and~$\ek$ are additive under connected sums
  and~$q_{\Sigma_7}$ is trivial whereas~$\ek(\Sigma_7)=\frac1{28}$,
  we conclude again by~\cite{C} that~$P_k$
  and~$E_{k,k}\mathbin{\#}\Sigma_7^{\#\frac{k-k^3}6}$ are oriented diffeomorphic.
\end{proof}

\begin{Remark}
  Grove, Verdiani and Ziller have already observed in~\cite{GVZ}
  that~$P_2$ is homeomorphic to the unit tangent bundle~$T_1S^4$ of~$S^4$.
  The group~$\Sp(1)$ acts with isolated fixpoints on~$S^4$,
  so this action induces a free action on~$T_1S^4$,
  hence~$T_1S^4$ is diffeomorphic to~$E_{2,2}$.
  Of course, this fits with our result above.
\end{Remark}


\begin{proof}[Proof of Corollary~\ref{Cor1}]
  We start with case~\eqref{Cor1.1}.
  We already know that Crowley's form~$q$ for~$P_k$ is the quadratic form
  of the principal sphere bundle~$E_{k,k}\to S^4$.
  This implies that
  the Pontrijagin number of a sphere bundle~$E_{p,k}$ homeomorphic to~$P_k$
  must be of the form~$p=ak$ with~$a$ odd if~$k$ is even,
  see equation~\eqref{C5.1}.

  It thus remains to solve~$\ek(P_k)=\ek(E_{ak,k})\in\Q/\Z$ depending on~$p=ak$.
  By~\cite{CE},
  we know that
  \begin{equation*}
    \ek(E_{ak,k})-\ek(P_k)
    =\frac{a^2k^2-k}{2^5\cdot 7\cdot k}-\frac{\frac{7k-4k^3}3-1}{2^5\cdot 7}
    =\frac{a^2k-\frac{7k-4k^3}3}{2^5\cdot 7}\quad\mod\Z\;.
  \end{equation*}
  In other words,
	$$a^2k\equiv\frac{7k-4k^3}3\quad\mod 224\Z\;.$$
  It suffices to solve this equation modulo~$7$ and~$32$.

  Modulo~$7$, the equation is trivial if~$7|k$.
  Otherwise, we can clearly solve
	$$a^2\equiv\frac{7-4k^2}3\equiv k^2\quad\mod 7\Z\;.$$

  Modulo~$32$, we start with the case that~$k$ is odd.
  Because~$3\times 11\equiv 1$, we have to solve
	$$a^2\equiv\frac{7-4k^2}3\equiv 77-44k^2
	\equiv 13-12k^2\quad\mod 32\Z\;.$$
  The right hand side equals~$1$ mod~$8$ and hence is a quadratic remainder
  modulo~$32$.
  Next, if~$k$ is even but not divisible by~$8$,
  then we would have at least
	$$a^2\equiv 77\equiv 5\quad\mod 8\Z\;,$$
  but~$5$ is not a quadratic remainder mod~$8$.
  Finally, if~$8|k$, we can clearly solve
	$$a^2\equiv 1\quad\mod 4\Z\;.$$
  In particular,
  if~$k$ is even then~$a$ will be odd,
  so the quadratic forms~$q$ agree as well.
  This settles~\eqref{Cor1.1}.

  In case~\eqref{Cor1.2},
  let~$n=k$ be as above.
  Then~$p=ak$ because we still have
	$$\lk_{P_k}(b,p)=q_{P_k}(b)-q_{P_k}(-b)=0\quad\in\Q/\Z$$
  by Theorem~\ref{Thm2}~\eqref{Thm2.1}.
  We will first try to solve~$\ek(P_k)+\ek(E_{ak,k})=0\in\Q/\Z$.
  We find that
  \begin{equation*}
    224(\ek(E_{ak,k})+\ek(P_k))
    =a^2k+\frac{7k-4k^3}3-2\quad\mod 224\Z\;.
  \end{equation*}

  Modulo~$7$,
  there is no solution if~$7|k$.
  On the other hand, a case-by-case check reveals that
	$$\frac2k-\frac{7-4k^2}3\equiv\frac2k-k^2\quad\mod7\Z$$
  is a quadratic remainder for~$k\in\{1,\dots,6\}$ mod~$7$.
  Thus, a solution mod~$7$ exists if and only if~\eqref{Cor1.2a} is satisfied.

  Modulo~$32$,
  if~$k$ is odd, we have~$k^3\equiv k$ modulo~$8$.
  Hence we have to solve
	$$a^2k\equiv 2-13\,k+12\,k^3\equiv2-k\quad\mod 32\;.$$
  The inverse of~$k=4\ell\pm 1$ mod~$16$ is~$-4\ell\pm 1$, hence we obtain
	$$a^2\equiv -8\ell\pm 2-1\quad\mod 32\;,$$
  which is a quadratic remainder if and only if~$k=4\ell+1\equiv 1$ mod~$4$.
  If~$k=2\ell$ is even, then~$12\,k^3\equiv 0$ mod~$32$, and we are left with
	$$a^2\ell\equiv 1-13\,\ell\equiv 1+3\,\ell\quad\mod 16\;,$$
  and~$\ell$ has to be odd.
  The inverse of~$\ell=\pm1+4m$ is~$\pm1-4m$, and
	$$a^2=\pm 1-4m+3$$
  is a square if and only if~$\ell=1+4m\in\{1,5\}$ modulo~16,
  hence~$k\in\{2,10\}$ modulo~$32$.
  This gives~\eqref{Cor1.2b}.

  Finally,
  we have
	$$-\lk_{E_{ak,k}}=\lk_{P_k}(b\cdot\mathord{},b\cdot\mathord{})$$
  for some~$b\in\Z/k\Z$ if and only if~$b^2\equiv-1$ mod~$k$.
  Because the half Pontrijagin forms vanish
  and the topological Eells-Kuiper invariants
  satisfy~$28\ek(P_k)+28\ek(E_{ak,k})\in\Z$ by the above,
  it follows from~\cite{C} that then the quadratic forms~$q$
  are isomorphic as well.
  Thus by~\cite{C},
  there exists an orientation reversing diffeomorphism~$E_{ak,k}\to P_k$
  if and only if the conditions~\eqref{Cor1.2} hold.
\end{proof}

\begin{Example}\label{C4.X1}\relax
  \begin{enumerate}
  \item\label{C4.X1.1} We have~$P_1=S^7$,
    which of course fibres over~$S^4$,
    independent of the orientation.
    More precisely, $P_1$ is diffeomorphic to~$E_{a,1}$
    if and only if
	$$\frac{a^2-1}{224}\equiv 0\quad\in\Q/\Z\;,$$
    that is,
    if and only if~$a\in\{\pm1,\pm15\}$ mod~$112$.
  \item\label{C4.X1.2} There is an orientation reversing diffeomorphism
    from~$P_2$ to~$E_{4,2}$. Indeed,
	$$\kern3em
	q_{P_2}(1)=\frac34=-q_{E_{4,2}}(1)
	\quad\text{and}\quad
	\ek(P_2)=-\frac1{32}=-\ek(E_{4,2})\in\Q/\Z\;.$$
    More generally, $P_2$ is orientation reversing diffeomorphic
    to~$E_{2a,2}$ if and only if~$a\equiv\pm2$ mod~$28$.
  \item\label{C4.X1.3} There is an orientation preserving diffeomorphism
    of~$P_3$ with~$E_{51,3}$ because
	$$\ek(P_3)=-\frac{15}{112}
	\equiv\frac{433}{112}=\ek(E_{51,3})\in\Q/\Z\;.$$
    More generally, $P_3$ is oriented diffeomorphic to~$E_{3a,3}$
    if and only if~$a\equiv\pm 17$, $\pm31$ mod~$112$.
  \item\label{C4.X1.4} For~$k=4$,
    there exists no diffeomorphic sphere bundle,
    regardless of the orientation.
  \item\label{C4.X1.5} For~$k=5$,
    we have oriented diffeomorphisms with~$E_{5a,5}$
    if and only if~$a\equiv\pm 33$, $\pm 47$ mod~$112$.
    We also have orientation reversing diffeomorphisms with~$E_{5a,5}$
    if and only if~$a\equiv\pm 11$, $\pm 53$ mod~$112$.
    For example,
    \begin{align*}
      \ek(P_5)=-\frac{156}{224}&\equiv\frac{5444}{224}=\ek(E_{165,5})\\
      &\equiv-\frac{604}{224}=-\ek(E_{55,5})\quad\mod\Z\;.
    \end{align*}
    Because~$-1\equiv2^2$ is a quadratic remainder mod~$5$,
    we can compare the quadratic forms in the latter case and find that
	$$q_{P_5}(2\ell)=\frac{2\ell(2\ell-5)}{10}
	\equiv-\frac{\ell(\ell-55)}{10}=-q_{E_{55,5}}(\ell)\quad\mod\Z\;.$$
  \end{enumerate}
\end{Example}

\section{Computation of the invariants}\label{M}
We write the spaces~$M_{(p_-,q_-),(q_-,q_+)}$ as Seifert fibrations
so that we can apply Theorem~\ref{Thm1} to compute their Eells-Kuiper
invariants and $t$-invariants.

\subsection{Description as a Seifert Fibration}\label{M1}
Recall the construction of the spaces~$M=M_{(p_-,q_-),(p_+,q_+)}$
as manifolds of cohomogeneity one with group diagram~\eqref{C1.3},
with the groups~$H$ and~$K_\pm\subset G=\Sp(1)\times\Sp(1)$
described in~\eqref{C1.1} and~\eqref{C1.2}.

The subgroups~$\Sp(1)\times\{e\}$ and~$\{e\}\times\Sp(1)\subset G$ act freely
from the left on the generic orbit~$G/H$.
We focus on the group~$L=\{e\}\times\Sp(1)$
and consider the quotient map
	$$M\longrightarrow B=L\backslash M\;.$$
The group~$L$ acts on the singular orbits~$G/K_\pm$
with finite stabilizer
\begin{equation*}
  L_{gK_\pm}=L\cap gK_\pm g^{-1}=g\Gamma_\pm g^{-1}\subset L\;,
\end{equation*}
at the point~$gK_\pm\in G/K_\pm$, where
\begin{equation}
  \begin{aligned}\label{M1.1}
    \Gamma_-
    &=\<\gamma_-\>\cong\Z/p_-\Z
    &\text{with}\quad
    \gamma_-
    &=\biggl(1,e^{2\pi i\tfrac{q_-}{p_-}}\biggr)\in K_-\cr
    \text{and}\qquad
    \Gamma_+
    &=\<\gamma_+\>\cong\Z/p_+\Z
    &\text{with}\quad
    \gamma_+
    &=\biggl(1,e^{2\pi j\tfrac{q_+}{p_+}}\biggr)\in K_+\;.
  \end{aligned}
\end{equation}

The quotient~$L\backslash M$ has a cohomogeneity one action
by the group~$\SO(3)\cong\penalty0\Sp(1)/\pm 1$.
It is induced by the action
of~$\Sp(1)\times\{e\}\subset G$ on~$M$,
with group diagram 
\begin{equation}\label{M1.0}
  \begin{CD}
    &&\SO(3)\\
    &\;\;\nearrow\;\;&&\;\;\nwarrow\;\;\\
    p_1K_-&\cong&\O(2)&\cong&p_1K_+&\quad.\\
    &\;\;\nwarrow\;\;&&\;\;\nearrow\;\;\\
    &&p_1H
  \end{CD}
\end{equation}
Here~$p_1$ denotes the projection
	$$G=\Sp(1)\times\Sp(1)
	\longrightarrow(\Sp(1)/\{\pm1\})\times\{e\}\cong\SO(3)\;.$$
In particular,
$p_1H\cong Q/\{\pm1\}\cong(\Z/2\Z)^2$ is the subgroup
of diagonal matrices in~$\SO(3)$.

If~$a$ is an imaginary unit quaternion,
let~$S^1_a\subset\Sp(1)$ denote the one-parameter subgroup
generated by~$a$.
Because
\begin{equation*}
  p_1K_-=(S^1_i+jS^1_i)/\{\pm1\}\cong\O(2)
  \quad\text{and}\quad
  p_1K_+=(S^1_j+iS^1_j)/\{\pm1\}\cong\O(2)\;,
\end{equation*}
the singular orbits of~$B$ are given by
\begin{equation*}
  B_\pm=L\backslash M_\pm 
  \cong\SO(3)/\O(2)\cong\R P^2\;.
\end{equation*}

We want to understand the geometry of~$p\colon M\to B$
near the singular orbits.
The action of~$K_-$ on~$S^1\cong K_-/H$ extends to~$\C\supset S^1$ by
\begin{equation}\label{M1.2}
  \bigl(e^{ip_-\thet},e^{iq_-\thet}\bigr)z=e^{4i\thet}z
	\quad\text{and}\quad
  \Bigl(e^{ip_-\thet}j,(-1)^{\tfrac{q_--p_-}2}e^{iq_-\thet}j\Bigr)z
  =e^{4i\thet}\bar z\;,
\end{equation}
and there is a similar action of~$K_+$ on~$\C$.
Thus,
the singular orbits~$M_\pm=G/K_\pm$ have neighbourhoods~$M\setminus M_\mp$
diffeomorphic to the normal bundles
\begin{equation}\label{M1.3}
  N_\pm=G\times_{K_\pm}\C\longrightarrow G/K_\pm\;.
\end{equation}
For the generator~$\gamma_\pm\in\Gamma_\pm$ of~\eqref{M1.1},
we have the angle~$\thet=\frac{2\pi}{p_\pm}$ in~\eqref{M1.2}.
So~$\gamma_\pm$ acts on the fibre of~$N_\pm$ by multiplication
with~$e^{\frac{8\pi i}{p_\pm}}\in\mu_{p_\pm}$,
where~$\mu_{p_\pm}$ denotes the group of $p_\pm$th roots of unity.

Projecting down to~$B$,
neighbourhoods of~$B_\pm$ are given by
\begin{equation}\label{M1.4}
  B\setminus B_\mp
  \cong\SO(3)\times_{\O(2)}\C/\mu_{p_\pm}\;,
\end{equation}
where the action of~$\O(2)\cong(S^1_i\cup jS^1_i)/\{\pm1\}$
on~$\C/\mu_{p_\pm}$ is given by
\begin{equation}\label{M1.5}
  \pm e^{i\thet}\colon z\mapsto e^{4i\thet}z
  \qquad\text{and}\qquad
  \pm e^{i\thet}j\colon z\mapsto e^{4i\thet}\bar z\;.
\end{equation}

We fix an origin~$o=(p_1K_-)$ in~$B_-$
and consider a path~$g_t$ from~$g_0=e$ to~$g_1=\{\pm j\}\in O(2)$,
then~$g_to$ describes a nontrivial loop in~$B_-\cong\R P^2$.
The stabiliser of~$g_to\in B_-$ is given by~$g\Gamma_-g^{-1}$,
and a get a path of generators~$\gamma_{-,t}=g_t\gamma_-g_t^{-1}$
from~$\gamma_-=\gamma_{-,0}$ to
	$$\gamma_{-,1}=\bigl(1,j\,e^{2\pi iq_-/p_-}\,(-j)\bigr)
	=\gamma_{-,0}^{-1}\;.$$
We conclude
that the twisted sectors of~$B$ are diffeomorphic to the universal covering
spaces~$\tilde B_\pm\cong S^2$ of~$B_\pm$.
Let us summarize our results so far.

\begin{Proposition}\label{M1.P1}
  The map~$p\colon M\to B=L\backslash M$ is a Seifert fibration
  and a left $\Sp(1)$-principal orbibundle.

  The inertia orbifold~$\Lambda B$ of the base orbifold~$B$
  is diffeomorphic to
  \begin{gather*}
    \Lambda B
    =B\sqcup\biggl(\tilde B_-\times\biggl\{1,\dots,\frac{p_--1}2\biggr\}\biggr)
    \sqcup\biggl(\tilde B_+\times\biggl\{1,\dots,\frac{p_+-1}2\biggr\}\biggr)\;.
	\tag{1}\label{M1.P1.1}
  \end{gather*}
  Elements~$(p,(\gamma_\pm^k))\in\Lambda B\setminus B$ are represented
  by~$(p,\ell)\in\tilde B_\pm\times\{\ell\}$
  with~$\pm k\equiv\ell$ mod~$p_\pm$
  and~$\ell\in\bigl\{1,\dots,\frac{p_\pm-1}2\bigr\}$.
  The components~$\tilde B_\pm\times\{k\}$ have multiplicity
  \begin{gather*}
    m(\gamma_\pm^k)=\#\Gamma_\pm=p_\pm\;.\tag{2}\label{M1.P1.2}
  \end{gather*}

  In a suitable orbifold chart around~$p$,
  the element~$\gamma_\pm^k$ acts by
  \begin{gather*}
    \rho(\gamma_\pm^k)=
    \begin{pmatrix}
      1&0\\0&e^{8\pi ik\tfrac1{p_\pm}}
    \end{pmatrix}
    \in\U(2)\;,\tag{3}\label{M1.P1.3}
  \end{gather*}
  and the fibrewise action on~$S^3$ is conjugate to
  \begin{gather*}
    e^{2\pi ik\frac{q_\pm}{p_\pm}}\in\Sp(1)\;.\tag{4}\label{M1.P1.4}
  \end{gather*}
\end{Proposition}

\begin{proof}
  From the discussion above,
  it is clear that~$p$ is both a Seifert fibration
  and an $\Sp(1)$-principal bundle,
  where the group~$L\cong\Sp(1)$ acts from the left.
  Assertion~\eqref{M1.P1.1} follows from the considerations above,
  and~\eqref{M1.P1.2} follows from the definition of multiplicity
  in~\eqref{A2.0}.

  We construct an orbifold chart by taking a neighbourhood of~$p$
  in~$B_\pm\cong\R P^2$ that is diffeomorphic to~$\C$ with
  trivial action of~$\gamma_\pm^k$.
  The normal bundle of~$B_\pm$ in~$B$ is represented by another
  copy of~$\C$ on which~$\gamma_\pm^k$ acts as in~\eqref{M1.5}.
  This proves~\eqref{M1.P1.3}.
  Finally,
  the action of~$\gamma_\pm^k$ on~$S^3$ follows from~\eqref{M1.1}
  and is conjugate to the expression in~\eqref{M1.P1.4}.
\end{proof}

\subsection{The geometry of the Seifert Fibration}\label{M2}

We want to study the metric structure on~$M$ and~$B$.
In particular, we want to derive formulas for the curvature
of the horizontal and vertical tangent bundles
of the Seifert fibration~$M\to B$.
By integration over~$B$,
we can determine the Pontrijagin numbers and the Cheeger-Simons numbers
that are necessary to compute the Eells-Kuiper invariant.

Recall that as a cohomogeneity one manifold, we may write
\begin{equation*}
  M=\bigl([-1,1]\times G/H\bigr)\bigm/\sim\;.
\end{equation*}
Let~$\tau\colon M\to[-1,1]$ denote the natural projection,
and define a curve~$c\colon[-1,1]\penalty0\to M$ joining~$G/K_-$ to~$G/K_+$
by
\begin{equation*}
  c(t)=[t,eH]\;.
\end{equation*}
We define $G$-invariant vector fields~$e_0$, \dots, $e_3$
and~$f_1$, $f_2$, $f_3$ on~$M\setminus(M_+\cup M_-)=\tau^{-1}(-1,1)$
by specifying them along~$c$.
Therefore put~$e_0(c(t))=\dot c(t)$ and
\begin{equation}\label{M2.1}
  \begin{aligned}\relax
    e_1(c(t))&=\frac d{dt}\Bigr|_{t=0}\bigl(e^{it},1\bigr)(c(t))\;,\qquad&
    f_1(c(t))&=\frac d{dt}\Bigr|_{t=0}\bigl(1,e^{it}\bigr)(c(t))\;,\\
    e_2(c(t))&=\frac d{dt}\Bigr|_{t=0}\bigl(e^{jt},1\bigr)(c(t))\;,\qquad&
    f_2(c(t))&=\frac d{dt}\Bigr|_{t=0}\bigl(1,e^{jt}\bigr)(c(t))\;,\\
    e_3(c(t))&=\frac d{dt}\Bigr|_{t=0}\bigl(e^{kt},1\bigr)(c(t))\;,\qquad&
    f_3(c(t))&=\frac d{dt}\Bigr|_{t=0}\bigl(1,e^{kt}\bigr)(c(t))\;.
  \end{aligned}
\end{equation}
We regard the vector fields~$e_0$, \dots, $e_3$ as horizontal
and~$f_1$, $f_2$, $f_3$ as vertical fields with respect to the Seifert
fibration~$M\to B$.
All Lie brackets between these vector fields vanish except
\begin{equation}\label{M2.2}
  \begin{aligned}\relax
    [e_1,e_2]&=2e_3\;,\qquad&
    [e_2,e_3]&=2e_1\;,\qquad&
    [e_3,e_1]&=2e_2\;,\\
    [f_1,f_2]&=2f_3\;,\qquad&
    [f_2,f_3]&=2f_1\;,\qquad&
    [f_3,f_1]&=2f_2\;.
  \end{aligned}
\end{equation}

Let~$\phy\colon[0,1]$ denote a cutoff function
with~$\phy|_{[0,\eps)}=0$ and~$\phy|_{(1-\eps,1)}=1$ for some small~$\eps>0$.
For~$x\in M$, let~$t=\tau(x)$, and define functions~$f$,
$g\colon\tau^{-1}((-1,0])\to\R$ by
\begin{equation}\label{M2.3}
  f=\frac{4+4\tau\cdot\phy(-\tau)}{4+(p_--4)\cdot\phy(-\tau)}
	\qquad\text{and}\qquad
  g=\frac{q_-}4\,f'
\end{equation}
These functions satisfy
\begin{equation}\label{M2.4}
  \begin{aligned}
  f|_{(-1,\eps-1)}&=\frac 4{p_-}(1+\tau)\;,&
  f|_{(-\eps,0]}&=1\;,\\
  g|_{(-1,\eps-1)}&=\frac{q_-}{p_-}\;,\qquad\text{and}&\qquad
  g|_{(-\eps,0]}&=0\;.
  \end{aligned}
\end{equation}

Let~$g^{TM}$ be a $G$-invariant metric
such that for~$t\le 0$,
the vectors~$e_0$, $e_2$, $e_3$, $f_2$ and~$f_3$ are orthonormal
and perpendicular to the subspace spanned by~$e_1$, $f_1$,
and such that on this subspace,
$g^{TM}$ is given by the matrix
\begin{equation}\label{M2.5}
  g^{TM}|_{\Span\{e_1,f_1\}}
  =\begin{pmatrix}f^2+g^2&-g\\-g&1\end{pmatrix}\;.
\end{equation}
This metric extends to a smooth metric around~$G/K_-=\tau^{-1}(-1)$
by Theorem~6.1 in~\cite{GVZ}.
The orbits of~$L=\{e\}\times\Sp(1)$ are all quotients of round spheres
with the standard metric.

A $g^{TM}$-orthonormal frame on~$\tau^{-1}(-1,0]$ is given by~$(\bar e_0,
\dots,\bar f_3)$,
where~$\bar f_i=f_i$ and~$\bar e_i=e_i$ except
\begin{equation}\label{M2.6}
  \bar e_1=\frac1f\,e_1+\frac gf\,f_1
\end{equation}

By~\eqref{M2.2},
the Lie brackets between the vector fields~$\bar e_0$, \dots, $\bar f_3$
vanish except
\begin{equation}\label{M2.7}
  \begin{aligned}\relax
    [\bar e_0,\bar e_1]
    &=-\frac{f'}f\,\bar e_1+\frac{g'}f\,\bar f_1\;,&
    [\bar e_1,\bar f_2]&=\frac{2g}f\,\bar f_3\;,&
    [\bar e_1,\bar f_3]&=-\frac{2g}f\,\bar f_2\;,\\
    [\bar e_1,\bar e_2]&=\frac2f\bar e_3\;,&
    [\bar e_2,\bar e_3]&=2f\,\bar e_1-2g\,\bar f_1\;,&
    [\bar e_3,\bar e_1]&=\frac2f\bar e_2\;,\\
    [\bar f_1,\bar f_2]&=2\bar f_3\;,&
    [\bar f_2,\bar f_3]&=2\bar f_1\;,\qquad\text{and}&
    [\bar f_3,\bar f_1]&=2\bar f_2\;.
  \end{aligned}
\end{equation}

For~$t\ge1$, we can proceed similarly.
We extend~$f$ and~$g$ to~$\tau^{-1}[0,1)$ by
\begin{equation}\label{M2.8}
  f=\frac{4-4\tau\cdot\phy(\tau)}{4+(p_+-4)\cdot\phy(\tau)}
	\qquad\text{and}\qquad
  g=-\frac{q_+}4\,f'\;,
\end{equation}
such that
\begin{equation}\label{M2.9}
  \begin{aligned}\relax
    f|_{[0,\eps)}&=1\;,&
    f|_{(1-\eps,1)}&=\frac 4{p_+}(1-\tau)\;,\\
    g|_{[0,\eps)}&=0\;,\qquad\text{and}&\qquad
    g|_{(1-\eps,1)}&=\frac{q_+}{p_+}\;.
  \end{aligned}
\end{equation}
We then modify the metric similarly,
such that we obtain a $g^{TM}$-orthonormal frame~$\bar e_0$, \dots, $\bar f_3$
that differs from~$e_0$, \dots, $f_3$ only by
\begin{equation*}
  \bar e_2=\frac1f\,e_2+\frac gf\,f_2\;.
\end{equation*}
Again,
this metric extends smoothly over~$\tau^{-1}[0,1]$,
and it is also compatible along~$\tau^{-1}(0)$
with the metric chosen above on~$\tau^{-1}[-1,0]$.

\subsection{Orbifold Characteristic Numbers of the Base Space}\label{M3}

The base orbifold~$B=L\backslash M$ has a principal cohomogeneity one action by~$\SO(3)$, see~\eqref{M1.0}.
We define~$\tau\colon B\to[-1,1]$ similar as above.
Then we can describe~$\tau^{-1}(-1,1)$
as a product~$(-1,1)\times\Sp(1)/Q$. 
The projection~$M\to B$ becomes a Riemannian submersion
with respect to an invariant Riemannian metric~$g^{TB}$
on~$(-1,1)\times\Sp(1)/Q$ 
that degenerates over~$\{-1,1\}$.

By abuse of notation,
let~$\bar e_0$, \dots, $\bar e_3$ also denote the projection of
the vector fields above to~$B$,
then these vector fields form a $g^{TB}$-orthonormal frame everywhere,
and their nonzero Lie brackets on~$\tau^{-1}[-1,0]$ are completely described by
\begin{alignat*}2
  [\bar e_0,\bar e_1]&=-\frac{f'}f\,\bar e_1\;,\qquad&
  [\bar e_1,\bar e_2]&=\frac2f\,\bar e_3\;,\\
  [\bar e_2,\bar e_3]&=2f\,\bar e_1\;,\qquad\text{and}\qquad&
  [\bar e_3,\bar e_1]&=\frac2f\,\bar e_2\;.
\end{alignat*}

The Christoffel symbols of the Levi-Civita connection on~$B$
with respect to these fields over~$\tau^{-1}((-1,0])\subset B$ are given by
\begin{alignat*}2
  \Gamma_{10}^1&=-\Gamma_{11}^0=\frac{f'}f\;,&
  \Gamma_{12}^3&=-\Gamma_{13}^2=\frac 2f-f\;,\\
  \Gamma_{23}^1&=-\Gamma_{21}^3=f\;,&\qquad\text{and}\qquad
  \Gamma_{31}^2&=-\Gamma_{32}^1=f\;,
\end{alignat*}
those~$\Gamma_{ij}^k$ not listed above vanish.
The Riemannian curvature tensor as a $4\times 4$-matrix is given by
\begin{equation}\label{M3.2}
  R=\left(\begin{smallmatrix}
    0&-\frac{f''}f\,\bar e^{01}+2f'\,\bar e^{23}&
	f'\,\bar e^{13}&-f'\,\bar e^{12}\\
    \frac{f''}f\,\bar e^{01}-2f'\,\bar e^{23}&0&
	-f'\,\bar e^{03}+f^2\,\bar e^{12}&
	f'\,\bar e^{02}+f^2\,\bar e^{13}\\
    -f'\,\bar e^{13}&f'\,\bar e^{03}-f^2\,\bar e^{12}&
	0&2f'\,\bar e^{01}+(4-3f^2)\,\bar e^{23}\\
    f'\,\bar e^{12}&-f'\,\bar e^{02}-f^2\,\bar e^{13}&
	-2f'\,\bar e^{01}-(4-3f^2)\,\bar e^{23}&0\\
  \end{smallmatrix}\right)\;,
\end{equation}
with~$\bar e^{ij}$ shorthand for~$\bar e^{i}\wedge\bar e^{j}$.
Over~$\tau^{-1}[0,1)$, the matrix looks similar,
but with the matrix indices and the form indices~$1$, $2$, $3$ permuted
cyclically.

The Euler and Pontrijagin forms are thus given by
\begin{align}
  \begin{split}\label{M3.5}
    p_1\bigl(TB,\nabla^{TB}\bigr)
    &=\frac1{8\pi^2}\trace(R^2)
    =\frac1{\pi^2}\,\biggl(\frac{f'f''}f+4f'f^2-4f'\biggr)
	\,\bar e^{0123}\;,\\
	\text{and}\qquad
    e\bigl(TB,\nabla^{TB}\bigr)
    &=\frac1{4\pi^2}\Pf(R)
    =\frac1{4\pi^2}
	\,\biggl(6f^{\prime2}+3f''f-\frac{4f''}f\biggr)
	\,\bar e^{0123}\;.
  \end{split}
\end{align}
The fibre~$\R P^3/(\Z/2\Z)^2$ over~$t$
has volume~$\frac{\pi^2}4\,f(t)$.
By~\eqref{M2.4}, \eqref{M2.9}, and~\eqref{M3.5},
we get the orbifold characteristic numbers
\begin{align}
  \begin{split}\label{M3.6}
    \int_Bp_1\bigl(TB,\nabla^{TB}\bigr)
    &=\frac{f'(t)^2+2f(t)^4-4f(t)^2}8\Bigr|_{t=-1}^1
	=\frac2{p_+^2}-\frac2{p_-^2}\;,\\
    \int_Be\bigl(TB,\nabla^{TB}\bigr)
    &=\frac{3f'(t)f(t)^2-4f'(t)}{16}\Bigr|_{t=-1}^1
	=\frac1{p_-}+\frac1{p_+}\;.
  \end{split}
\end{align}

\subsection{Characteristic Numbers of the Seifert
\texorpdfstring{$S^3$}{S3}-Fibration}\label{M4}

The Seifert Fibration~$M\to B$ is a principal bundle
with structure group~$\Sp(1)$.
Let~$TX=\ker p_*$ denote the vertical tangent bundle.

A connection~$\omega\in\Hom(TM,TX)$
acts as the identity on~$TX$ and is $\Sp(1)$-invariant.
It is uniquely described by its horizontal bundle~$T^HM=\ker\omega$.
We define~$\omega$ such that
\begin{equation*}
  T^HM=\Span\{\bar e_0,\bar e_1,\bar e_2,\bar e_3\}\;,
\end{equation*}
then by~\eqref{M2.7}, its curvature~$\Omega$ is given by
\begin{equation}\label{M4.0}
  \Omega=
  \begin{cases}
    \Bigl(-\frac{g'}f\,\bar e^{01}+2g\,\bar e^{23}\Bigr)\,\bar f_1
	&\text{for~$t\in[-1,0]$, and}\\
    \noalign{\smallskip}
    \Bigl(-\frac{g'}f\,\bar e^{02}-2g\,\bar e^{13}\Bigr)\,\bar f_2
	&\text{for~$t\in[0,1]$.}
  \end{cases}
\end{equation}


The Seifert fibration~$M\to B$ is the unit sphere orbibundle
of the vector orbibundle
\begin{equation*}
  V=M\times_{\Sp(1)}\H\to B\;.
\end{equation*}
The vectors~$\bar f_1$, $\bar f_2$ correspond
to the elements~$i$, $j\in\sP(1)\subset\H$.
These act on~$\H\cong\R^4$ with Pfaffian~$\Pf(i)=\Pf(j)=1$,
hence the Euler form of the connection~$\nabla^V$ on~$V$ induced by~$\omega$
is given by
\begin{equation}\label{M4.1}
  e(V,\nabla^V)
  =-\frac{g'g}{\pi^2f}\,\bar e^0\wedge\bar e^1\wedge\bar e^2\wedge\bar e^3\;.
\end{equation}

As in~\eqref{M3.6}, integration over~$B$ gives the characteristic number
\begin{equation}\label{M4.2}
  \int_Be(V,\nabla^V)
  =-\int_{-1}^1\frac{g'(t)g(t)}4\,dt
  =-\frac{g(t)^2}8\Bigr|_{t=-1}^1
  =\frac{q_-^2}{8p_-^2}-\frac{q_+^2}{8p_+^2}\;.
\end{equation}

With these computations,
we can now compute the first term in the adiabatic limit
of formula~\eqref{C2.1} for the Eells-Kuiper invariant.
Recall that~$B\cong B\times\{e\}\subset\Lambda B$.

\begin{Proposition}\label{M4.P1}
  For the Seifert fibration~$p\colon M_{(p_-,q_-),(p_+,q_+)}\to B$,
  we have
  \begin{multline*}
    \frac12\int_B\hat A_{\Lambda B}\bigl(TB,\nabla^{TB}\bigr)
	\,2\eta_{\Lambda B}\bigl(D_{S^3}\bigr)
    +\frac1{2^5\cdot7}\int_B\hat L_{\Lambda B}\bigl(TB,\nabla^{TB}\bigr)
	\,2\eta_{\Lambda B}\bigl(B_{S^3}\bigr)\\
    =-\frac1{2^7\cdot7}\,\int_Be\bigl(V,\nabla^V\bigr)
    =\frac1{2^{10}\cdot7}
	\biggl(\frac{q_+^2}{p_+^2}-\frac{q_-^2}{p_-^2}\biggr)\;.
  \end{multline*}
\end{Proposition}

\begin{proof}
  Let~$V\to B$ be the induced vector bundle with connection~$\nabla^V$
  as above.
  The $\eta$-form of the untwisted fibrewise Dirac operator and the fibrewise
  signature operator are given by
  \begin{align*}
    2\eta_{\Lambda B}\bigl(D_{S^3}\bigr)|_B
    &=\eta_{\frac\Omega{2\pi i}}\bigl(D_{S^3}\bigr)
    =-\frac1{960}\,e\bigl(V,\nabla^V\bigr)\\
	\text{and}\qquad
    2\eta_{\Lambda B}\bigl(B_{S^3}\bigr)|_B
    &=\eta_{\frac\Omega{2\pi i}}\bigl(B_{S^3}\bigr)
    =-\frac1{30}\,e\bigl(V,\nabla^V\bigr)\in\Omega^4(B)
  \end{align*}
  by Theorem~\ref{A6.T1} and~\cite[Theorem~3.9]{Go}.

  Because both $\eta$-forms are homogeneous of degree~$4$,
  we only need the degree zero components of~$\hat A$ and~$\hat L$,
  which are given by
  \begin{equation*}
    \hat A\bigl(TB,\nabla^{TB}\bigr)^{[0]}=1
	\qquad\text{and}\qquad
    \hat L\bigl(TB,\nabla^{TB}\bigr)^{[0]}=2^{\tfrac{\dim TB}2}=4\;.
  \end{equation*}

  From the above and~\eqref{M4.2},
  we obtain our result because
  \begin{align*}
    \frac12\int_B\hat A\bigl(TB,\nabla^{TB}\bigr)
	\,\eta_{\frac\Omega{2\pi i}}\bigl(D_{S^3}\bigr)
    &=\frac1{2^{10}\cdot3\cdot5}
	\biggl(\frac{q_+^2}{p_+^2}-\frac{q_-^2}{p_-^2}\biggr)\\
    \text{and}\quad
    \frac1{2^5\cdot7}\int_B\hat L\bigl(TB,\nabla^{TB}\bigr)
	\,\eta_{\frac\Omega{2\pi i}}\bigl(B_{S^3}\bigr)
    &=\frac1{2^7\cdot3\cdot5\cdot7}
	\biggl(\frac{q_+^2}{p_+^2}-\frac{q_-^2}{p_-^2}\biggr)\;.\qedhere
  \end{align*}
\end{proof}

\subsection{The Contributions from the Twisted Sectors}\label{M5}

To compute the contribution from the twisted sectors,
we need some equivariant characteristic numbers
and the equivariant $\eta$-forms of the pullback of~$M$ to~$\tilde B_\pm$.
Let~$(p,(\gamma_\pm^a))\in\Lambda B\setminus B$,
let~$N_\pm\to B_\pm$ the normal bundle of~$B_\pm\cong\R P^2$ in~$B$,
and let~$\tilde N_\pm$ denote its pullback to~$\tilde B_\pm\cong S^2$.

In an orbifold chart,
the elements~$\gamma_\pm^a$ for~$a=1$, \dots, $\frac{p_\pm-1}2$
act on~$\tilde N_\pm$ by multiplication with~$e^{8\pi ia\tfrac1{p_\pm}}\in S^1\cong\SO(2)$,
see Proposition~\ref{M1.P1}~\eqref{M1.P1.3}.
Because~$\Gamma_\pm\cong\Z/p_\pm\Z$ is an odd cyclic group,
this action has a unique lift to~$\Spin(2)$,
represented by
	$$\tilde\gamma_\pm^a=e^{4\pi ia\tfrac1{p_\pm}}\in S^1\cong\Spin(2)\;.$$
This lift provides us with a unique section of the bundle~$\widetilde{\Lambda B}\to\Lambda B$ of~\eqref{A2.4}. All forms in~$\Omega^\bullet(\Lambda B;\widetilde{\Lambda B})$ and in~$\Omega^\bullet(\Lambda B;\widetilde{\Lambda B}\otimes o(\Lambda B))$ will be computed with respect to this lift and with respect to the orientation of~$\tilde B_\pm\cong S^2$ with volume form~$\bar e^{23}$ or~$\bar e^{31}$, respectively.

The curvature~$R^{\tilde{N}_-}$ can be computed as the limit
of~$\<R\bar e_0,\bar e_1\>|_{\Span\{\bar e_2,\bar e_3\}}$ as~$t\to-1$,
so by~\eqref{M3.2} and~\eqref{M2.4} we have
\begin{equation*}
  R^{\tilde{N}_-}
  =\begin{pmatrix}0&2f'\,\bar e^{23}\\-2f'\,\bar e^{23}&0\end{pmatrix}
  =-\frac{8i}{p_-}\,\bar e^{23}
\end{equation*}
with~$\bar e_1=i\bar e_0$.
The induced curvature of the spinor bundle at the origin is
\begin{equation*}
  R^{S^\pm\tilde N_-}=\mp\frac{4i}{p_-}\,\bar e^{23}
\end{equation*}
and similarly for~$\tilde N_+$.
By~\eqref{A2.6}--\eqref{A2.3},
the orbifold $\hat A$-form on~$\tilde B_-\times\{a\}\subset\Lambda B$
is represented by
\begin{multline}\label{M5.1}
    \hat A_{\Lambda B}\bigl(TB,\nabla^{TB}\bigr)
    =-\frac{\hat A(TB_-,\nabla^{TB_-})}
	{m(\tilde\gamma_-^a)
		\,\ch_{\tilde\gamma_-^a}(S^{+}\tilde N_--S^{-}\tilde N_-,
		\nabla^{S\tilde N_{B_-}})}\\
    =-\frac1{p_-\cdot 2i\sin\bigl(\frac 4{p_-}
		\bigl(\pi a+\frac{\bar e^{23}}{2\pi i}\bigr)\bigr)}
      \quad\in\Omega^\bullet\bigl(\tilde B_-\bigr)
\end{multline}

A similar computation gives the orbifold $\hat L$-form
\begin{multline}\label{M5.2}
    \hat L_{\Lambda B}(TB,\nabla^{TB})
    =\hat A_{\Lambda B}(TB,\nabla^{TB})
	\,\ch_{\Lambda B}\bigl(\mathcal S^+B_-
		+\mathcal S^-B_-,\nabla^{\mathcal SB_-}\bigr)\\
    =\frac{2i}{p_-}\,\cot\biggl(\frac4{p_-}\biggl(\pi a
		+\frac{\bar e^{23}}{2\pi i}\biggr)\biggr)
      \quad\in\Omega^\bullet\bigl(\tilde B_-\bigr)\;.
\end{multline}


We also need the equivariant $\eta$-forms of~$G|_{B_\pm}\to B_\pm$.
We know by Proposition~\ref{M1.P1}~\eqref{M1.P1.4}
that~$\gamma_\pm^a$ act on the fibres~$S^3$ as
\begin{equation*}
  \gamma_-^a=e^{2\pi i\,\frac{aq_-}{p_-}}\qquad\text{and}\qquad
  \gamma_+^a=e^{2\pi j\,\frac{aq_+}{p_+}}\;,
\end{equation*}
and the curvatures at~$B_\pm$ are given by~\eqref{M2.4} and~\eqref{M4.0} as
\begin{equation*}
  \Omega_-=-\frac{2q_-}{p_-}\,\bar e^{23}\,\bar f_1\qquad\text{and}\qquad
  \Omega_+=-\frac{2q_+}{p_+}\,\bar e^{31}\,\bar f_2\;.
\end{equation*}
We note that both the curvature and the action of~$\Gamma_\pm$
are $L$-invariant,
so both act from the same side on the generic fibre~$S^3\cong\Sp(1)$.

We compute the mixed equivariant $\eta$-invariant,
from which we derive the orbifold $\eta$-form of Definition~\ref{A3.D2}
using Theorem~\ref{A6.T1}.
We use the formulas for the equivariant $\eta$-invariants of
the untwisted Dirac operator~$D_{S^3}$ in~\cite{HR} and of the
odd signature operator~$B_{S^3}$ in~\cite{APS}.
On~$\tilde B_-\times\{a\}\subset\Lambda B$, we obtain in particular
\begin{align}
  \begin{split}\label{M5.4}
    2\eta_{\Lambda B}(D_{S^3})
    &=\eta_{\tilde\gamma_-^a\,e^{-\frac{\Omega_-}{2\pi i}}}(D_{S^3})
    =-\frac1{2\sin^2\bigl(\frac{q_-}{p_-}\bigl(\pi a
		+\frac{\bar e^{23}}{2\pi i}\bigr)\bigr)}\\
	\text{and}\qquad
    2\eta_{\Lambda B}(B_{S^3})
    &=\eta_{\tilde\gamma_-^a\,e^{-\frac{\Omega_-}{2\pi i}}}(B_{S^3})
    =-\cot^2\biggl(\frac{q_-}{p_-}\biggl(\pi a
		+\frac{\bar e^{23}}{2\pi i}\biggr)\biggr)
  \end{split}
\end{align}

We can now compute the contribution from the singular orbits~$M_\pm$
to the adiabatic limit of the $\eta$-invariants
and the Eells-Kuiper invariant
and relate it to the generalised Dedekind sums~$D(p,q)$
of Definition~\ref{C2.D1}.

\begin{Proposition}\label{M5.P1}
  The singular orbits~$M_\pm$ contribute to the Eells-Kuiper invariant by
  the generalised Dedekind sums
  \begin{multline*}
    \int_{\tilde B_-\times\left\{1,\dots,\frac{p_--1}2\right\}}
	\biggl(\frac12\hat A_{\Lambda B}\bigl(TB,\nabla^{TB}\bigr)
		\,2\eta_{\Lambda B}(D_{S^3})\\
	+\frac1{2^5\cdot 7}\hat L_{\Lambda B}\bigl(TB,\nabla^{TB}\bigr)
		\,2\eta_{\Lambda B}(B_{S^3})\biggr)
      =D(p_-,q_-)\;,
  \end{multline*}
  \begin{multline*}
   \int_{\tilde B_+\times\left\{1,\dots,\frac{p_+-1}2\right\}}
	\biggl(\frac12\hat A_{\Lambda B}\bigl(TB,\nabla^{TB}\bigr)
		\,2\eta_{\Lambda B}(D_{S^3})\\
	+\frac1{2^5\cdot 7}\hat L_{\Lambda B}\bigl(TB,\nabla^{TB}\bigr)
		\,2\eta_{\Lambda B}(B_{S^3})\biggr)
      =-D(p_+,q_+)\;.
  \end{multline*}
\end{Proposition}

\begin{proof}
  The twisted sectors~$\tilde B_\pm\times\{a\}$ are spheres of
  sectional curvature~$4$ by~\eqref{M3.2},
  in particular, their volume is~$\pi$.
  We combine~\eqref{A2.3} and~\eqref{A3.E1}
  with~\eqref{M5.1}--\eqref{M5.4} and find that
  \begin{multline*}
    \int_{\tilde B_-\times\{a\}}
	\biggl(\frac12\hat A_{\Lambda B}\bigl(TB,\nabla^{TB}\bigr)
		\,2\eta_{\Lambda B}(D_{S^3})
	+\frac1{2^5\cdot 7}\hat L_{\Lambda B}\bigl(TB,\nabla^{TB}\bigr)
		\,2\eta_{\Lambda B}(B_{S^3})\biggr)\\
    \begin{aligned}
      &=\int_{\tilde B_-}\Biggl(\frac1{p_-\cdot 8i\cdot\sin\bigl(\frac 4{p_-}
		\bigl(\pi a+\frac{\bar e^{23}}{2\pi i}\bigr)\bigr)
	        \cdot\sin^2\bigl(\frac{q_-}{p_-}\bigl(\pi a
		+\frac{\bar e^{23}}{2\pi i}\bigr)\bigr)}\\
      &\kern4em
	-\frac i{2^4\cdot 7\,p_-}\cdot\cot\biggl(\frac4{p_-}\biggl(\pi a
		+\frac{\bar e^{23}}{2\pi i}\biggr)\biggr)
	\cdot\cot^2\biggl(\frac{q_-}{p_-}\biggl(\pi a
		+\frac{\bar e^{23}}{2\pi i}\biggr)\biggr)\Biggr)\\
      &=\frac d{dx}\biggr|_{x=\pi a}\biggl(\frac1{8i\,p_-\,\sin\frac{4x}{p_-}
	        \,\sin^2\frac{q_-x}{p_-}}
	-\frac i{2^4\cdot 7\,p_-}\,\cot\frac{4x}{p_-}
	\,\cot^2\frac{q_-x}{p_-}\biggr)\,
	\int_{\tilde B_-}\frac{\bar e^{23}}{2\pi i}\\
      &=\frac{14\cos\frac{4\pi a}{p_-}+\cos^2\frac{q_-\pi a}{p_-}}
	{2^3\cdot 7\,p_-^2\,\sin^2\frac{4\pi a}{p_-}\sin^2\frac{q_-\pi a}{p_-}}
	+\frac{q_-\cos\frac{q_-\pi a}{p_-}
		\bigl(14+\cos\frac{4\pi a}{p_-}\bigr)}
	{2^4\cdot 7\,p_-^2\,\sin\frac{4\pi a}{p_-}\sin^3\frac{q_-\pi a}{p_-}}\;.
    \end{aligned}
  \end{multline*}
  To obtain the first equation above,
  we note that the summands for~$a$ and~$p_--a$ in Definition~\ref{C2.D1}
  are identical.
  The second equation is proved similarly.
\end{proof}

\subsection{Cheeger-Simons terms}\label{M6}
For the computation of~$\mu(M_{(p_-,q_-),(p_+,q_+)})$ using formula~\eqref{C2.1},
it remains to compute the Cheeger-Simons correction term in the adiabatic
limit.

If we regard the limit of the Levi-Civita connections on~$(M,g_\eps)$
as in~\eqref{B1.3}, we find that
	$$\lim_{\eps\to 0}p_1\bigl(TM,\nabla^{TM,\eps}\bigr)
	=p_1\bigl(TX,\nabla^{TX}\bigr)+p^*p_1\bigl(TB,\nabla^{TB}\bigr)\;.$$
The form~$p_1\bigl(TB,\nabla^{TB}\bigr)$ has already been determined
in~\eqref{M3.5}.
By the variation formula for Cheeger-Simons classes,
it is clear that
\begin{multline}\label{M4.4}
  \lim_{\eps\to 0}\int_M(p_1\wedge\hat p_1)\bigl(TM,\nabla^{TM,\eps}\bigr)\\
  =\int_M\Bigl(p_1\bigl(TX,\nabla^{TX}\bigr)
	+p^*p_1\bigl(TB,\nabla^{TB}\bigr)\Bigr)\\
	\wedge\Bigl(\hat p_1\bigl(TX,\nabla^{TX}\bigr)
	+\hat p_1\bigl(p^*TB,\nabla^{p^*TB}\bigr)\Bigr)\;,
\end{multline}
where again
\begin{align*}
  d\hat p_1\bigl(TX,\nabla^{TX}\bigr)&=p_1\bigl(TX,\nabla^{TX}\bigr)\\
	\text{and}\qquad
  d\hat p_1\bigl(p^*TB,\nabla^{p^*TB}\bigr)&=p^*p_1\bigl(TB,\nabla^{TB}\bigr)\;.
\end{align*}
Note that since~$H^4_{\mathrm{dR}}(B)\ne 0$,
we cannot expect to construct~$\hat p_1\bigl(TB,\nabla^{TB}\bigr)\in\Omega^3(B)$.

We start by computing~$p_1(TX,\nabla^{TX})$.
The connection~$\nabla^{TX}$ is defined as the compression of the Levi-Civita
connection~$\nabla^{TM}$ on~$M$ to~$TX$.
Hence, we can compute it with respect to the basis~$\bar f^1$, $\bar f^2$,
$\bar f^3$ of~$TX$ using~\eqref{M2.7}.
Its connection one-form is given by
\begin{equation*}
  \omega^{TX}=
  \begin{pmatrix}
    0&-\bar f^3&\bar f^2\\
    \bar f^3&0&-\frac{2g}f\,\bar e^1-\bar f^1\\
    -\bar f^2&\frac{2g}f\,\bar e^1+\bar f^1&0
  \end{pmatrix}\;.
\end{equation*}
The corresponding curvature is then given by
\begin{align*}
  \Omega^{TX}
  &=d\omega^{TX}+\omega^{TX}\wedge\omega^{TX}\\
  &=\begin{pmatrix}
    0&\bar f^{12}&\bar f^{13}\\
    -\bar f^{12}&0&-\frac{g'}f\,\bar e^{01}+2g\,\bar e^{23}+\bar f^{23}\\
    -\bar f^{13}&\frac{g'}f\,\bar e^{01}-2g\,\bar e^{23}-\bar f^{23}&0
  \end{pmatrix}\;.
\end{align*}

Note that since the group~$G=\SO(3)\times\SO(3)$ does not act freely
on~$\tau^{-1}\{-1,1\}$,
the basis~$(\bar f_1,\bar f_2,\bar f_3)$ does not extend over~$M_\pm$.
Hence, the form~$\omega^{TX}$ and its curvature~$\Omega^{TX}$
are not necessarily smooth at~$t=\pm 1$.
Nevertheless,
the Pontrijagin form~$p_1(TX,\nabla^{TX})$ will be smooth.
It is given by
\begin{multline}\label{M4.3}
  p_1\bigl(TX,\nabla^{TX}\bigr)
  =\frac1{8\pi^2}\,\trace\bigl((\Omega^{TX})^2\bigr)\\
  =\frac1{4\pi^2}\,\biggl(\frac{4gg'}f\,\bar e^{0123}
	+\frac{2g'}f\,\bar e^{01}\bar f^{23}
	-4g\,\bar e^{23}\bar f^{23}\biggr)\;.
\end{multline}

The forms~$p_1(TX,\nabla^{TX})$ and~$p^*p_1(TB,\nabla^{TB})$
are clearly $G$-invariant.
We will now construct $G$-invariant forms~$\hat p_1(TX,\nabla^{TX})$
and~$\hat p_1(p^*TB,\nabla^{p^*TB})$.
The complex of smooth $G$-invariant forms on~$M$
can be described as
	$$\Omega^\bullet(M)^G
	=\bigl(C^\infty([-1,1])\otimes\Lambda^\bullet\R^7\bigr)
	\cap\Omega^\bullet(M)\;,$$
where~$\R^7$ is spanned by the dual basis~$\bar e^0$, \dots, $\bar f^3$
to the basis~$\bar e_0$, \dots, $\bar f_3$ of section~\ref{M2}.
Smoothness at the singular orbits gives boundary conditions.
In particular,
functions on~$[-1,1]$ extend to smooth $G$-invariant functions
if and only if they are even at~$\pm1$,
and among others,
the monomials~$f\bar e^0$, $f\bar e^1$, $\bar e^{01}$, $\bar e^{23}$,
$\bar f^1$, $\bar f^{23}$ are smooth at~$M_-$,
and~$f\bar e^0$, $f\bar e^2$, $\bar e^{02}$, $\bar e^{31}$, $\bar f^2$,
$\bar f^{31}$ are smooth at~$M_+$.

From~\eqref{M2.7} and Cartan's formula for the exterior derivative,
we deduce that on~$\tau^{-1}(-1,0]$,
\begin{equation*}
  \begin{aligned}
    dh&=h'\,\bar e^0\;,&
    d\bar e^0&=0\;,\\
    d\bar e^1&=\frac{f'}f\,\bar e^{01}-2f\,\bar e^{23}\;,&
    d\bar f^1&=-\frac{g'}f\,\bar e^{01}+2g\,\bar e^{23}-2\,\bar f^{23}\;,\\
    d\bar e^2&=\frac2f\,\bar e^{13}\;,&
    d\bar f^2&=\biggl(\frac{2g}f\,\bar e^1+2\,\bar f^1\biggr)\,\bar f^3\;,\\
    d\bar e^3&=-\frac2f\,\bar e^{12}\;,\qquad\text{and}&\qquad
    d\bar f^3&=-\biggl(\frac{2g}f\,\bar e^1+2\,\bar f^1\biggr)\,\bar f^2\;,\\
  \end{aligned}
\end{equation*}
for functions~$h$ of~$\tau$.
Similar formulas with the indices~$1$, $2$, $3$ rotated
hold over~$\tau^{-1}[0,1)$.

From this we conclude that
\begin{equation}\label{M6.2}
  d\biggl(\frac1f\,\bar e^{123}\biggr)
  =0
	\qquad\text{and}\qquad
  d\bar f^{123}
  =\biggl(-\frac{g'}f\,\bar e^{01}+2g\,\bar e^{23}\biggr)\,\bar f^{23}\;.
\end{equation}
We also find that over~$[-1,0]$,
\begin{align}\begin{split}\label{M6.3}
  d\biggl(\frac{g'}{f}\,\bar e^{01}-2g\,\bar e^{23}-2\bar f^{23}\biggr)
  &=d\bigl(-d\bar f^1-4\bar f^{23}\bigr)=0\;,\\
  d\biggl(\biggl(\frac{g'}{f}\,\bar e^{01}-2g\,\bar e^{23}-2\bar f^{23}\biggr)
	\,\bar f^1\biggr)
  &=\biggl(\frac{g'}{f}\,\bar e^{01}-2g\,\bar e^{23}-2\bar f^{23}\biggr)\\
  &\qquad\cdot
    \biggl(-\frac{g'}{f}\,\bar e^{01}+2g\,\bar e^{23}-2\bar f^{23}\biggr)\\
  &=\frac{4gg'}f\,\bar e^{0123}\;.
\end{split}\end{align}
Similarly over~$[0,1]$, we have
\begin{equation}\label{M6.4}
  d\biggl(\biggl(\frac{g'}{f}\,\bar e^{02}-2g\,\bar e^{31}-2\bar f^{31}\biggr)
	\,\bar f^2\biggr)
  =\frac{4gg'}f\,\bar e^{0123}\;.
\end{equation}

Thus,
if we put
\begin{equation*}
  \hat p_1\bigl(TX,\nabla^{TX}\bigr)=\frac1{4\pi^2}
  \begin{cases}
    \frac{g'}{f}\,\bar e^{01}\bar f^1-2g\,\bar e^{23}\bar f^1-4\bar f^{123}
	&\text{on~$[-1,0]$, and}\\
    \frac{g'}{f}\,\bar e^{02}\bar f^2-2g\,\bar e^{31}\bar f^2-4\bar f^{123}
	&\text{on~$[0,1]$,}
  \end{cases}
\end{equation*}
then the form~$\hat p_1(TX,\nabla^{TX})$ is smooth on~$M$
because near~$\tau^{-1}(0)$, only the term~$-4\bar f^{123}$ is present.
From~\eqref{M4.3}--\eqref{M6.4},
we immediately find that
\begin{equation}\label{M6.5}
  d\hat p_1\bigl(TX,\nabla^{TX}\bigr)
  =p_1\bigl(TX,\nabla^{TX}\bigr)\;.
\end{equation}

For the next step,
we assume that~$q_+p_-\ne q_-p_+$,
because~$H^4(M;\R)=0$ in this case by Theorem~13.1 in~\cite{GWZ}.
Recall that by~\eqref{M2.3} and~\eqref{M2.8},
we have
\begin{equation*}
  f'(t)\,f''(t)=
  \begin{cases}
    \frac{16}{q_-^2}\,g(t)\,g'(t)	&\text{if~$t\in[-1,0]$, and}\\
    \frac{16}{q_+^2}\,g(t)\,g'(t)	&\text{if~$t\in[0,1]$.}
  \end{cases}
\end{equation*}

We now consider the form
\begin{multline}\label{M6.6}
  \hat p_1\bigl(p^*TB,\nabla^{p^*TB}\bigr)
  =\frac4{\pi^2}\,\frac{p_+^2-p_-^2}{q_-^2p_+^2-q_+^2p_-^2}
	\,\biggl(\frac{g'}{f}\,\bar e^{01}-2g\,\bar e^{23}-2\bar f^{23}\biggr)
	\,\bar f^1\\
  +\frac1{2\pi^2}\,\biggl(f^{\prime2}
	-16\,\frac{p_+^2-p_-^2}{q_-^2p_+^2-q_+^2p_-^2}\,g^2
	-16\,\frac{q_-^2-q_+^2}{q_-^2p_+^2-q_+^2p_-^2}
	+2f^4-4f^2\biggr)\,\frac1f\,\bar e^{123}
\end{multline}
over~$[-1,0]$ and similarly over~$[0,1]$ using~\eqref{M6.4}.
Using~\eqref{M2.3}, \eqref{M2.4}, \eqref{M2.8} and~\eqref{M2.9},
we can check that the coefficient of~$\bar e^{123}$ vanishes to first order
near~$\pm 1$,
so the form above is indeed smooth.
By~\eqref{M3.5} and~\eqref{M6.2}--\eqref{M6.4},
we conclude that
\begin{equation}\label{M6.7}
  d\hat p_1\bigl(p^*TB,\nabla^{p^*TB}\bigr)
  =\frac1{\pi^2}\,\biggl(\frac{f'f''}f+4f^2f'-4f'\biggr)\,\bar e^{0123}
  =p^*p_1\bigl(TB,\nabla^{TB}\bigr)\;.
\end{equation}

We can now compute the Cheeger-Simons correction term in the Eells-Kuiper
invariant.

\begin{Proposition}\label{M6.P1}
  The adiabatic limit of the Cheeger-Simons term is given by
  \begin{multline*}
    \frac1{2^7\cdot 7}\lim_{\eps\to0}\int_M
	(p_1\wedge\hat p_1)\bigl(TM,\nabla^{TM,\eps}\bigr)\,\\
    =\frac{(p_+^2-p_-^2)^2}
	{2^2\cdot 7\,p_-^2p_+^2(q_-^2p_+^2-q_+^2p_-^2)}
      +\frac{2^6(p_+^2-p_-^2)+3(q_-^2p_+^2-q_+^2p_-^2)}
	{2^{10}\cdot 7\,p_-^2p_+^2}
  \end{multline*}
\end{Proposition}

\begin{proof}
  The forms~$p_1\bigl(TX,\nabla^{TX}\bigr)|_{\tau^{-1}[-1,0]}$
  and~$p^*p_1\bigl(TB,\nabla^{TB}\bigr)$ do not contain the exterior
  variable~$\bar f^1$ by~\eqref{M3.5} and~\eqref{M4.3}.
  Hence only the terms in~$\hat p_1\bigl(p^*TB,\nabla^{p^*TB}\bigr)$
  containing the exterior variable~$\bar f^1$
  contribute to the integral over~$\tau^{-1}[-1,0]\subset M$.
  Using~\eqref{M4.4}, we find that
  \begin{multline*}
    \lim_{\eps\to0}\int_{\tau^{-1}[-1,0]}
	p_1\bigl(TM,\nabla^{TM,\eps}\bigr)\,\hat p_1\bigl(TM,\nabla^{TM,\eps}\bigr)\\
    \begin{aligned}
      &=\int_{\tau^{-1}[-1,0]}\Bigl(p^*p_1\bigl(TB,\nabla^{TB}\bigr)
	+p_1\bigl(TX,\nabla^{TX,0}\bigr)\Bigr)\\
      &\kern7em
	\wedge\Bigl(\hat p_1\bigl(p^*TB,\nabla^{p^*TB}\bigr)+\hat p_1\bigl(TX,\nabla^{TX,0}\bigr)\Bigr)\\
      &=\int_{\tau^{-1}[-1,0]}\frac1{4\pi^2}\,\Biggl(
	\biggl(\frac{4f'f''}f+16f^2f'-16f'\biggr)\,\bar e^{0123}\\
      &\kern9em
	+\frac{4gg'}f\,\bar e^{0123}
	+\frac{2g'}f\,\bar e^{01}\bar f^{23}
	-4g\,\bar e^{23}\bar f^{23}\Biggr)\\
      &\kern2em\cdot\frac1{4\pi^2}
	\,\Biggl(\biggl(\frac{16\,p_+^2-16\,p_-^2}
		{q_-^2p_+^2-q_+^2p_-^2}+1\biggr)
	\,\biggl(\frac{g'}{f}\,\bar e^{01}-2g\,\bar e^{23}-2\bar f^{23}\biggr)
	\,\bar f^1-2\bar f^{123}\Biggr)\;,
    \end{aligned}
  \end{multline*}
  and a similar formula gives the integral over~$\tau^{-1}[0,1]$.
  Recall that the generic fibres of~$p$ have volume~$\Vol(S^3)=2\pi^2$,
  and that the slices~$\tau^{-1}(t)\subset B$
  have volume~$f(t)\,\Vol(\R P^3/(\Z/2\Z)^2)=f(t)\,\frac{\pi^2}4$.
  Hence we have
  \begin{equation}\label{M6.8}
    \Vol\bigl(\tau^{-1}(t)\bigr)=f(t)\,\frac{\pi^4}2\;.
  \end{equation}
  Combining this with the above,
  we obtain 
  \begin{multline*}
    \lim_{\eps\to0}\int_M
	p_1\bigl(TM,\nabla^{TM,\eps}\bigr)\,h_\eps\\
    \begin{aligned}
      &=-\int_{-1}^1\,\Biggl(
	\biggl(\frac{p_+^2-p_-^2}{q_-^2p_+^2-q_+^2p_-^2}+\frac18\biggr)
	\,\bigl(4f'f''+16f^3f'-16ff'+4gg'\bigr)(t)\\
      &\kern6em
	+\biggl(\frac{p_+^2-p_-^2}{q_-^2p_+^2-q_+^2p_-^2}+\frac1{16}\biggr)
	\,\bigl(4gg'\bigr)(t)\Biggr)\,dt\\
      &=32\,\frac{(p_+^2-p_-^2)^2}{p_-^2p_+^2(q_-^2p_+^2-q_+^2p_-^2)}
	+\frac8{p_-^2}-\frac8{p_+^2}+\frac{3q_-^2}{8p_-^2}
		-\frac{3q_+^2}{8p_+^2}\;.\qquad\qedhere
    \end{aligned}
  \end{multline*}
\end{proof}

\subsection{The Leray-Serre Spectral Sequence}\label{M7}

The adiabatic limit of the $\eta$-invariant of the odd signature operator
consists of terms
that correspond to the various terms in the Leray spectral sequence.
The $E_0$-term gives the integral of the $\eta$-form of the fibre
against a characteristic form on the base.
The $E_1$-term contributes by an $\eta$-invariant of the base orbifold.
This invariant vanishes here because the base is even-dimensional.
The higher terms contribute by the signs of the corresponding eigenvalues.
There are no similar contributions for~$\eta(D)$
because the fibres have positive scalar curvature 
and hence the fibrewise operator does not admit harmonic spinors.

To see that the Leray spectral sequence does not degenerate at $E_2$,
we note that the fibrewise cohomology forms a trivial bundle over~$B$
with generators~$1$ and~$\bar f^{123}$,
so we have
\begin{equation*}
  E^{p,q}_2=E^{p,q}_3=E^{p,q}_4\cong
  \begin{cases}
    \R	&\text{if~$p\in\{0,4\}$ and~$q\in\{0,3\}$, and}\\
    0	&\text{otherwise,}
  \end{cases}
\end{equation*}
whereas~$E^{0,3}_n=E^{4,0}_n=0$ for~$n\ge 5$ if the Euler class
of~\eqref{M4.1} does not vanish.

\begin{Proposition}\label{M7.P1}
  In the adiabatic limit, we have
  \begin{equation*}
    \frac1{2^5\cdot 7}\lim_{\eps\to 0}
	\sum_{\lambda_0=\lambda_1=0}\sign\lambda_\eps
    =\frac{\sign(q_-^2p_+^2-q_+^2p_-^2)}{2^5\cdot 7}\;.
  \end{equation*}
\end{Proposition}

\begin{proof}
  From Theorem~0.3 in~\cite{Dai},
  we know that it is sufficient to study the signature of the quadratic form
  \begin{equation*}
    \<\alpha,\beta\>=(\alpha\wedge d_4\beta)[M]
  \end{equation*}
  on~$E^{0,3}_4$.
  Since~$\dim E^{0,3}_4=1$,
  we only have to compute the sign of~$(\alpha\wedge d_4\alpha)[M]$
  for one~$\alpha\in E^{0,3}_4\setminus\{0\}$.
  As a representative of~$\alpha$,
  we may choose
  \begin{equation*}
    \alpha=
    \begin{cases}
      \frac{g'}{f}\,\bar e^{01}\bar f^1-2g\,\bar e^{23}\bar f^1-2\bar f^{123}
	&\text{on~$[-1,0]$, and}\\
      \frac{g'}{f}\,\bar e^{02}\bar f^2-2g\,\bar e^{31}\bar f^2-2\bar f^{123}
	&\text{on~$[0,1]$.}
    \end{cases}
  \end{equation*}
  By~\eqref{M6.3},
  we know that
  \begin{equation*}
    d\alpha=\frac{4gg'}f\bar e^{0123}
  \end{equation*}
  is of horizontal degree~$4$ as required.
  Moreover,
  integration over the generic fibre of~$p\colon M\to B$ shows
  that~$\alpha$ represents a nontrivial class in~$E_2^{0,3}\cong E_4^{0,3}$.

  The proof is completed by the computation of the sign of
  \begin{multline}\label{M7.1}
      \int_M\alpha\,d\alpha
      =-\int_M\frac{8gg'}f\,\bar e^{0123}\bar f^{123}
      =-\int_{-1}^14\pi^4\,g(t)g'(t)\,dt\\
      =-2\pi^4\,g(t)^2\bigr|_{t=-1}^1
      =2\pi^4\,\biggl(\frac{q_-^2}{p_-^2}-\frac{q_+^2}{p_+^2}\biggr)\;.
      \quad\qedhere
  \end{multline}
\end{proof}

\subsection{The Eells-Kuiper Invariant}\label{M8}
We combine Propositions~\ref{M4.P1}--\ref{M7.P1} and prove Theorem~\ref{C2.T1}
by computing the Eells-Kuiper invariants
of the spaces~$M=M_{(p_-,q_-),(p_+,q_+)}$.

\begin{proof}[Proof of Theorem~\ref{C2.T1}]
  Using Donnelly's formula for the Eells-Kuiper invariant
  and the formulas of Bismut-Cheeger and Dai for the adiabatic limit
  of $\eta$-invariants, we find that
  \begin{multline*}
    \ek\bigl(M_{(p_-,q_-),(p_+,q_+)}\bigr)\\
    \begin{aligned}
      &=\frac{\eta(B)}{2^5\cdot 7}+\frac{\eta+h}2(D)
	-\frac1{2^7\cdot 7}\int_M(p_1\wedge\hat p_1)\bigl(TM,\nabla^{TM,0}\bigr)\\
      &=\frac1{2^{10}\cdot7}
	\biggl(\frac{q_+^2}{p_+^2}-\frac{q_-^2}{p_-^2}\biggr)
	+\frac{\sign(q_-^2p_+^2-q_+^2p_-^2)}{2^5\cdot 7}
	+D(p_-,q_-)-D(p_+,q_+)
    \end{aligned}\\
	-\frac{(p_+^2-p_-^2)^2}
		{2^2\cdot 7\,p_-^2p_+^2(q_-^2p_+^2-q_+^2p_-^2)}
	-\frac{2^6(p_+^2-p_-^2)+3(q_-^2p_+^2-q_+^2p_-^2)}
		{2^{10}\cdot 7\,p_-^2p_+^2}\;.\qedhere
  \end{multline*}
\end{proof}

\subsection{Quaternionic Line Bundles}\label{M11}
In this subsection,
we will prove Theorem~\ref{C2.T2}.
We will compute the $t$-invariant of~\cite{CG}
for sufficiently many vector bundles on~$M$
to determine Crowley's quadratic form~$q\colon H^4(M)\to\Q/\Z$.
To keep computations simple,
we only consider bundle~$p^*E\to M$
where~$E$ is a honest vector bundle over the base orbifold~$B$,
which becomes trivial after restriction to~$B_-$ and~$B_+$.
This will turn out to be sufficient if~$p_-$ and~$p_+$ are relatively prime.

To construct~$E$, we regard a map~$B\to S^4$ of degree one,
where the coordinate~$\tau$ introduced in section~\ref{M2}
is mapped to the hight function~$\R^5\supset S^4\to\R$,
and where~$B_0\cong S^3/Q$ is mapped to the equator~$S^3\subset S^4$
by a map of degree one.
In particular,
for each~$\ell\in\Z$,
there is a quaternionic bundle~$E\to B$,
pulled back from~$S^4$ by the map above,
such that
\begin{equation*}
  c_2(E)[B]=\ell\;.
\end{equation*}
We choose a connection~$\nabla^{E}$ on~$E$
that is flat near the singular strata of~$B$.

To compute class~$c_2(p^*E)$, we have to study the map
	$$p^*\colon\Z\cong H^4(B)\to H^4(M)\cong\Z/k\Z\;.$$
We consider the following commutative diagram.
$$\begin{CD}
  H^3(S^3\times S^3)@<\eta^*<<H^3(M_0)@>\delta>>H^4(M)\\
  @A(\id,0)AA@AAp_0^*A@AAp^*A\\
  H^3(S^3)@<\bar\eta^*<<H^3(B_0)@>\bar\delta>>H^4(B)&\;,
\end{CD}$$
where~$\eta\colon S^3\times S^3\to(S^3\times S^3)/H\cong M_0$
and~$\bar\eta\colon S^3\to S^3/Q\cong B_0$ are quotient maps,
and~$\delta$ and~$\bar\delta$ are the connecting homomorphisms
from the Mayer-Vietoris sequences for the decompositions
\begin{align*}
  M&=(M\setminus M_+)\cup(M\setminus M_-)&&\text{with}&
  (M\setminus M_+)\cup(M\setminus M_-)&\sim M_0\;,\\
  B&=(B\setminus B_+)\cup(B\setminus B_-)&&\text{with}&
  (B\setminus B_+)\cup(B\setminus B_-)&\sim B_0\;.
\end{align*}
From~\cite[section~13]{GWZ}, we know that~$\eta^*$ is injective with
	$$\im\eta^*=\bigl\{\,(a,b)\bigm|a+b\equiv 0\text{ mod }8\,\bigr\}
	\subset\Z^2\cong H^3(S^3\times S^3)\;.$$
Similarly,
	$$\bar\eta^*=8\punkt\colon\Z\cong H^3(B_0)\to H^3(S^3)\cong\Z\;.$$
It follows that~$\eta^*$ maps~$\im p_0^*$ to~$\<(8,0)\>$.
By~\cite{GWZ}, we also know that~$\delta$ is surjective with
	$$\eta^*\ker\delta=\<(-q_-^2,p_-^2),(-q_+^2,p_+^2)\>
	\subset\Z^2\cong H^3(S^3\times S^3)\;.$$
Similarly, $\bar\delta$ is an isomorphism.

Let us determine the subset~$\im p^*\subset H^4(M)$.
All our computations will be done in the standard coordinates
on~$H^3(S^3\times S^3)\cong\Z\times\Z$.
Then~$p^*$ maps a generator of~$H^4(B)$ to the image of~$(8,0)$,
and~$\delta(8\ell,0)=0\in H^4(M)$ if and only if
	$$(8\ell,0)=a\,(-q_-^2,p_-^2)+b\,(-q_+^2,p_+^2)\;.$$
If~$c$ denote the greatest common divisor of~$p_-$ and~$p_+$,
then~$(8\ell,0)\in\ker\delta$ if and only if we can
choose~$a=n\,\frac{p_+^2}{c^2}$ and~$b=-n\,\frac{p_-^2}{c^2}$, so
	$$\ell=n\,\frac{p_-^2q_+^2-p_+^2q_-^2}{8c^2}=\pm\frac{nk}{c^2}\;.$$
Note that~$c^2$ divides~$k$.
In particular,
the image of~$p^*$ has index~$c^2$ in~$H^4(M)\cong\Z/k\Z$.
If~$p_-$ and~$p_+$ are relatively prime,
then~$p^*$ is the isomorphism referred to in Theorem~\ref{C2.T2}.

By~\cite{CG}, see Definition~\ref{C2.D2},
\begin{multline*}
  t(p^*E)
  =\frac{\eta+h}4\,\bigl(D_M^{p^*E}\bigr)-\frac{\eta+h}2\,\bigl(D_M\bigr)\\
   -\frac1{24}\int_M\Bigl(\frac{p_1}2\bigl(TM,\nabla^{TM}\bigr)
		+c_2\bigl(p^*E,\nabla^{p^*E}\bigr)\Bigr)
	\wedge\hat c_2\bigl(p^*E,\nabla^{p^*E}\bigr)\;.
\end{multline*}
Because the fibres are of positive scalar curvature,
we can apply Corollary~\ref{A5.C1}.
Hence,
\begin{multline*}
  \lim_{\eps\to 0}\biggl(\frac{\eta+h}4(D_{M,\eps}^{p^*E})
	-\frac{\eta+h}2(D_{M,\eps})\biggr)\\
    =\frac14\int_{\Lambda B}\hat A_{\Lambda B}\bigl(TB,\nabla^{TB}\bigr)
	\,\eta_{\Lambda B}(\A)
	\,\bigl(\ch\bigl(E,\nabla^E\bigr)-2\bigr)
    =0\;.
\end{multline*}
Here, the singular strata do not contribute
because~$\ch\bigl(E,\nabla^E\bigr)-2$ vanishes near the singular strata.
Over the regular stratum,
the degree~$0$ part of the $\eta$-form is the $\eta$-invariant of the
untwisted Dirac operator on the fibre,
which vanishes because the fibre is a spin symmetric space.
Hence both~$\eta(\A)$ and~$\ch\bigl(E,\nabla^E\bigr)-2$ are of degree~$4$,
so the whole integrand vanishes for degree reasons.

In order to determine~$\hat c_2\bigl(p^*E,\nabla^{p^*E}\bigr)$,
we first check that
\begin{multline*}
  \int_B\frac{4\ell}{\pi^2}\,\frac{p_-^2p_+^2}{p_-^2q_+^2-p_+^2q_-^2}
	\,\frac{2gg'}f\,\bar e^{0123}
  =\ell\,\frac{p_-^2p_+^2}{p_-^2q_+^2-p_+^2q_-^2}
	\int_{-1}^1 2g(t)g'(t)\,dt\\
  =\ell\,\frac{p_-^2p_+^2}{p_-^2q_+^2-p_+^2q_-^2}
	\,\biggl(\frac{q_+^2}{p_+^2}-\frac{q_-^2}{p_-^2}\biggr)
  =\ell
\end{multline*}
because~$\Vol(\tau^{-1}(t))=f(t)\,\frac{\pi^2}4$.
Because we have chosen~$\nabla^E$ flat near the singular strata,
we conclude that
\begin{equation*}
  c_2\bigl(E,\nabla^E\bigr)
  =\frac{4\ell}{\pi^2}\,\frac{p_-^2p_+^2}{p_-^2q_+^2-p_+^2q_-^2}
	\,\frac{2gg'}f\,\bar e^{0123}+d\gamma
\end{equation*}
for some form~$\gamma\in\Omega^3(B)$
that is supported away from the singular set~$B_-\cup B_+$.
By~\eqref{M6.3}, we may put
\begin{equation*}
  \hat c_2\bigl(p^*E,\nabla^{p^*E}\bigr)|_{\tau^{-1}[-1,0]}
  =\frac{2\ell}{\pi^2}\,\frac{p_-^2p_+^2}{p_-^2q_+^2-p_+^2q_-^2}
	\,\biggl(\frac{g'}{f}\,\bar e^{01}
		-2g\,\bar e^{23}-2\bar f^{23}\biggr)\,\bar f^1
	+p^*\gamma\;,
\end{equation*}
and similarly on~$\tau^{-1}[0,1]$.

As in the proof of Proposition~\ref{M6.P1},
we compute the Cheeger-Simons term in the adiabatic limit~$\eps\to 0$.
We have computed the Pontrijagin forms of~$TB$ and~$TX$
in~\eqref{M3.5} and~\eqref{M4.3}.
Over~$\tau^{-1}(-1,0)$, we have
\allowdisplaybreaks
\begin{multline*}
  \Bigl(\frac{p_1}2\bigl(TX,\nabla^{TX}\bigr)
	+p^*\frac{p_1}2\bigl(TB,\nabla^{TB}\bigr)+p^*c_2(E,\nabla^E)\Bigr)
	\cdot\hat c_2\bigl(p^*E,\nabla^{p^*E}\bigr)\\
  \begin{aligned}
    &=\frac1{4\pi^2}\,\Biggl(
	\biggl(\frac{2f'f''}f+8f^2f'-8f'\biggr)\,\bar e^{0123}\\
    &\kern8em
	+\frac{2gg'}f\,\bar e^{0123}
	+\frac{g'}f\,\bar e^{01}\bar f^{23}
	-2g\,\bar e^{23}\bar f^{23}\\
    &\kern8em
	+16\ell\,\frac{p_-^2p_+^2}{p_-^2q_+^2-p_+^2q_-^2}
		\,\frac{2gg'}f\,\bar e^{0123}
	+4\pi^2\,p^*d\gamma\Biggr)\\
    &\kern2em\cdot
	\,\Biggl(\frac{2\ell}{\pi^2}\,\frac{p_-^2p_+^2}{p_-^2q_+^2-p_+^2q_-^2}
	\,\biggl(\frac{g'}{f}\,\bar e^{01}
		-2g\,\bar e^{23}-2\bar f^{23}\biggr)\,\bar f^1
	+p^*\gamma\Biggr)\\
    &=-\frac{2\ell}{\pi^4}\,\frac{p_-^2p_+^2}{p_-^2q_+^2-xp_+^2q_-^2}
	\cdot\Biggl(\biggl(\frac{f'f''}f+4f^2f'-4f'+\frac{2gg'}f\\
    &\kern8em
	+16\ell\,\frac{p_-^2p_+^2}{p_-^2q_+^2-p_+^2q_-^2}
		\,\frac{gg'}f\biggr)\,\bar e^{0123}
	+2\pi^2\,p^*d\gamma\Biggr)\,\bar f^{123}
  \end{aligned}
\end{multline*}
\interdisplaylinepenalty1000
Over~$\tau^{-1}(0,1)$,
we obtain the same right hand side.
By partial integration and~\eqref{M6.2},
we see that there is no contribution from~$p^*\gamma$ and~$p^*d\gamma$.

By~\eqref{M6.8},
we compute
\begin{align*}
    t(p^*E)
    &=-\lim_{\eps\to0}\,\frac1{24}\int_M
	\Bigl(\frac{p_1}2\bigl(TM,\nabla^{TM,\eps}\bigr)
		+p^*c_2(E,\nabla^E)\Bigr)
	\cdot\hat c_2\bigl(p^*E,\nabla^{p^*E}\bigr)\\
    &=\frac{\ell}{24}\,\frac{p_-^2p_+^2}{p_-^2q_+^2-p_+^2q_-^2}\int_{-1}^1
	\biggl(f'f''+4f'f^3-4f'f+2gg'\\
    &\kern13em
	+16\ell\,\frac{p_-^2p_+^2}{p_-^2q_+^2-p_+^2q_-^2}\,gg'\biggr)(t)\,dt\\
    &=\frac{\ell}{24}\,\frac{p_-^2p_+^2}{p_-^2q_+^2-p_+^2q_-^2}
	\,\biggl(\frac{8}{p_+^2}-\frac{8}{p_-^2}
		+\frac{q_+^2}{p_+^2}-\frac{q_-^2}{p_-^2}+8\ell\biggr)\\
    &=\frac \ell3\,\frac{p_-^2-p_+^2+\ell\,p_-^2p_+^2}
	{p_-^2q_+^2-p_+^2q_-^2}+\frac \ell{24}\;.\quad\qed
\end{align*}

\bibliographystyle{alpha}

\bigskip
\noindent Sebastian Goette\\
Mathematisches Institut\\ Universit\"at Freiburg\\ Eckerstra\ss e~1\\
79104 Freiburg\\ Germany\\
E-mail: sebastian.goette@math.uni-freiburg.de
\enddocument